\documentclass[aos]{imsart}

\RequirePackage{amsthm,amsmath,amsfonts,amssymb}
\RequirePackage[numbers]{natbib}
\RequirePackage[colorlinks,citecolor=blue,urlcolor=blue]{hyperref}
\RequirePackage{graphicx}

\startlocaldefs
\theoremstyle{plain}

\newtheorem{theorem}{Theorem}[section]
\newtheorem{lemma}[theorem]{Lemma}
\newtheorem{proposition}[theorem]{Proposition}
\newtheorem{corollary}[theorem]{Corollary}

\theoremstyle{remark}

\newtheorem{remark}{Remark}

\endlocaldefs

\def\argmin{\mathop{\rm argmin}}
\def\argmax{\mathop{\rm argmax}}
\newcommand{\bbZ}{\mathbb{Z}}
\newcommand{\rme}{\mathrm{e}}

\newcommand{\eref}[1]{Eq.~\textup{(\ref{#1})}}
\newcommand{\lref}[1]{Lemma~\ref{#1}}
\newcommand{\coref}[1]{Corollary~\ref{#1}}
\newcommand{\lrefs}[2]{Lemmas~\ref{#1} and~\ref{#2}}
\newcommand{\tref}[1]{Theorem~\ref{#1}}
\newcommand{\pref}[1]{Proposition~\ref{#1}}
\newcommand{\eqsref}[2]{Eqs.~(\ref{#1}) and (\ref{#2})}

\newcommand{\iid}{\operatorname{iid}}

\def\<{\langle}  
\def\>{\rangle}  

\def\Label#1{\label{#1}\ [\ \text{#1}\ ]\ }
\def\Label{\label}
\newcommand{\red}[1]{{\leavevmode\color{red}#1}}

\def\suppl#1{Appendix \ref{#1}}
\def\supple#1#2{Appendices \ref{#1} and \ref{#2}}
\def\main{}
\def\intro{Appendices~\ref{app:Usefullemma} and~\ref{app:Propertyz*}}

\begin{document}

\begin{frontmatter}
\title{Significance improvement by randomized test in random sampling without replacement}
\runtitle{Significance improvement by randomized test}

\begin{aug}
\author[A]{\fnms{Zihao}~\snm{Li}\ead[label=e1]{zihaoli20@fudan.edu.cn}},
\author[A]{\fnms{Huangjun}~\snm{Zhu}\ead[label=e2]{zhuhuangjun@fudan.edu.cn}\orcid{0000-0001-7257-0764}}
\and
\author[B]{\fnms{Masahito}~\snm{Hayashi}\ead[label=e3]{hayashi@sustech.edu.cn}
\orcid{0000-0003-3104-1000}}
\address[A]{Department of Physics and State Key Laboratory of Surface Physics, Fudan University\printead[presep={,\ }]{e1,e2}}

\address[B]{Shenzhen Institute for Quantum Science and Engineering, Southern University of Science and Technology\printead[presep={,\ }]{e3}}
\end{aug}

\begin{abstract}
This paper studies one-sided hypothesis testing
under random sampling without replacement.
That is, when $n+1$ binary random variables $X_1,\ldots, X_{n+1}$
are subject to a permutation invariant distribution and
$n$ binary random variables $X_1,\ldots, X_{n}$ are observed,
we have proposed randomized tests with a randomization parameter
for the upper confidence limit of the expectation of the
$(n+1)$th random variable $X_{n+1}$
under a given significance level $\delta>0$.
Our proposed randomized test significantly improves over
deterministic test unlike random sampling with replacement.
\end{abstract}



\end{frontmatter}



\section{Introduction}
This paper discusses one-sided hypothesis testing for random sampling without replacement.
Assume that $n+1$ binary random variables $X_1, \ldots, X_{n+1}$ are generated
according to random sampling without replacement.
The task of this paper is
to verify whether the expectation of the $(n+1)$th binary random variable $X_{n+1}$
is smaller than a threshold $\epsilon$
when $n$ observations $X_1, \ldots, X_{n}$ are given.
We investigate this problem under the framework of statistical hypothesis testing, i.e.,
we study the upper confidence limit for the above expectation
under a given significance level $\delta>0$.
In the case of random sampling with replacement, the optimal test has been well studied.
In the optimal test, the decision should be made based on whether $K$
is larger or smaller than a certain threshold $k(n,\delta)$ \cite{Lehmann}, where $K$ denotes
the number of observed outcome $0$ among $X_1, \ldots, X_{n}$.
Only when $K$ is equal to the threshold $k(n,\delta)$, a randomized decision is needed.
That is, the improvement by a randomized decision is incremental.

When the random variables $X_1, X_2, \ldots, X_{n+1}$ are generated in a natural setting, we can assume that they are drawn with replacement.
Suppose we have $n+m$ random variables $X_1,\ldots, X_{n+m}$ and $n$ observations, and
we are interested in the expectation of the remaining $m$ variables.
If $n$ and $m$ goes to infinity,
we can use normal approximation, which allows us to handle our problem
in a similar way as the case with normal distribution.

However, when the random variables $X_1, \ldots, X_{n+1}$ are controlled by an adversarial player
and we have only one remaining variable,
we cannot assume random sampling with replacement,
and need to consider that they are generated according to random sampling without replacement.
The hypothesis testing in this case has not been studied sufficiently.
This case requires us to discuss the problem in a different way.
Such a problem setting occurs very frequently when we consider a cryptographic setting.
In particular, our companion paper \cite{q-paper} applies the results obtained in this paper
to the verification of blind quantum computation based on the relation pointed in \cite{HM}
although quantum key distribution is related to the case when $n$ and $m$ increase simultaneously \cite{HT}.
In addition, it is expected that it appears in adversarial machine learning \cite{OS}.

In this paper, to address this problem, we introduce a randomized test that
makes a randomized decision by using a randomization parameter
$0\leq \lambda <1$
when the outcome $K$ is greater than a certain threshold $l$.
That is, our randomized test makes deterministic decision
when the outcome $K$ is not greater than the certain threshold $l$.
This randomized test with parameter $0<\lambda<1$ significantly improves over the deterministic test.
For example, when the significance level $\delta>0$ is exponentially small in $n$,
the upper confidence limit for the expectation of $X_{n+1}$
is always the trivial value $1$, i.e.,
any deterministic test cannot provide a nontrivial value.
However, the above randomized test provides a nontrivial value for the upper confidence limit
even with an exponentially small significance level $\delta$ in $n$.
We study the asymptotic behavior of the performance of
this new randomized test under two regimes depending on the behavior of the threshold $l$.
One is the constant regime where the threshold $l$ is a constant.
The other is the linear regime where the threshold $l$ increases in proportion to $n$.
Under both regimes,
we show that our randomized tests significantly improve over deterministic tests.

The rest of this paper is organized as follows.
Section \ref{S2} formulates our problem setting and introduces
our new randomized test based on the randomization parameter $0\leq\lambda<1$.
It also presents preliminary results that illustrate the advantage of our randomized test,
while their proofs will be given in Section \ref{S8-2} and Appendix \ref{S3}.
Section \ref{linearregime} presents the asymptotic characterizations
under the linear regime.
Section \ref{S7} considers the probability of detecting a given upper confidence limit
for the independent and identically distributed (iid) case
under this regime.
Section \ref{constregime} presents the asymptotic characterizations of our randomized test
under the constant regime.
To show our results, Section \ref{S8} introduces a new approach to
calculate the upper confidence limit for the expectation of $X_{n+1}$.
Section \ref{sec:ProofThms} presents brief proofs of theorems stated in Sections~\ref{linearregime} and~\ref{constregime}.
Section \ref{S12} gives the conclusion.
Several auxiliary lemmas are given in \intro.

\section{Problem setting}\Label{S2}

\subsection{Formulation}
We consider $n+1$ binary random variables $X_1, \ldots, X_{n+1}$ that take values in $\{0,1\}$
and are subject to a joint distribution $P_{X_1, \ldots, X_{n+1}}$.
By randomly drawing  $n$ samples
$Y_1, \ldots, Y_n$ without replacement from  $X_1, \ldots, X_{n+1}$,
we would like to estimate the marginal distribution of the remaining binary variable $Y_{n+1}$.
To characterize this problem,
we introduce the set $S_{n+1}$ of permutations on $[n+1]:=
\{1, \ldots, n+1\}$.
The joint distribution $P_{Y_1, \ldots, Y_{n+1}}$ of $Y_1, \ldots, Y_{n+1}$
is given as
\begin{align}
P_{Y_1, \ldots, Y_{n+1}}(y_1, \ldots, y_{n+1})=
\sum_{\sigma \in S_{n+1}}\frac{1}{(n+1) ! }P_{X_{\sigma(1)}, \ldots, X_{\sigma(n+1)}}
(y_1, \ldots, y_{n+1}).
\end{align}
Denote the set of permutation invariant distributions on
$[n+1]$ by ${\cal Q}_{n+1}$, then
the joint distribution $P_{Y_1, \ldots, Y_{n+1}}$ can be considered as an arbitrary element of
${\cal Q}_{n+1}$.

Our problem setting is formulated as follows.
Suppose $n$ observed variables $Y_1, \ldots, Y_n$ are $y_1, \ldots, y_n$,
our task is to estimate the conditional distribution
$P_{Y_{n+1}|  (Y_1, \ldots, Y_n) =(y_1, \ldots, y_n)}$.
For each observed variable $Y_i$, the observation outcome $Y_i=0$
is labeled as a ``success'' and the outcome $Y_i=1$ is labeled as a ``failure''.
Since $P_{Y_1, \ldots, Y_{n+1}}$ is permutation invariant,
the above conditional distribution $P_{Y_{n+1}|  Y_1, \ldots, Y_n}$ depends only on the
random variable $K$ that describes the number of failures among the $n$ observations $Y_1, \ldots, Y_n$.
Hence, we discuss only the conditional distribution $P_{Y_{n+1}|K}$.
Since $Y_{n+1}$ is a binary variable taking values in $\{0,1\}$,
it suffices to estimate the conditional probability
$P(Y_{n+1}=1|K \le k)$ for each realization $k$.

Now, we discuss this problem in the framework of one-sided
hypothesis testing with significance level $\delta$.
That is, we estimate the upper bound of $P(Y_{n+1}=1|K \le k)$
under the significance level $\delta$.
One simple method is
considering the maximum as
\begin{align}
\overline{\epsilon}_0(k,n,\delta):=
\max_{Q \in {\cal Q}_{n+1}}\{Q(Y_{n+1}=1|K \le k)|Q(K \le k) \ge \delta \},
\end{align}
which is called the upper confidence limit.
That is, we can guarantee that the conditional probability
$P(Y_{n+1}=1 |K \le k)$ is not greater than the upper confidence limit $\overline{\epsilon}_0(k,n,\delta)$
under the significance level $\delta$.
Here the threshold $k$ denotes the number of allowed failures. 

As shown in Section \ref{S8-2},
the upper confidence limit $\overline{\epsilon}_0(k,n,\delta)$ reads
\begin{align}
\overline{\epsilon}_0(k,n,\delta)
=
\left\{
\begin{array}{ll}
\frac{(k+1) (n+1-k) -\delta(n+1)}{\delta(n-k)(n+1)} &
\hbox{ when }
 \delta \ge \frac{k+1}{n+1}, \\
1&
\hbox{ when }
\delta < \frac{k+1}{n+1} .
\end{array}
\right.\Label{XMP9}
\end{align}
Eq. \eqref{XMP9} shows that when the significance level $\delta$ is smaller than
$\frac{1}{n+1}= \min_{k \ge 0} \frac{k+1}{n+1}$,
we cannot make any nontrivial upper confidence limit for
the conditional probability $P(Y_{n+1}=1|K \le k)$
for any integer $k=0,1,\ldots, n$.
That is, it is impossible to realize an exponentially small
significance level with respect to $n$
for any nontrivial upper confidence limit.

Fortunately, this problem can be resolved by introducing a randomized test.
To be specific,
we consider the following artificial noise before our data processing.
In stead of $Y_1, \ldots, Y_{n}$, we focus on $n$ variables
$U_1, \ldots, U_{n}$ with the following conditional distribution $P_{U_i|Y_i}$;
\begin{align}\label{eq:deflambda}
P_{U_i|Y_i}(0|0)=1, \quad P_{U_i|Y_i}(1|0)=0 ,\quad
P_{U_i|Y_i}(0|1)=\lambda, \quad P_{U_i|Y_i}(1|1)=1-\lambda ,
\end{align}
where $0\leq\lambda<1$ is the randomization  parameter.
For convenience, we define $\nu:=1-\lambda$.
The outcome $U_i=0$ is labeled as a ``success'' and the outcome $U_i=1$ is labeled as a ``failure''.
Denote the number of failures among $U_1, \ldots, U_{n}$
by the random variable $L$, which satisfies
the following conditional distribution $P_{L|K}$;
\begin{align}\label{LK}
P_{L|K}(l|k)= 
\left\{
\begin{array}{ll}
0 &
\hbox{ when } l> k, \\
{k \choose l} \lambda^{k-l} (1-\lambda)^{l}
& \hbox{ when }
l \le k,
\end{array}
\right.
\end{align}
which implies $L \le K$.
Then, instead of $\overline{\epsilon}_0(k,n,\delta)$, we consider
\begin{align}\label{eq:epslamDef}
\overline{\epsilon}_\lambda(l,n,\delta):=
\max_{Q \in {\cal Q}_{n+1}}\{Q(Y_{n+1}=1|L \le l)|Q(L \le l) \ge \delta \},
\end{align}
where the threshold $l$ denotes the number of allowed failures. 
When $K \le l$, the relation $L\le K$ guarantees that 
the probability $P(Y_{n+1}=1)$ under the observed condition is not  greater than the upper limit.
Otherwise, we make a randomized decision. 
Note that $\overline{\epsilon}_0(l,n,\delta)$ is the special case of
$\overline{\epsilon}_\lambda(l,n,\delta)$ with $\lambda=0$.
This method guarantees that the conditional probability
$P(Y_{n+1}=1|L \le l)$ is not greater than $\overline{\epsilon}_\lambda(l,n,\delta)$
under the significance level $\delta$.

The one-sided hypothesis testing introduced above plays a crucial role in the
task of verifying quantum states in the adversarial scenario (QSVadv) \cite{ZhuEVQPSshort19,ZhuEVQPSlong19}.
In QSVadv, one aims to ensure that the quantum state on one system has
sufficiently high fidelity with the target state by performing $n$ measurements on other systems.
So far, only the verification protocol with $l=0$ has been studied in literature.
The performance with general $l>0$ is still not clear.

To see how our new method with $\lambda>0$
realizes exponentially small significance level,
we set 
$l=0$.
For $z\in\{0,1,\dots,n+1\}$, let $\delta=\delta_z:=
\frac{(n+1-z)\lambda^z+ z \lambda^{z-1}}{n+1}$, then
we have
\begin{align}
\overline{\epsilon}_\lambda(0,n,\delta_z)
=\frac{z}{z+(n-z+1)\lambda}.  \Label{XMY}
\end{align}
Eq. \eqref{XMY} follows from Eq.~(57) in Ref.~\cite{ZhuEVQPSlong19},
and a simple proof is also given in Appendix \ref{S3}.
Since $\delta_z$ has a similar behavior as $\lambda^z$,
this equation shows that
the above randomized test can realize nontrivial statement
even with exponentially small significance level.
So randomized tests can significantly outperform deterministic tests.

By contrast, randomized tests are not necessary to achieve exponentially small significance level for iid variables.
To see this, we define the upper confidence limit
in the respective setting as follows;
\begin{align}\label{eq:epsiidDef}
\overline{\epsilon}_{{\rm iid}}(k,n,\delta):=&
\max_{Q \in {\cal Q}_{\iid,n+1}}\{Q(Y_{n+1}=1|K \le k)|Q(K\le k) \ge \delta \} ,
\end{align}
where
${\cal Q}_{{\rm iid},n+1} $ is the set of independent and identical distributions on $[n+1]$.
We denote the independent and identical distribution on $[n+1]$
where the probability of the event $Y_i=1$ is $\theta$,
by $P_{\theta}^{n+1}$.
Hence, ${\cal Q}_{{\rm iid},n+1} $ can be written as $\{P_{\theta}^{n+1}\}_{\theta \in [0,1]}$, which  implies that
\begin{align}
\overline{\epsilon}_{{\rm iid}}(k,n,\delta)=&
\max_{\theta \in [0,1]} \{\theta | P_{\theta}^{n+1}(K\le k) \ge \delta \} .\Label{XMR}
\end{align}
As shown in Sections \ref{linearregime} and \ref{constregime},
$\overline{\epsilon}_{{\rm iid}}(k,n,\delta)$ has nontrivial values even when the
significance level $\delta$ is exponentially small with respect to $n$.

\subsection{Frequently used notations}\Label{notation}
In preparation for further discussions, here we
introduce a few more notations, which will be frequently used throughout the paper.

Let 
$\bbZ^{\geq j}$ the set of integers larger than or equal to $j$.
For $0\leq p \leq 1$ and $z,l,j\in\bbZ^{\geq 0}$, let
\begin{align}\Label{eq:binomCFD}
	b_{z,j}(p):= {z \choose j} p^j (1-p)^{z-j},\qquad
	B_{z,l}(p):= \sum_{j=0}^l b_{z,j}(p)= \sum_{j=0}^l {z \choose j} p^j (1-p)^{z-j}.
\end{align}
Here it is understood that $b_{z,j}(p)=0$ whenever $j\geq z+1$, and $x^0=1$ even if $x=0$.
If there is no danger of confusions, we shall use
$b_{z,l}$ and $B_{z,l}$ as shorthand for $b_{z,l}(\nu)$ and $B_{z,l}(\nu)$, respectively.
Define
\begin{align}\Label{eq:Deltazkmain}
	\Delta_{z,l}:=B_{z,l}-B_{z+1,l},\quad z,l=0,1,2,\ldots
\end{align}

For $0<\delta\leq 1$, define
\begin{align}\Label{eq:definez}
	z^*(l,\delta,\lambda):=\min\{z\in\bbZ^{\geq l}|B_{z,l}\leq \delta\},\quad
	z_*(l,\delta,\lambda):=z^*(l,\delta,\lambda)-1.
\end{align}
Here $z^*(l,\delta,\lambda)$ and $z_*(l,\delta,\lambda)$ are well defined because  $B_{l,l}(\nu)=1$, $\lim_{z\to\infty} B_{z,l}(\nu)=0$,
and that $B_{z,l}(\nu)$ is nonincreasing in $z$ for $z\geq l$
according to Lemma \ref{lem:Bzkmono}. 
In this paper we shall frequently use
$z^*$ and $z_*$ as shorthand for $z^*(l,\delta,\lambda)$ and $z_*(l,\delta,\lambda)$, respectively.

The standard normal density function reads
\begin{align}\Label{eq:phix}
	\phi(x):=\frac{1}{\sqrt{2\pi}}\rme^{-x^2/2},
\end{align}
and its distribution function reads
\begin{align}\Label{eq:Phix}
	\Phi(x):=\int_{-\infty}^x \phi(t)  dt.
\end{align}
The function $\Phi^{-1}(x)$ is defined as the inverse function of $\Phi(x)$.
For $0\leq p,q\leq1$,
the relative entropy between two binary non-negative vectors $(p,1-p)$ and $(q,1-q)$ reads
\begin{align}\Label{eq:RelEntropy}
	D(p\|q):=p\ln\frac{p}{q}+(1-p)\ln\frac{1-p}{1-q}.
\end{align}

For $x\geq 0$ and $k\in\bbZ^{\geq 0}$, the Poisson cumulative distribution function reads
\begin{align}\Label{eq:PoisFunc}
	\mathrm{Pois}(k,x):=\rme^{-x} \sum_{j=0}^k \frac{x^j}{j!}.
\end{align}
For $0<\delta<1$ and $k\in\bbZ^{\geq 0}$, define $t_\mathrm{P}(k,\delta)$ as the unique solution of $t$
to the equation
\begin{align}\Label{eq:t1kdelta}
	\mathrm{Pois}(k,t) = \delta, \qquad t\geq 0.
\end{align}
Note that $t_\mathrm{P}(k,\delta)$ is well defined because $\mathrm{Pois}(k,0)=1$, $\lim_{t\to\infty}\mathrm{Pois}(k,t)=0$, and that
$\mathrm{Pois}(k,t)$ is monotonically decreasing in $t$ for $t\geq0$ according to Lemma \ref{lem:monotPois}.

For $\gamma\geq0$ and $0<x<1$,
define $t_\mathrm{D}(\gamma,x)$ as the unique solution of $t$
to the equation
\begin{align}\Label{eq:t2rnueps}
	t D( \frac{\gamma}{t}\|x)=1, \qquad t\geq  \frac{\gamma}{x}.
\end{align}
Note that $t_\mathrm{D}(\gamma,x)$ is well defined because $D(x\|x)=0$, $D(0\|x)=-\ln(1-x)$,
and that $D( \gamma/t\|x)$ is nondecreasing in $t$ for $t\geq \gamma/x$ thanks to
Lemma \ref{lem:DpqMonoton}.

For $0\leq s<1$ and $r>0$, define $\epsilon_D(s,r)$ as the unique solution of $\epsilon$
to the equation
\begin{align}\Label{CAO}
	r=D(s\|\epsilon ), \qquad 0\leq s \leq \epsilon<1 .
\end{align}
Similarly, for $0<\epsilon<1$ and $0<r\leq-\ln(1-\epsilon)$, define $s_D(\epsilon,r)$ 
as the unique solution of $s$ to \eqref{CAO}.
Both $\epsilon_D(s,r)$ and $s_D(\epsilon,r)$ are well defined according to \lref{lem:DpqMonoton}.

\subsection{Properties of the upper confidence limit}\label{sec:Property}
Here we summarize some basic properties of
the upper confidence limit.
In particular,
\lref{LLP} provides useful bounds for $\overline{\epsilon}_0(k,n,\delta)$, $\overline{\epsilon}_\lambda(l,n,\delta)$, and $\overline{\epsilon}_{{\rm iid}}(k,n,\delta)$; and \pref{prop:epsiidMonoton} clarifies the monotonicities of $\overline{\epsilon}_{\iid}(k,n,\delta)$.
The monotonicities of $\overline{\epsilon}_{\lambda}(k,n,\delta)$ will be clarified later in \pref{lem-2-2} after several preparations.
\pref{prop:epsiidMonoton} and \lref{LLP} are proved in 
\supple{PFepsiidMonoton}{PF-Th68}, respectively.

\begin{proposition}\Label{prop:epsiidMonoton}
	Suppose $0<\delta\leq1$, $n,k\in\bbZ^{\geq 0}$, and $n>k$.
	Then $\overline{\epsilon}_{\iid}(k,n,\delta)$ is
	nonincreasing in $\delta$ for $0<\delta\leq1$,
	nonincreasing in $n$ for $n\geq k+1$, and
	nondecreasing in $k$ for $0\leq k\leq n-1$.
\end{proposition}

\begin{lemma}\Label{LLP}
	Suppose  $0< \delta \leq 1/2$, $k,l,n\in \bbZ^{\geq 0}$, $n>k$, $n>l$.
	In the iid case, we have
	\begin{align}\Label{XMA6T}
		\overline{\epsilon}_{{\rm iid}}(k,n,\delta) \ge
		\overline{\epsilon}_{{\rm iid}}(k,n,1/2) > \frac{k}{n}.
	\end{align}
	As its generalization, for $0\leq \lambda<1$
	we have
	\begin{align}\Label{XMA6C}
		\overline{\epsilon}_\lambda(l,n,\delta)\ge
		\overline{\epsilon}_\lambda(l,n,1/2)
		\left\{
		\begin{array}{ll}
			=1 &\quad
			\hbox{ when \ }
			l\geq \nu n ,\\
			> \frac{l}{\nu n}&\quad
			\hbox{ when \ }
			l<\nu n.
		\end{array}
		\right.
	\end{align}
	In addition, when $\lambda=0$, we have
	\begin{align}\Label{XMA6J}
		\overline{\epsilon}_0(k,n,\delta) > \frac{k}{n\delta}
	\end{align}
	for $k/n<\delta\leq 1/2$.
\end{lemma}

\section{Asymptotic evaluations in the linear regime}\Label{linearregime}
We consider two types of asymptotic regimes.
In the first type, the number $l$ of allowed failures is proportional to
$n$, and we take the limit
$n \to \infty$.
That is, we set the number of allowed failures $l$ to be $\lceil s\nu  n\rceil $, where $0\leq s<1$.
In the iid setting, this regime leads to the central limit theorem.
We call it the linear regime, and assume it in this section.

\subsection{Preparations}

To discuss the linear regime,
we prepare the function $\psi(\delta)$ as
\begin{align}
\psi(\delta):= \frac{\phi(\Phi^{-1}(\delta)) }{\delta  \Phi^{-1}(\delta)},
\end{align}
which takes a central role for our analysis on $\overline{\epsilon}_\lambda(\lceil s \nu  n\rceil,n,\delta)$.
This quantity can be approximated as
\begin{align}
\psi(\delta)\cong -B(1-\rme^{c \Phi^{-1}(\delta)})^{-1},
\end{align}
where $B=1.135$
and $c=\frac{1.98}{\sqrt{2}}$ \cite{KL,TR}.

When $x\ge 0$,
the function $\eta(x):=\frac{x \Phi(x)}{\phi(x)}$ satisfies the condition
$\frac{d \eta(x)}{dx}> 0$.
Hence, $\psi(\delta)$ is strictly decreasing in  $\delta > 1/2$.
Since
$ \lim_{\delta\to \frac{1}{2}+0}\psi(\delta)=+\infty$
and $ \lim_{\delta\to 1-0}\psi(\delta)=+0$,
we can define $\psi^{-1}$ on $\mathbb{R}^+$ with the range $(\frac{1}{2},1) $.
Then, we define the function
\begin{align}\Label{CSLY}
\begin{split}
C(\lambda,s,\delta):=&\frac{\sqrt{(1-\lambda)s}(1-s)}
{\sqrt{\lambda}}
\Phi^{-1}(\delta)
\bigg[
\psi(\delta)
-\frac{\lambda}{(1-\lambda)(1-s)} \bigg] \\
=&
\sqrt{s}(1-s)
\Phi^{-1}(\delta)
\bigg[
\sqrt{\frac{1-\lambda}{\lambda}}
\psi(\delta)
-\sqrt{\frac{\lambda}{1-\lambda}} \frac{1}{1-s}\bigg] ,
\end{split}
\end{align}
where $\Phi(x)$ and $\phi(x)$ are defined in \eqref{eq:phix} and \eqref{eq:Phix}, respectively.
When $s \le 1$,
the value $C(\lambda,s,\delta)$ is non-negative for
$ \delta \le
\psi^{-1}(\frac{\lambda}{(1-\lambda)(1-s)})$.
Its minimization over $\lambda$ is given as
\begin{align}
\underline{C}(s,\delta):=
\inf_{\lambda \in (0,1)}C(\lambda,s,\delta)
=
\left\{
\begin{array}{ll}
2\sqrt{\frac{-s (1-s)\Phi^{-1}(\delta)\phi(\Phi^{-1}(\delta))}{\delta}}
&\hbox{ when } \delta \le \frac{1}{2} \\
-\infty &\hbox{ otherwise. }
\end{array}
\right.
\Label{AMPZ}
\end{align}
When $\delta \le \frac{1}{2}$, due to the relation of the arithmetic-geometric means,
the above minimum value is attained when
$\sqrt{\frac{1-\lambda}{\lambda}}
\psi(\delta)
=-\sqrt{\frac{\lambda}{1-\lambda}} \frac{1}{1-s} $, i.e.,
$\lambda=\frac{(1-s)\psi(\delta)}{(1-s)\psi(\delta)-1 }$
because the minimum is calculated as
\begin{align*}
& \min_{\lambda \in (0,1)}\sqrt{s}(1-s)
\Phi^{-1}(\delta)
\bigg[
\sqrt{\frac{1-\lambda}{\lambda}}
\psi(\delta)
-\sqrt{\frac{\lambda}{1-\lambda}} \frac{1}{1-s}\bigg] \\
=& -2
\sqrt{s}(1-s)
\Phi^{-1}(\delta)
\sqrt{\bigg[
-\sqrt{\frac{1-\lambda}{\lambda}}
\psi(\delta)
\cdot
\sqrt{\frac{\lambda}{1-\lambda}} \frac{1}{1-s}\bigg]} \\
=&-2\sqrt{s(1-s)}\Phi^{-1}(\delta)\sqrt{-\psi(\delta)}
=2\sqrt{\frac{-s(1-s) \Phi^{-1}(\delta)\phi(\Phi^{-1}(\delta))}{\delta}}.
\end{align*}
When $\delta >\frac{1}{2}$, the infimum is realized when $\lambda \to 1$.

\subsection{Upper confidence limit}
In the linear regime, according to \eqref{XMA6C} of Lemma \ref{LLP}, when $0<\delta\leq1/2$
and $0\leq s<1$,
we have the following relation;
\begin{align}\Label{XMA6CB}
	\overline{\epsilon}_\lambda(\lceil s\nu  n\rceil,n,\delta) >
	s.
\end{align}
Hence, our interest is the difference between
$\overline{\epsilon}_\lambda(\lceil s \nu  n\rceil,n,\delta)$ and $s$.

\begin{theorem}\Label{thm:zeta3Asympt}
We consider the linear regime.
\begin{enumerate}
	\item[(a)]
We fix $0\leq s<1$, $0<\delta<1$, and $0<\lambda<1$.
When $s \le \delta$ and $n$ goes to infinity,
    we have
    \begin{align}
&\overline{\epsilon}_0(\lceil s n\rceil,n,\delta)=
\frac{s}{\delta}
+\Big(\frac{1-s+s^2-\delta}{\delta(1-s)}
+\frac{\upsilon_n}{\delta}\Big)
\frac{1}{n}
+O\Big(\frac{1}{n^2}\Big),
\Label{eq:MMD}
    \end{align}
where $\upsilon_n:= \lceil s n\rceil- s n\le 1$.
When $s > \delta$ and $n$ is sufficiently large,
    we have
    \begin{align}
&\overline{\epsilon}_0(\lceil s n\rceil,n,\delta)=1.
\Label{eq:MMD2}
    \end{align}

When $n$ goes to infinity,
    we have
    \begin{align}
&\overline{\epsilon}_\lambda(\lceil \nu s n\rceil,n,\delta)=
s +C(\lambda,s,\delta)\sqrt{\frac{1}{n}}+O\Big(\frac{1}{n}\Big).
\Label{eq-zeta2krn}
\end{align}
The coefficient $C(\lambda,s,\delta)$ takes a non-negative value
when
$ \delta \le
\psi^{-1}(\frac{\lambda}{\nu (1-s)})$.
That is, it is a positive number in a realistic situation.
    \item[(b)]
    Suppose $\delta=\rme^{- r n}$ for a certain constant $r>0$.
    We fix $0\leq s<1$, $r>0$, and $0<\lambda<1$.
    When $n$ goes to infinity
    we have
\begin{align}
&\overline{\epsilon}_0(\lceil s n\rceil,n,\delta=\rme^{-r n})
=1, \Label{eq:MGP}\\
&\overline{\epsilon}_\lambda(\lceil \nu s n\rceil,n,\delta=\rme^{-r n})
=\left\{
\begin{array}{ll}
1 & \quad r > D(s\nu\|\nu),\\
E_{\lambda, s}(r) +o(1)& \quad 0<r\leq D(s\nu\|\nu),
\end{array}
\right.
\Label{eq:zeta2deltarnN}
\end{align}
where
\begin{align}
E_{\lambda, s}(r):=
\frac{r t_\mathrm{D}(\frac{\nu s}{r},\nu)-\nu s}
{\lambda -\nu s + \nu r t_\mathrm{D}(\frac{\nu s}{r},\nu) }.
\end{align}
\end{enumerate}
\end{theorem}
The comparison between \eqref{eq-zeta2krn} and \eqref{eq:MMD}
leads to an advantage of our randomized test.
Eq. \eqref{eq-zeta2krn} shows that
the upper confidence limit attains $s$ for $\delta >0$ and $\lambda> 0$
while \eqref{eq:MMD} shows that it is impossible for $\delta >0$ and $\lambda= 0$.
That is, to achieve the upper confidence limit $s$,
we need to choose non-zero $\lambda$.

Eqs. \eqref{eq:MMD2} and \eqref{eq:MGP}
are shown by the substitution of $sn$ into $k$ at \eqref{XMP9}
with the asymptotic expansion for $n$.
The derivation of \eref{eq:MMD} is slightly complicated.
We choose $s_n$ as $ s_n= \lceil s n\rceil/n$.
Hence, $0 \le s_n - s=\upsilon_n/n \le 1/n$.
The substitution of $sn$ into $k$ at \eqref{XMP9}
with the asymptotic expansion for $n$ implies that
    \begin{align}
&\overline{\epsilon}_0(\lceil s n\rceil,n,\delta)=
\frac{s_n}{\delta}
+\frac{1-s_n+s_n^2-\delta}{\delta(1-s_n)}
\frac{1}{n}
+O\Big(\frac{1}{n^2}\Big),
\Label{eq:MMDY}
    \end{align}
Substituting $s_n=s +\upsilon_n/n $ in \eqref{eq:MMDY}, we obtain \eqref{eq:MMD}.
Eqs.~\eqref{eq-zeta2krn} and \eqref{eq:zeta2deltarnN} will be shown in Section \ref{PF-Th61} after several preparations.

Eq. \eqref{eq-zeta2krn} shows that
the term $C(\lambda,s,\delta)n^{-1/2}$ is an extra value to guarantee
the significance level $\delta$.
Hence, a smaller coefficient $C(\lambda,s,\delta)$ is better.
That is, 
the minimizer
$\lambda_{s,\delta}:= \argmin_{\lambda}C(\lambda,s,\delta)
=\frac{(1-s)\psi(\delta)}{(1-s)\psi(\delta)-1 }$
gives
the asymptotically minimum upper confidence limit
when $s$ and $\delta\in(0, 1/2)$ are fixed.
To see this characterization, we consider the minimum upper confidence limit:
    \begin{align}
\overline{\epsilon}_{\min}(s ,n,\delta):=
\min_{\lambda \in [0,1)}
\overline{\epsilon}_\lambda(\lceil \nu s n\rceil,n,\delta).
\end{align}
Lemma \ref{LLP} implies that
    \begin{align}\Label{MML}
\overline{\epsilon}_{\min}(s ,n,\delta)\ge \frac{\lceil \nu s n\rceil}{\nu  n}
\ge s \quad \forall 0< \delta \le \frac{1}{2}.
\end{align}
\eref{eq-zeta2krn} yields the following corollary for 
the minimum upper confidence limit.

\begin{corollary}
We fix $0\leq s<1$, $0<\delta<1/2$, and $0<\lambda<1$. When $n$ goes to infinity, we have
    \begin{align}
\overline{\epsilon}_{\min}(s ,n,\delta)=
s +\underline{C}(s,\delta)\sqrt{\frac{1}{n}}+O\Big(\frac{1}{n}\Big).\Label{XMU}
\end{align}

\end{corollary}

Due to Eq. \eqref{eq:MGP},
when the significance level $\delta$ is exponentially small and $\lambda=0$,
the upper confidence level takes the trivial value $1$.
That is, to achieve an exponentially small significance level $\delta$, we need to choose non-zero $\lambda$, which is an advantage of our randomized test.
Indeed,
when the significance level $\delta$ is exponentially small value $\rme^{-nr}$
and $\lambda>0$, Eq. \eqref{eq:zeta2deltarnN} shows
the upper confidence limit.
Hence, when $s$ and $r$ are fixed, the asymptotically minimum upper confidence limit
equals the minimum
$\min_{\lambda}\min\{1,E_{\lambda, s}(r)\}$,
because the condition $E_{\lambda, s}(r)\geq 1$ is equivalent to the condition
$r \geq D(s\nu\|\nu)$.
Since we have $D(s\nu\|\nu)\to\infty$ when $\lambda\to 0^+$,
this value is a nontrivial value.

In fact,
when $r$ approaches $0^+$, $r t_\mathrm{D}(\frac{\nu s}{r},\nu)$ approaches $s+0$,
and $E_{\lambda, s}(r)$ approaches $s+0$.
Also, when $r$ approaches $+\infty$, we have
$\min\{1,E_{\lambda, s}(r)\}=1$.
This fact shows that
the quantity $\overline{\epsilon}_\lambda(\lceil \nu s n\rceil,n,\delta=\rme^{-r n})$ takes values in $(s,1]$.

In contrast, we have the following proposition under the
iid case.
Under this case with the linear regime, we set $k$ to be $\lceil s n\rceil$, where $0\leq s<1$.
According to \lref{LLP}, when $0<\delta\leq 1/2$
we have
\begin{align}\label{eq:epsiid>s}
\overline{\epsilon}_{{\rm iid}}(\lceil s n\rceil,n,\delta)>s.
\end{align}

\begin{proposition}\Label{P4-3}
We consider the linear regime.
\begin{enumerate}
	\item[(a)]  We fix $0\leq s<1$ and $0<\delta<1$.
When $n$ goes to infinity, we have
\begin{align}
\overline{\epsilon}_{{\rm iid}}(\lceil s n\rceil,n,\delta)
=s- \Phi^{-1}(\delta) \sqrt{s(1-s) } \sqrt{\frac{1}{ n}}
+O(\frac{1}{n}). \Label{ZMP}
\end{align}
    \item[(b)]  Suppose $\delta=\rme^{-r n}$ for a certain constant $r>0$.
    We fix $0\leq s<1$ and $r>0$.
    When $n$ goes to infinity, we have
\begin{align}\Label{eq:zeta2deltarnU}
&\overline{\epsilon}_{{\rm iid}}(\lceil s n\rceil,n,\delta=\rme^{-r n})
=\epsilon_D(s,r)+o(1).
\end{align}
\end{enumerate}
\end{proposition}
Items (a) and (b) immediately follow from
the central limit theorem and large deviation principle, respectively.
For readers' convenience, \suppl{PP4-3} gives their proof.
The numerical comparison among 
\eqref{eq-zeta2krn}, \eqref{XMU}, and \eqref{ZMP}
is presented in the right plot of Fig. \ref{fig-ab}.

\begin{figure}
\begin{center}
	\includegraphics[width=7cm]{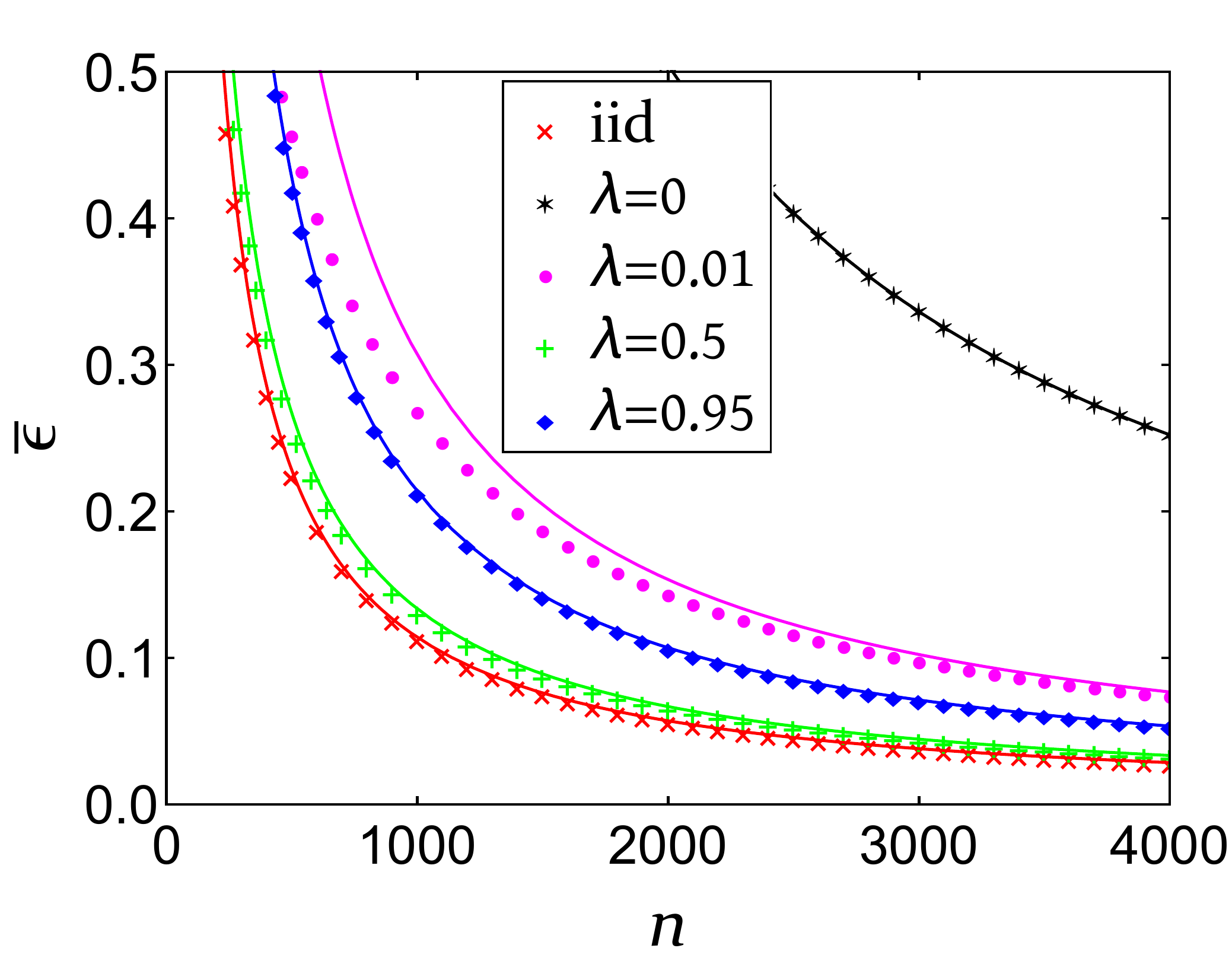}
	\includegraphics[width=7cm]{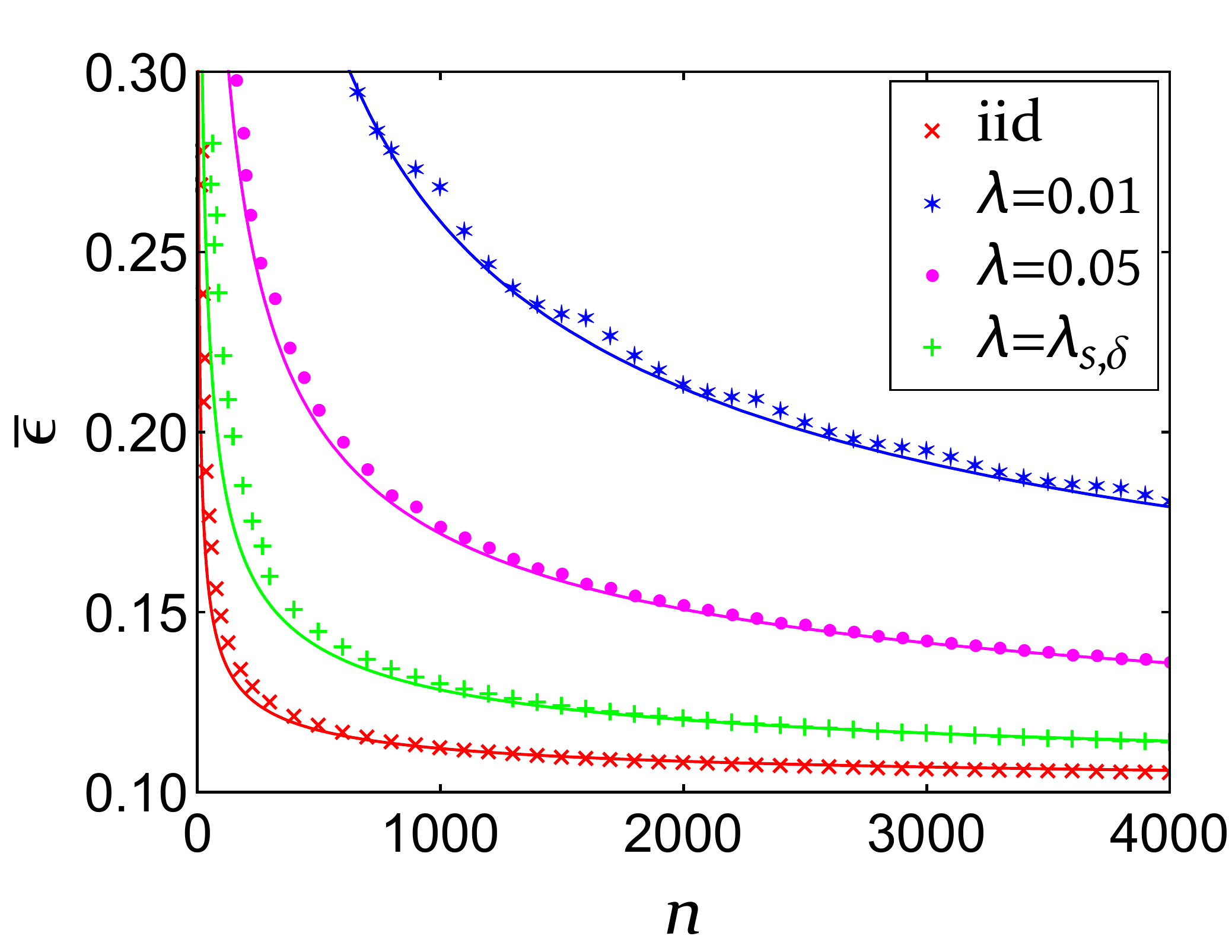}
	\caption{\Label{fig-ab}
The curves in the left graph show the numerical plots of \eqref{eq:zeta2ninftyG}, \eqref{eq:zeta2ninfty}, and \eqref{AOT} when $k_0=100$, $\delta=0.1$.
The curves in the right graph show the numerical plots of \eqref{eq-zeta2krn}, \eqref{XMU}, and \eqref{ZMP} when $s=0.1$, $\delta=0.1$.
The horizontal axis shows the number $n$, which runs from $0$ to $4000$.
The vertical axis shows the value $\epsilon$.
The discrete points represent the true value of LHS, and
the continuous plots presents the RHS by ignoring the term $O(\frac{1}{n^2})$ or
$O(\frac{1}{n})$.
These discrete points are calculated by using the formula given in \eqref{True-plot}.
In the right graph, \eqref{eq:MMD} is not plotted because
the values of \eqref{eq:MMD} are too large in comparison with other curves.
The detail of the left graph is the following.
	The red curve expresses \eqref{AOT}.
	The green curve expresses \eqref{eq:zeta2ninfty} with $\lambda=0.5$.
	The blue curve expresses \eqref{eq:zeta2ninfty} with $\lambda=0.95$.
	The pink curve expresses \eqref{eq:zeta2ninfty} with $\lambda=0.01$.
	The black curve expresses \eqref{eq:zeta2ninftyG}, which corresponds to the case  with $\lambda=0$.
The detail of the right graph is the following.
	The red curve expresses \eqref{ZMP}.
	The green curve expresses \eqref{XMU}.
	The pink curve expresses \eqref{eq-zeta2krn} with $\lambda=0.05$.
	The blue curve expresses \eqref{eq-zeta2krn} with $\lambda=0.01$.
}
\end{center}
\end{figure}


\subsection{Minimum number of required observations $n$}
Next, under the linear regime,
we consider the minimum number of observations $n$
required to guarantee
the upper confidence limit
$\epsilon$
under the significance level $\delta$.
To this end, for $0\leq\lambda<1$, $0<\epsilon,\delta<1$ and $0\leq s<1$, we define the following values;
\begin{align}
N_{\lambda,{\rm l}}(s,\epsilon, \delta):=&
\min \{n \ge \lceil \nu s n\rceil+1 \,|\,  \overline{\epsilon}_\lambda(\lceil \nu s n\rceil,n,\delta) \le \epsilon \} \Label{MLA},
\\
\underline{N}_{{\rm l}}(s,\epsilon, \delta):=&
\min_{\lambda \in [0,1)} N_{\lambda,{\rm l}}(s,\epsilon, \delta)
\Label{MLA2},
\\
N_{{\rm iid},{\rm l}}(s,\epsilon, \delta):=&
\min \{n \ge \lceil s n\rceil+1 \,|\,  \overline{\epsilon}_{{\rm iid}}(\lceil s n\rceil,n,\delta) \le \epsilon \} .
\end{align}
In the above notation, the subscript l expresses the linear regime.
To make a fair comparison of the required number $n$, we need to choose
the number of allowed failures $l$ or $k$ depending on $\lambda$ and $s$ in the above way.

For the asymptotic evaluation of the above values, we have the following corollary.
\begin{corollary}\Label{thm:EpsT}
Suppose $0<\lambda,\delta<1$ and $0\leq s<\epsilon<1$.
\begin{enumerate}
	\item[(a)]
	We fix $s, \lambda$, and $\delta$.
	When $0< \delta \leq 1/2$
	and $\epsilon$ goes to $s+0$,
    we have
\begin{align}
N_{\lambda,{\rm l}}(s,\epsilon, \delta)
&=
\left(\frac{C(\lambda,s,\delta)}{\epsilon - s}\right)^2
+O(\frac{1}{\epsilon - s}),\Label{eq:UU}
\\
\underline{N}_{{\rm l}}(s,\epsilon, \delta)&=
\left(\frac{\underline{C}(s,\delta)}{\epsilon - s}\right)^2
+O(\frac{1}{\epsilon - s}),
\Label{eq:NsmallepsTB}\\
N_{{\rm iid},{\rm l}}(s,\epsilon,\delta)&
=
\frac{[\Phi^{-1}(\delta)]^2s(1-s) }{(\epsilon - s)^2} +O(\frac{1}{\epsilon - s}) .
\Label{eq:NsmallepsT}
\end{align}
where
$C(\lambda,s,\delta)$ and $\underline{C}(s,\delta)$ are defined in Eqs.~\eqref{CSLY} and \eqref{AMPZ}, respectively.

	We fix $s,\delta$ with the condition $s< \delta \leq \frac{1}{2}$.
We have
\begin{align}\Label{eq:NsmallepsTC}
\limsup_{\epsilon\to \frac{s}{\delta}+0}
\Big|(\epsilon -\frac{s}{\delta})N_{0,{\rm l}}(s,\epsilon,\delta)
-\frac{1-s+s^2-\delta}{\delta(1-s)}
-\frac{1}{2\delta}\Big|
< \frac{1}{2\delta}
\end{align}
    \item[(b)]
    	When $\epsilon,s$, and $\lambda$ are fixed and
    	$\delta$ goes to zero, we have
\begin{align}
N_{\lambda,{\rm l}}(s,\epsilon,\delta)
&=\frac{1}{E_{\lambda, s}^{-1}(\epsilon)} \ln\delta^{-1}
+o(\ln\delta^{-1}), \Label{eq:NsmalldeltaT7} \\
N_{{\rm iid},{\rm l}}(s,\epsilon,\delta)&
=\frac{1}{D(s\|\epsilon)}\ln \delta^{-1} +o(\ln \delta^{-1}).\Label{eq:NsmalldeltaT8}
\end{align}
\end{enumerate}
\end{corollary}

\begin{proof}
The condition $0<\delta \leq 1/2$ guarantees that
$ C(\lambda,s,\delta)\geq 0$ and $\overline{\epsilon}_\lambda(\lceil \nu s n\rceil,n,\delta)>s$.
The definition \eqref{MLA} guarantees that $N_{\lambda,{\rm l}}(s,\epsilon, \delta)$ is nonincreasing in $\epsilon$.
Due to \eqref{eq-zeta2krn} and $ C(\lambda,s,\delta)\geq0$, the equation
$\epsilon= \overline{\epsilon}_\lambda(\lceil \nu s n\rceil,n,\delta)$
guarantees
$C(\lambda,s,\delta)\sqrt{\frac{1}{n}}
=\epsilon -s +O(n^{-1})$,
which implies that
$n
=\frac{C(\lambda,s,\delta)^2}{(\epsilon-s)^2} +O\big((\epsilon -s)^{-1}\big)$.
Hence, there exists an integer $n_0 > 0$ and some constant $C_0>0$ such that
\begin{align}
\Big|n
-\frac{C(\lambda,s,\delta)^2}{(\epsilon -s)^2}\Big|\le
\frac{ C_0}{\epsilon -s}
\qquad \forall n \ge n_0.
\end{align}

On the other hand, due to Lemma \ref{LLP}, there exits an $s<\epsilon_0<1 $ such that
\begin{align}
\overline{\epsilon}_\lambda(\lceil \nu s n\rceil,n,\delta) \ge \epsilon_0 \qquad \forall n \le n_0.
\end{align}
Therefore, for $\epsilon \in (s,\epsilon_0)$, we have
\begin{align}
N_{\lambda,{\rm l}}(s,\epsilon, \delta)
=\frac{C(\lambda,s,\delta)^2}{(\epsilon -s)^2} + O(\frac{1}{\epsilon - s}).
\end{align}
Hence, we obtain \eqref{eq:UU}.
In the same way, we can show
\eqref{eq:NsmallepsTB} from
\eqref{XMU} and \eqref{MML}, and show \eqref{eq:NsmallepsT} from
\eqref{ZMP} and \eqref{eq:epsiid>s}.
Also, \eqref{eq:NsmallepsTC} follows from \eqref{eq:MMD} and \eqref{XMA6J} of Lemma \ref{LLP}
in the same way
because of $1-s+s^2\ge 1/2$.
By solving the equation $E_{\lambda, s}(r)=\epsilon$ for $r$,
we obtain \eqref{eq:NsmalldeltaT7} from
\eqref{eq:zeta2deltarnN}.
Similarly, by solving $\epsilon_D(s,r)=\epsilon$ for $r$,
we obtain \eqref{eq:NsmalldeltaT8} from
\eqref{eq:zeta2deltarnU}.
\end{proof}

\subsection{Maximum number of allowed failures $l$}
Also, we may consider the following quantities;
\begin{align}\label{eq:MaxAllowFail}
l_{\lambda}(n,\epsilon, \delta):=&
\max \{l \ge 0 \,|\,  \overline{\epsilon}_\lambda(l,n,\delta) \le \epsilon \}, \\
l_{{\rm iid}}(n,\epsilon, \delta):=&
\max \{l \ge 0 \,|\,  \overline{\epsilon}_{{\rm iid}}(l,n,\delta) \le \epsilon \},
\end{align}
which denote the maximum number of allowed failures $l$
allowed to guarantee the upper confidence limit $\epsilon$ under the significance level $\delta$.

\begin{proposition}\Label{thm:EpsT}
Suppose $0<\lambda,\delta,\epsilon<1$.
\begin{enumerate}
	\item[(a)]
	When $\epsilon, \lambda,\delta$ are fixed and
	$n$ goes to infinity,
    we have
    \begin{align}
    &l_{0}(n,\epsilon, \delta)=
    \Big\lfloor\delta \epsilon n
    -\frac{1-\delta-\delta\epsilon+\delta^2\epsilon^2}{1-\delta\epsilon}
+O(\frac{1}{n}) \Big\rfloor, \Label{eq:UU-}
\end{align}
	When $\delta \le
	\psi^{-1}(\frac{\lambda}{(1-\lambda)(1-s)})$
	and $n$ goes to infinity, we have
    \begin{align}
    &l_{\lambda}(n,\epsilon, \delta)=\epsilon \nu  n
-C(\lambda,\epsilon,\delta)\nu \sqrt{n}+O(1), \Label{eq:UU+}
\end{align}
	When $\delta \leq 1/2$
	and $n$ goes to infinity, we have
    \begin{align}
&l_{{\rm iid}}(n,\epsilon, \delta)
=\epsilon n + \Phi^{-1}(\delta) \sqrt{\epsilon(1-\epsilon) } \sqrt{n}
+O(1),\Label{eq:UU+2}
\end{align}
where
$C(\lambda,s,\delta)$ is defined in \eqref{CSLY}.
	\item[(b)]
Suppose $\delta=\rme^{-rn}$ for a certain constant $r>0$.
    When $r$ is fixed and $n$ goes to infinity,
    we have
\begin{align}
l_{\lambda}(n,\epsilon, \rme^{-r n})
&=
    \nu  \bar{E}_{\lambda,r}^{-1}(\epsilon) n+o(n)
    \quad \text{for}\quad
    r \leq \frac{-\lambda\epsilon\ln\lambda}{1-\nu\epsilon}
, \Label{eq:UU+A}
\\
l_{{\rm iid}}(n,\epsilon, \rme^{-r n})
&=s_D(\epsilon,r) n +o(n) \quad \text{for}\quad  r\leq-\ln(1-\epsilon)
,\Label{eq:UU+2A}
\end{align}
where $s_D(\epsilon,r)$ is defined in \eref{CAO},
$\bar{E}_{\lambda,r}$ is defined as $\bar{E}_{\lambda,r}(s):=E_{\lambda,s}(r)$,
and $\bar{E}_{\lambda,r}^{-1}(\epsilon)$ denotes the inverse function of $\bar{E}_{\lambda,r}$.
\end{enumerate}
\end{proposition}

\begin{proof}
The condition
$\overline{\epsilon}_0(l,n,\delta) \le \epsilon$
is equivalent to the conditions
\begin{align}
l^2 -[ (1+\epsilon\delta) n  + \epsilon\delta]l
-(n+1) (1-\delta -\epsilon \delta n )\ge 0
\Label{BHA}
\end{align}
and $l \le \delta (n+1)-1$.
The condition \eqref{BHA} is equivalent to
\begin{align}
l \le    \delta \epsilon n
    -\frac{1-\delta-\delta\epsilon+\delta^2\epsilon^2}{1-\delta\epsilon}
+O(\frac{1}{n}) \quad \text{or} \quad
l \ge    n
    +\frac{1-\delta}{1-\delta\epsilon}
+O(\frac{1}{n}) \Label{MGN}.
\end{align}
However, the second case in \eqref{MGN} is not allowed due to
the condition $l \le \delta (n+1)-1$.
Hence, we have only the first case in \eqref{MGN}.
Since $l_{0}(n,\epsilon, \delta)$ is the maximum integer that satisfies
the first inequality in \eqref{MGN},
we have \eqref{eq:UU-}.

The relation \eqref{eq-zeta2krn} shows that
$\frac{l}{(1-\lambda) n}-\epsilon = - C(\lambda,\frac{l}{(1-\lambda)n},\delta) \sqrt{\frac{1}{n}}+O(\frac{1}{n})$.
Hence, we have
$C(\lambda,\frac{l}{(1-\lambda)n},\delta)
=C(\lambda,\epsilon,\delta)+ O(\sqrt{\frac{1}{n}})$.
Combining these two relations, we have \eqref{eq:UU+}.
In the same way, we can show \eqref{eq:UU+2}.

Solving the equation $\bar{E}_{\lambda, r}(s)=E_{\lambda, s}(r)=\epsilon$ for $s$,
we obtain $\frac{l}{(1-\lambda)n}=\bar{E}_{\lambda, r}^{-1}(\epsilon)$
from \eqref{eq:zeta2deltarnN}.
Note that $E_{\lambda,s}(r)$ is increasing in $s$ and 
this equation has a solution for $s\geq0$ when the constraint $E_{\lambda,0}(r)\leq\epsilon$, i.e., $r \leq \frac{-\lambda\epsilon\ln\lambda}{1-\nu\epsilon}$ holds.
Hence, we obtain \eqref{eq:UU+A}
under the condition $r \leq \frac{-\lambda\epsilon\ln\lambda}{1-\nu\epsilon}$.
In the same way, we obtain \eqref{eq:UU+2A} from \eqref{eq:zeta2deltarnU}.
\end{proof}

\section{Detection probability under the iid distribution}\Label{S7}
Until now, we have discussed how to design our test with the randomization parameter $0\leq \lambda <1$.
By \eref{eq:MaxAllowFail}, if we choose the number of allowed failures $l=l_{\lambda}(n,\epsilon, \delta)$, then
we can guarantee that the conditional probability
$P(Y_{n+1}=1|L \le l)$ is not greater than the upper confidence limit $\epsilon$
under the significance level $\delta$.
However, we have not discussed how our test efficiently detects the event guaranteeing 
the above condition.
In the following,
we consider the detection probability $P(L \le l)$ of our test
under random sampling without replacement
when the true distribution belongs to the set ${\cal Q}_{{\rm iid},n+1}$.
We assume that the number $n$ of observations is sufficiently large for the normal approximation, and that
the true distribution is given as $P_{\theta_0+\frac{t}{\sqrt{n}}}^{n+1}$.
For $\epsilon>\theta_0+\frac{t}{\sqrt{n}}$, a higher detection probability is preferred.

When $\epsilon= \theta_0+\frac{e}{\sqrt{n}}$ and $0< \lambda < 1$,
the maximum number of allowed failures is
\begin{align}
l_{\lambda}(n,\theta_0+\frac{e}{\sqrt{n}}, \delta)
=
\theta_0 (1-\lambda) n+e (1-\lambda) \sqrt{n}
-C(\lambda,\theta_0,\delta)(1-\lambda)\sqrt{n}+O(1)
\end{align}
due to \eqref{eq:UU+}.
To make a fair comparison, we focus on the threshold detecting the desired event with respect to
a random variable asymptotically subject to the standard normal distribution under the normal approximation.

First, we fix $\lambda=0$. Then,
the maximum number of allowed failures is
\begin{align}
l_{0}(n,\theta_0+\frac{e}{\sqrt{n}}, \delta)
=\delta \theta_0 n +\delta e\sqrt{n}
+O(1)
, \Label{eLLT8}
\end{align}
The random variable
$K^{(n)} := \frac{K-(\theta_0 +\frac{t}{\sqrt{n}}) n}{\sqrt{n \theta_0(1-\theta_0)  }}$
is subject to the standard normal distribution under the distribution
$P_{\theta_0+\frac{t}{\sqrt{n}}}^{n+1} $.
The threshold for the variable $K^{(n)}$ is given as
\begin{align}
\begin{split}
\frac{l_{0}(n,\theta_0+\frac{e}{\sqrt{n}}, \delta)
-(\theta_0 +\frac{t}{\sqrt{n}}) n}{\sqrt{n \theta_0(1-\theta_0)  }}
=
-\frac{(1-\delta) \sqrt{\theta_0 n}}{\sqrt{1-\theta_0}}
 +\frac{\delta e-t}{\sqrt{\theta_0(1-\theta_0)}}
+O(\frac{1}{\sqrt{n}})
. \Label{eLLT9}
\end{split}
\end{align}
Since the probability $\Phi(\eqref{eLLT9})$ goes to zero,
the detection probability converges to zero.
That is, it is almost impossible to 
guarantee that the conditional probability
$P(Y_{n+1}=1|L \le l)$ is not greater than the upper confidence limit $\epsilon
=\theta_0+\frac{e}{\sqrt{n}}$
under the significance level $\delta$
nevertheless of the values of $e$, $t$, and $\delta$ 
when $\lambda=0$ and the true distribution is $P_{\theta_0+\frac{t}{\sqrt{n}}}^{n+1} $.
This can be considered as a serious defect of 
random sampling without replacement.

However, our randomized test with non-zero $\lambda>0$
resolves this problem as follows.
In this case, the random variable
$L^{(n)}(\lambda) := \frac{L-(\theta_0 +\frac{t}{\sqrt{n}})(1-\lambda) n}
{\sqrt{n \theta_0 (1-\lambda)(1-\theta_0 (1-\lambda))  }}$
is subject to the standard normal distribution under the normal approximation
when the true distribution is $P_{\theta_0+\frac{t}{\sqrt{n}}}^{n+1} $.
The threshold for the random variable $L^{(n)}(\lambda)$ is given as
\begin{align}
\begin{split}
&\frac{l_{\lambda}(n,\theta_0+\frac{e}{\sqrt{n}}, \delta)-(\theta_0 +\frac{t}{\sqrt{n}}) (1-\lambda) n}
{\sqrt{n \theta_0 (1-\lambda)(1-\theta_0 (1-\lambda))  }} \\
=&\frac{(e-t) (1-\lambda)}{\sqrt{ \theta_0 (1-\lambda)(1-\theta_0 (1-\lambda))}  }
-\frac{C(\lambda,\theta_0,\delta)(1-\lambda)}{\sqrt{ \theta_0 (1-\lambda)(1-\theta_0 (1-\lambda))  }}+O(\frac{1}{\sqrt{n}}) \\
=&\frac{\sqrt{1-\lambda}}{\sqrt{ \theta_0 (1-\theta_0 (1-\lambda))}  }
((e-t) -C(\lambda,\theta_0,\delta))
+O(\frac{1}{\sqrt{n}}).
\end{split}\Label{AA2}
\end{align}
That is, the limiting value of the detection probability is
\begin{align}
\Phi\Big(\frac{\sqrt{1-\lambda}}{\sqrt{ \theta_0 (1-\theta_0 (1-\lambda))}  }
\big((e-t) -C(\lambda,\theta_0,\delta)\big)
\Big),
\Label{Det1}
\end{align}
which is a strictly positive value.
In other words, with the above probability, our randomized test guarantees 
that the conditional probability
$P(Y_{n+1}=1|L \le l)$ is not greater than the upper confidence limit $\epsilon
=\theta_0+\frac{e}{\sqrt{n}}$
under the significance level $\delta$
when the true distribution is $P_{\theta_0+\frac{t}{\sqrt{n}}}^{n+1} $.
This can be considered as a significant advantage of our randomized test.

In addition, to achieve the detection probability $p_0$,
the parameters $e$, $t$, $\delta$, and $\lambda$
 need to satisfy the condition
\begin{align}
\frac{\sqrt{1-\lambda}}{\sqrt{ \theta_0 (1-\theta_0 (1-\lambda))}  }
\big((e-t) -C(\lambda,\theta_0,\delta)\big)
\ge \Phi^{-1}(p_0)
+O(\frac{1}{\sqrt{n}}),
\end{align}
which is equivalent to
\begin{align}
e-t \ge
C(\lambda,\theta_0,\delta)
+\frac{\sqrt{ \theta_0 (1-\theta_0 (1-\lambda))}  }{\sqrt{1-\lambda}}
\Phi^{-1}(p_0)
+O(\frac{1}{\sqrt{n}}).\Label{XIA}
\end{align}
To achieve the detection probability $p_0$,
the difference $e-t$ needs to satisfy the above condition.

In the following, 
we maximize the detection probability \eqref{Det1} by changing $\lambda$
when the parameters $\theta_0$, $e$, $t$, and $\delta$ are fixed.
That is, we maximize
$\frac{(e-t) \sqrt{1-\lambda}}{\sqrt{ \theta_0 (1-\theta_0 (1-\lambda))}  }
-\frac{C(\lambda,\theta_0,\delta)\sqrt{1-\lambda}}{\sqrt{ \theta_0 (1-\theta_0 (1-\lambda))  }}$
with respect to $\lambda$.
This term is written as
$\kappa(\lambda):=\frac{\alpha\sqrt{\lambda}+\frac{\beta(1-\lambda)
}{\sqrt{\lambda}}+\gamma\sqrt{1-\lambda}
}{(1-\theta_0(1-\lambda))^{1/2}}$
with
$\alpha=\Phi^{-1}(\delta)$, $\beta=-\Phi^{-1}(\delta) \psi(\delta)(1-\theta_0) $,
and $\gamma=\frac{e-t}{\sqrt{\theta_0}}$.
Then, the optimal choice of $\lambda$ is characterized by the following lemma (See \suppl{P-XMLA} for a proof)
while the optimal $\lambda$ is numerically given in the left graph of Fig. \ref{fig-de}.

\begin{lemma}\Label{L-XMLA}
For $\lambda \in (0,1)$,
the equation $\kappa'(\lambda)=0$ is equivalent to the following algebraic equation with degree 3;
\begin{align}
	(1-\lambda)(\Phi^{-1}(\delta)(1-\theta_0))^2
	\Big(
	(1 +(1+\theta_0)\psi(\delta) )\lambda+(1-\theta_0)\psi(\delta)\Big)^2
	=\frac{(e-t)^2}{\theta} \lambda^3.
	\Label{XMLA}
\end{align}
For $\delta<1/2$, the equation \eqref{XMLA} has but only one solution in $(0,1)$.
	If we denote the solution by
	$\lambda_0$, then we have
	\begin{align}
		\max_{\lambda \in (0,1)}
		\kappa(\lambda)=\kappa(\lambda_0).
		\Label{XMF}
	\end{align}
\end{lemma}

\begin{figure}
\begin{center}
	\includegraphics[width=7cm]{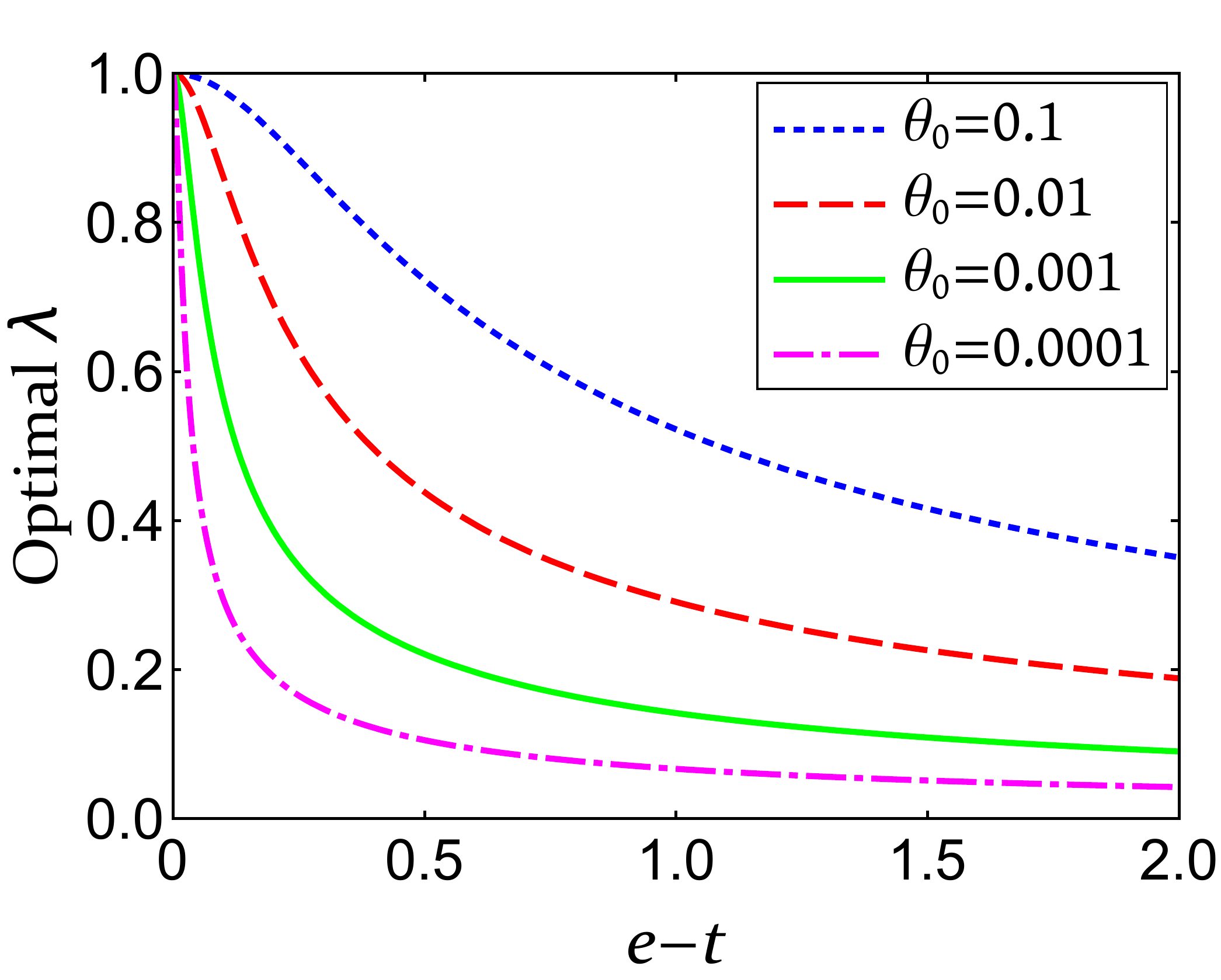}
	\includegraphics[width=7cm]{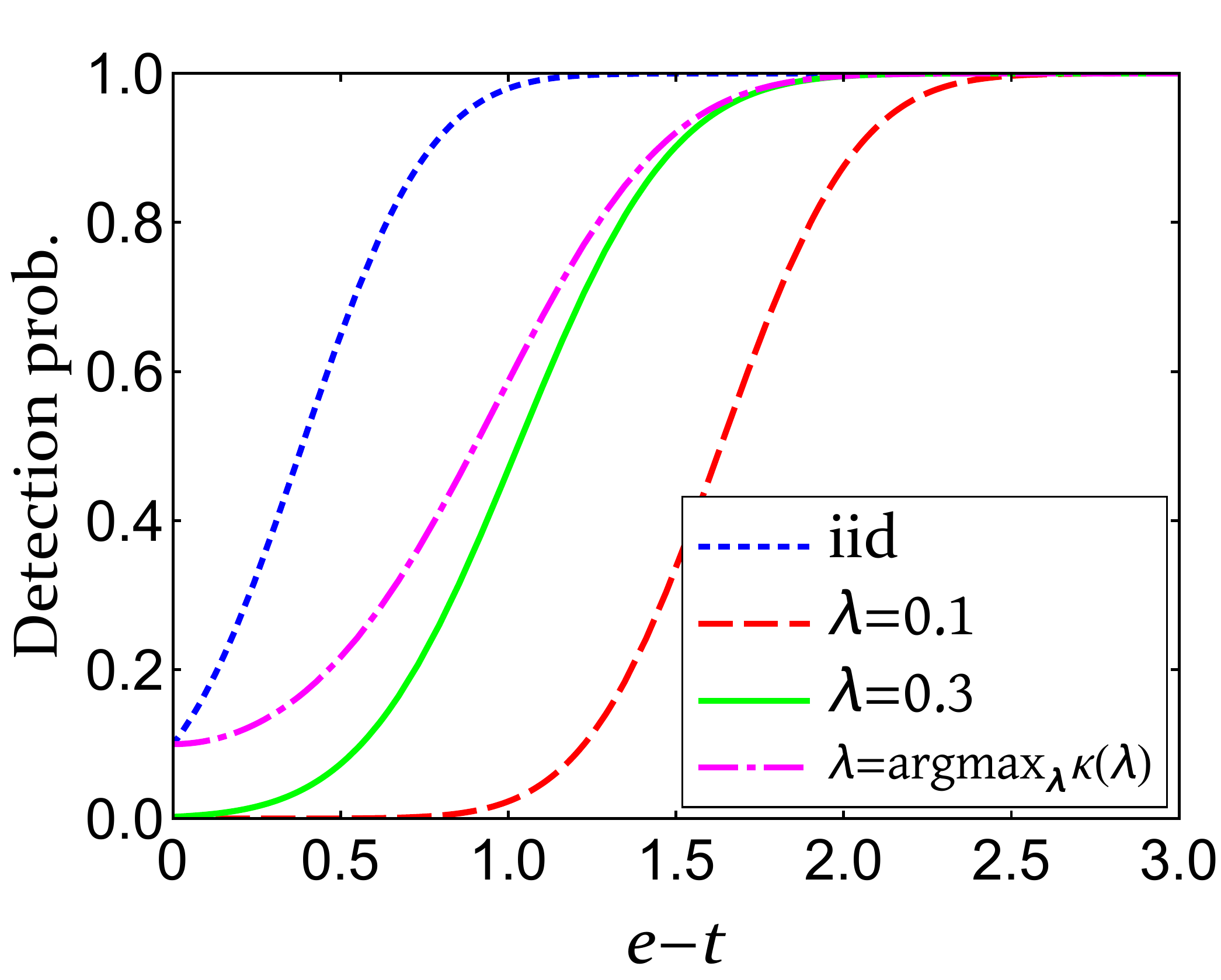}
	\caption{\Label{fig-de}
The curves in the left graph show the numerical plots of the optimal $\lambda$, i.e.,
$\argmax_{\lambda \in (0,1)}\kappa(\lambda) $ based on \eqref{XMF}
with various $\theta_0$ and $\delta=0.1$.
The curves in the right graph show the numerical plots of
the detection probabilities
\eqref{Det1}, \eqref{AA22},
and $\eqref{Det1}$
with $\lambda=\argmax_{\lambda \in (0,1)}\kappa(\lambda) $
when $\theta_0=0.1$, $\delta=0.1$.
The horizontal axis shows the value $e-t$, which runs from $0$ to $3$.
The vertical axes in the left and right graph show the values of the optimal $\lambda$ 
and the approximated detection probability, respectively.
In the right graph, 
$\Phi(\eqref{eLLT9})$ 
is not plotted because
the detection probability $\Phi(\eqref{eLLT9})$ is too small in comparison with other curves.
The detail of the left graph is the following.
	The pink curve expresses the case with $\theta_0=0.0001$.
	The green curve expresses the case with $\theta_0=0.001$.
	The red curve expresses the case with $\theta_0=0.01$.
	The blue curve expresses the case with $\theta_0=0.1$.
The detail of the right graph is the following.
	The blue curve expresses \eqref{AA22}, 
	which corresponds to the iid case.
	The pink curve expresses \eqref{Det1} with $\lambda=\argmax_{\lambda \in (0,1)}\kappa(\lambda) $.
	The green curve expresses \eqref{Det1} with $\lambda=0.3$.
	The red curve expresses \eqref{Det1} with $\lambda=0.1$.
}
\end{center}
\end{figure}

In contract, when we employ the above test under the iid assumption,
the maximum number of allowed failures is
\begin{align}
l_{{\rm iid}}(n,\theta_0+\frac{e}{\sqrt{n}}, \delta)
=
\theta_0 n+e \sqrt{n}
+\Phi^{-1}(\delta)\sqrt{\theta_0(1-\theta_0)}\sqrt{n}+O(1).
\end{align}
The threshold for the variable $K^{(n)}$ is given as
\begin{align}
\frac{l_{{\rm iid}}(n,\theta_0+\frac{e}{\sqrt{n}}, \delta)-(\theta_0 +\frac{t}{\sqrt{n}}) n}
{\sqrt{n \theta_0 (1-\theta_0)  }}
=\frac{e -t}{\sqrt{ \theta_0 (1-\theta_0 )}  }
+\Phi^{-1}(\delta)+O(\frac{1}{\sqrt{n}}).
\Label{AA}
\end{align}
That is, the limiting value of the detection probability is
\begin{align}
\Phi \Big(\frac{e -t}{\sqrt{ \theta_0 (1-\theta_0 )}  }
+\Phi^{-1}(\delta)\Big). 
\Label{AA22}
\end{align}

We compare the detection probabilities \eqref{Det1} and \eqref{AA22} as follows
while their numerical comparison is given in the right graph of Fig. \ref{fig-de}.
Their magnitude relation follows from the magnitude relation of their input values.
The first and second terms of their input values are evaluated as
\begin{align}
&\frac{e -t}{\sqrt{ \theta_0 (1-\theta_0 )}  } \Label{M12}
\ge
\frac{\sqrt{1-\lambda}}{\sqrt{ \theta_0 (1-\theta_0 (1-\lambda))}  } (e-t) ,
\\
& \Phi^{-1}(\delta)
\ge
-\frac{\sqrt{1-\lambda}\,C(\lambda,\theta_0,\delta)}{\sqrt{ \theta_0 (1-\theta_0 (1-\lambda))}  }
=
\frac{\sqrt{\theta_0}(1-\theta_0)
\Phi^{-1}(\delta)
}{\sqrt{ \theta_0 (1-\theta_0 (1-\lambda))}  }
\bigg[
\frac{\sqrt{\lambda}}{1-\theta_0}
-\frac{1-\lambda}{\sqrt{\lambda}}
\psi(\delta)
\bigg] .\Label{M13}
\end{align}
Since $\Phi$ is monotonic increasing,
their combination yields the relation \eqref{Det1} $\le$ \eqref{AA22}, i.e., 
the iid assumption realizes a larger detection probability.
This inequality can not be saturated 
because the equality in \eqref{M12} holds in the limit $\lambda\to 0$,
but the equality in \eqref{M13} holds in the limit $\lambda\to 1$.

\section{Asymptotic evaluations in the constant regime}\Label{constregime}
In the second kind of regime, we fix the number of allowed failures $l$ to be a certain constant and take the limit
$n \to \infty$.
In the iid setting, this regime leads to the law of small numbers, i.e.,
Poisson limit theorem.
We call it the constant regime.
In this section, we assume the constant regime, i.e., the number of allowed failures $l$ is fixed to be $\lceil\nu k_0\rceil$
depending on $1>\lambda\ge 0$ and an integer $k_0\ge 0$
because the average of $L$ is proportional to $\nu$.

\subsection{Upper confidence limit}
For the asymptotic evaluation of $\overline{\epsilon}_\lambda(l,n,\delta)$,
we define the coefficient
\begin{align}
G(l,\delta,\lambda):=
\frac{(B_{z_*,l}-\delta)
z^*B_{z_*,l}
+(\delta-B_{z^*,l})
z_*B_{z_*-1,l}}
{ \delta\Delta_{z_*,l}},
\end{align}
where $z^*$ and $z_*$ are shorthand for
$z^*(l,\delta,\lambda)$ and $z_*(l,\delta,\lambda)$, respectively.

Then, we have the following theorem.
\begin{theorem}\Label{thm:zeta2Asympt}
We consider the constant regime.
We fix $0<\lambda<1$ and $k_0\in\bbZ^{\geq 0}$.
Then the following statements hold:
\begin{enumerate}
    \item[(a)]  When $0<\delta<1$ is fixed and $n$ goes to infinity,
the upper confidence limit has the following behavior
\begin{align}
\overline{\epsilon}_0(k_0,n,\delta)
&=\frac{ k_0+1-\delta}{\delta }n^{-1} +o(n^{-1}) \Label{eq:zeta2ninftyG}\\
\overline{\epsilon}_\lambda(\lceil\nu k_0\rceil ,n,\delta)
&=G(\lceil\nu k_0\rceil,\delta,\lambda) n^{-1}
 +o(n^{-1}).\Label{eq:zeta2ninfty}
\end{align}
    \item[(b)]  Suppose $\delta=\rme^{- r n}$ for a certain constant $r>0$.
    When $r$ is fixed and $n$ goes to infinity,
the upper confidence limit has the following behavior
\begin{align}
\overline{\epsilon}_0(k_0,n,\delta=\rme^{-r n})
&=1 \Label{eq:MMP} \\
\overline{\epsilon}_\lambda(\lceil\nu k_0\rceil,n,\delta=\rme^{-r n})
&=
\begin{cases}
\frac{r}{r +\lambda(\ln\lambda^{-1}-r)}+o(1) &  0< r\leq\ln\lambda^{-1}, \\
1   &  r>\ln\lambda^{-1}.
\end{cases}\Label{eq:zeta2deltarn}
\end{align}
\end{enumerate}
\end{theorem}

Eqs. \eqref{eq:zeta2ninftyG} and \eqref{eq:MMP}
follow from \eref{XMP9} with the limit $n \to \infty$.
Other equations will be shown in
Section \ref{PF-Th51} after several preparations.

In contrast, we have the following proposition under the
independent and identically distributed case.

\begin{proposition}\Label{P3-2}
We consider the constant regime.
We fix $k_0\in\bbZ^{\geq 0}$.
\begin{enumerate}
    \item[(a)]  When $0<\delta<1$ is fixed and $n$ goes to infinity,
    we have
\begin{align}
\overline{\epsilon}_{{\rm iid}}(k_0,n,\delta)
=t_\mathrm{P}(k_0,\delta)\frac{1}{ n}
+o(\frac{1}{n})\Label{AOT}.
\end{align}
    \item[(b)]  Suppose $\delta=\rme^{-r n}$ for a certain constant $r>0$.
    When $r$ is fixed and $n$ goes to infinity,
    we have
\begin{align}\Label{eq:zeta2deltarnY}
&\overline{\epsilon}_{{\rm iid}}(k_0,n,\delta=\rme^{-r n})
=1-\rme^{-r}+o(1).
\end{align}
\end{enumerate}
\end{proposition}
Items (a) and (b) immediately follow from
the law of small number and large deviation principle, respectively.
For readers' convenience, we give its proof in \suppl{PP3-2}.
The numerical comparison among 
\eqref{eq:zeta2ninftyG}, \eqref{eq:zeta2ninfty}, and \eqref{AOT}
is presented as the left plot of Fig. \ref{fig-ab}.


\subsection{Minimum number of required observations $n$}
Next, 
we consider
the minimum number of observations $n$
required to guarantee
the upper confidence limit $\epsilon$
under the significance level $\delta$.
To this problem, for $0\leq\lambda<1$, $0<\epsilon,\delta<1$ and $k,l\in\bbZ^{\geq 0}$, we define the following values;
\begin{align}
N_{\lambda,{\rm c}}(l,\epsilon, \delta):=&
\min \{n \in\bbZ^{\geq l+1}|\,  \overline{\epsilon}_\lambda(l, n,\delta) \le \epsilon \} \\
N_{{\rm iid},{\rm c}}(k,\epsilon, \delta):=&
\min \{n \in\bbZ^{\geq k+1}|\,  \overline{\epsilon}_{{\rm iid}}(k,n,\delta) \le \epsilon \} .
\end{align}
In the above notation, the subscript c express the constant regime.

For the asymptotic evaluation of $N_{\lambda,{\rm c}}(l,\epsilon, \delta)$, we have the following theorem.
\begin{corollary}\Label{thm:EpsGo0}
Suppose $0<\lambda,\epsilon,\delta< 1$ and $k_0\in\bbZ^{\geq 0}$. Then the following statements hold:
\begin{enumerate}
	\item[(a)]
	When $k_0$, $\delta$, and $\lambda$ are fixed and $\epsilon$ goes to zero, we have
\begin{align}
N_{0,{\rm c}}(k_0,\epsilon,\delta)&=\frac{ k_0+1-\delta}{\delta \epsilon}
 +o(\epsilon^{-1}) \Label{eq:NsmallepsR}\\
N_{\lambda,{\rm c}}( \lceil\nu k_0\rceil,\epsilon,\delta)&=G(\lceil\nu k_0\rceil,\delta,\lambda) \epsilon^{-1}
 +o(\epsilon^{-1}) \Label{eq:Nsmalleps}\\
\Label{eq:Nsmalleps2}
N_{{\rm iid},{\rm c}}(k_0,\epsilon,\delta)&=
t_\mathrm{P}(k_0,\delta) \epsilon^{-1}
 +o(\epsilon^{-1}) .
\end{align}

    \item[(b)]
        	When $k_0$, $\epsilon$, and $\lambda$ are fixed and
    	$\delta$ goes to zero, we have
\begin{align}
N_{0,{\rm c}}(k_0,\epsilon,\delta)
&=\frac{k_0+1}{\delta\epsilon}
+o(\delta^{-1})\Label{eq:NMMK} \\
N_{\lambda,{\rm c}}(\lceil\nu k_0\rceil,\epsilon,\delta)
&=\frac{1-\nu\epsilon}{\lambda\epsilon\ln\lambda^{-1}}\ln\delta^{-1}
+o(\ln\delta^{-1})\Label{eq:NsmalldeltaT} \\
N_{{\rm iid},{\rm c}}(k_0,\epsilon,\delta)&
=\frac{\ln \delta^{-1}}{\ln (1-\epsilon)^{-1}} +o(\ln \delta^{-1}).\Label{eq:NsmalldeltaT2}
\end{align}
\end{enumerate}
\end{corollary}

In case (b), choosing non-zero $\lambda$ improves the required number
$N_{\lambda,{\rm c}}(k_0,\epsilon,\delta)$ exponentially.

\begin{proof}
Eqs. \eqref{eq:NsmallepsR}, \eqref{eq:Nsmalleps}, and \eqref{eq:Nsmalleps2} follow from
\eqref{eq:zeta2ninftyG}, \eqref{eq:zeta2ninfty}, and \eqref{AOT}, respectively.
When $\delta^{-1}$ and $n$ goes to infinity,
\eqref{eq:NMMK} follows from \eqref{XMP9}.

Eq. \eqref{eq:NsmalldeltaT} can be shown from \eqref{eq:zeta2deltarn} as follows.
We consider the following equation for $r$;
\begin{align}
\lim_{n\to \infty}\overline{\epsilon}_\lambda(\lceil\nu k_0\rceil,n,
\delta=\rme^{-r n})
=
\frac{r}{r +\lambda(\ln\lambda^{-1}-r)}=\epsilon.
\end{align}
The solution is $ r= \frac{\lambda\epsilon\ln\lambda^{-1}}{1-\nu\epsilon}$. Since $\delta=\rme^{-nr}$, we obtain \eqref{eq:NsmalldeltaT}.
In the same way,
solving the equation $1-\rme^{-r}=\epsilon$ for $r$,
we obtain \eqref{eq:NsmalldeltaT2} from
\eqref{eq:zeta2deltarnY}.
\end{proof}

As shown in \suppl{PF-MLT} after several preparations,
we have the following lemma.
\begin{lemma}\Label{MLT}
	Suppose $0<\lambda,\epsilon,\delta< 1$ and $k_0\in\bbZ^{\geq 0}$. Then
\begin{align}
\lim_{\delta,\epsilon\to 0} \frac{\epsilon N_{\lambda,{\rm c}}(\lceil\nu k_0\rceil,\epsilon,\delta)}{\ln \delta^{-1}}
=\frac{1}{\lambda \ln \lambda^{-1}},\Label{XMP98}
\end{align}
where the order of the two limits $\epsilon\to 0$ and $\delta\to 0$ does not matter.
\end{lemma}
Since the value $\frac{1}{\lambda \ln \lambda^{-1}}$ takes the minimum value
$\rme$ when $\lambda= 1/\rme$.
That is, the choice $\lambda= 1/\rme$ is optimal under the constant regime.

\section{Computation of the upper confidence limit $\overline{\epsilon}_\lambda(l,n,\delta)$}\Label{S8}
\subsection{Region approach}\Label{S8-1}
In order to determine the upper confidence limit $\overline{\epsilon}_\lambda(l,n,\delta)$,
we introduce the quantities $p_l (Q)$ and $f_l(Q)$
for $0\leq\lambda<1$, $l\in\bbZ^{\geq 0}$, $n\in\bbZ^{\geq l+1}$,
and $Q \in {\cal Q}_{n+1}$ as
\begin{align}
p_l (Q):=Q(L \le l), \quad
f_l (Q):=Q(Y_{n+1}=0, L \le l).
\end{align}
Then, instead of $\overline{\epsilon}_\lambda(l,n,\delta)$,
we discuss $\overline{\epsilon}_\lambda^{\rm c}(l,n,\delta)
:=1-\overline{\epsilon}_\lambda(l,n,\delta)$
and the function;
\begin{align}
\zeta_\lambda(l,n,\delta)=
\min_{Q \in {\cal Q}_{n+1}} \big\{f_l (Q)
\big|p_l (Q) \ge \delta \big\}.
\end{align}
To discuss $\overline{\epsilon}_\lambda^{\rm c}(l,n,\delta)$,
we introduce the two-dimensional region ${\cal R}_{l,n,\lambda}$ as
\begin{align}\Label{eq:calR}
{\cal R}_{l,n,\lambda}:=
\big\{\big(p_l (Q),f_l (Q)\big)\in \mathbb{R}_+^2\,\big|
Q \in {\cal Q}_{n+1}\big\}.
\end{align}
This region ${\cal R}_{l,n,\lambda}$ is convex because
$p_l (Q)$ and $f_l (Q)$ are both linear in $Q$,
 and the set ${\cal Q}_{n+1}$ is convex.
The geometric picture of ${\cal R}_{l,n,\lambda}$ is helpful to understanding and calculating $\overline{\epsilon}_\lambda(l,n,\delta)$.
In particular, $\overline{\epsilon}_\lambda^{\rm c}(l,n,\delta)$
can be determined for all $0<\delta\leq 1$ once the lower boundary of ${\cal R}_{l,n,\lambda}$ is known because it has the following form
\begin{align}
\begin{split}
\overline{\epsilon}_\lambda^{\rm c}(l,n,\delta)=&
\min_{Q \in {\cal Q}_{n+1}}\{Q(Y_{n+1}=0|L \le l)|Q(L \le l) \ge \delta \} \\
=&
\min_{Q \in {\cal Q}_{n+1}} \bigg\{\frac{f_l (Q)}{p_l (Q)}
\bigg|p_l (Q) \ge \delta \bigg\}.
\end{split}
\Label{AYU}
\end{align}
However, the computation of ${\cal R}_{l,n,\lambda}$ seems difficult since the region contains an infinite number of points.

To characterize it,
we define the random variable $Z:=\sum_{i=1}^{n+1}Y_i$.
Then any $Q \in {\cal Q}_{n+1}$ is completely determined by the $n+1$ probabilities $\{Q(Z=z)\}_{z=0}^n$,
and the joint distribution $P_{Y_1, \ldots, Y_{n+1}}$ of $Y_1, \ldots, Y_{n+1}$
is given as
\begin{align}
	P_{Y_1, \ldots, Y_{n+1}}(y_1, \ldots, y_{n+1})=
	{n+1 \choose \sum_{i=1}^{n+1}y_i}^{-1}
	Q\bigg(Z=\sum_{i=1}^{n+1}y_i \bigg) .
\end{align}
For integer $0\leq z\leq n+1$,
we define
\begin{align}
h_z(l,n,\lambda)&:=
Q(L \le l| Z=z)=
\begin{cases}
1 & 0\leq z \le l, \\
\frac{(n-z+1) B_{z,l} +z B_{z-1,l} }{n+1} & l+1 \leq z \leq n+1,
\end{cases}
\Label{eq:hzHomo} \\
g_z(l,n,\lambda)&:=
Q(Y_{n+1}=0, L \le l| Z=z)=
\begin{cases}
\frac{n-z+1}{n+1} &  0\leq z \le l, \\
\frac{(n-z+1)B_{z,l}}{n+1}   &  l+1 \leq z \leq n+1,
\end{cases}
\Label{eq:gzHomo}
\end{align}
where $B_{z,l}$ is shorthand for $B_{z,l}(\nu)$ as defined in \eref{eq:binomCFD}.
The second equality of \eref{eq:hzHomo} and the second equality of \eref{eq:gzHomo} are proved in 
\suppl{app:ProofhzgzHomo}.

Then, we have
\begin{align*}
&Q(L \le l)
=\sum_{z=0}^{n+1}Q(Z=z)Q(L \le l | Z=z)
=\sum_{z=0}^{n+1}Q(Z=z)h_z(l,n,\lambda) ,\\
&Q(Y_{n+1}=0,L \le l)
=\sum_{z=0}^{n+1}Q(Z=z)Q(Y_{n+1}=0,L \le l | Z=z)
=\sum_{z=0}^{n+1}Q(Z=z)g_z(l,n,\lambda) .
\end{align*}
Therefore, any element of the region ${\cal R}_{l,n,\lambda}$ can be written as
a convex combination of the $n+2$ points
$\left\{ (h_z(l,n,\lambda),g_z(l,n,\lambda))\right\}_{z=0}^{n+1}$.
That is, we obtain the following proposition.
\begin{proposition}\Label{prop:ConvexHull}
The two-dimensional region ${\cal R}_{l,n,\lambda}$ 
is the convex hull of the following set of points
\begin{align}\Label{eq:setpoints}
\left\{\big(h_z(l,n,\lambda),g_z(l,n,\lambda)\big)\right\}_{z=0}^{n+1}.
\end{align}
\end{proposition}

\begin{figure}
\begin{center}
\includegraphics[width=3.942cm]{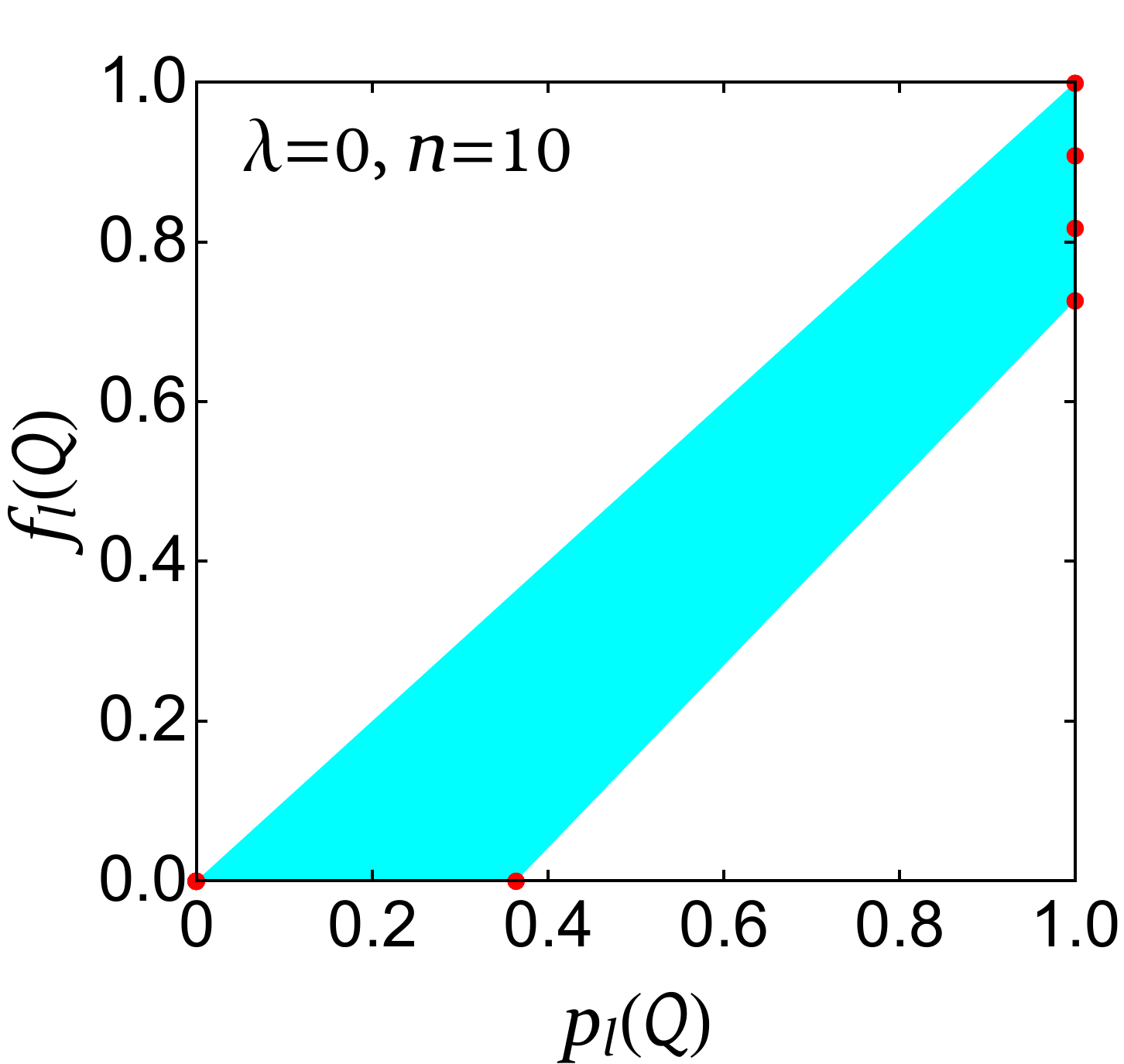}\!\!\!\!\!
\includegraphics[width=3.62cm]{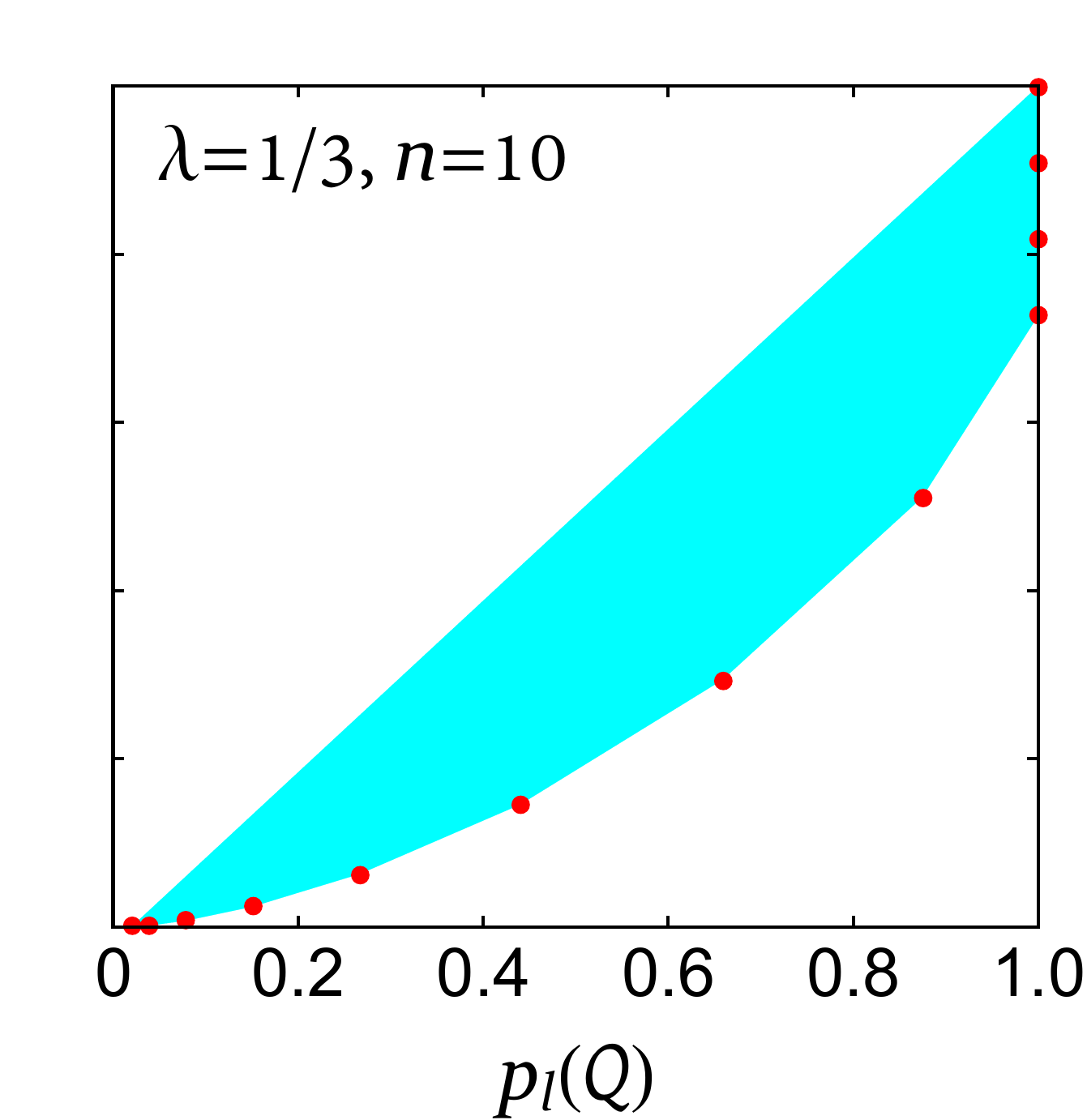}\!\!\!\!\!
\includegraphics[width=3.62cm]{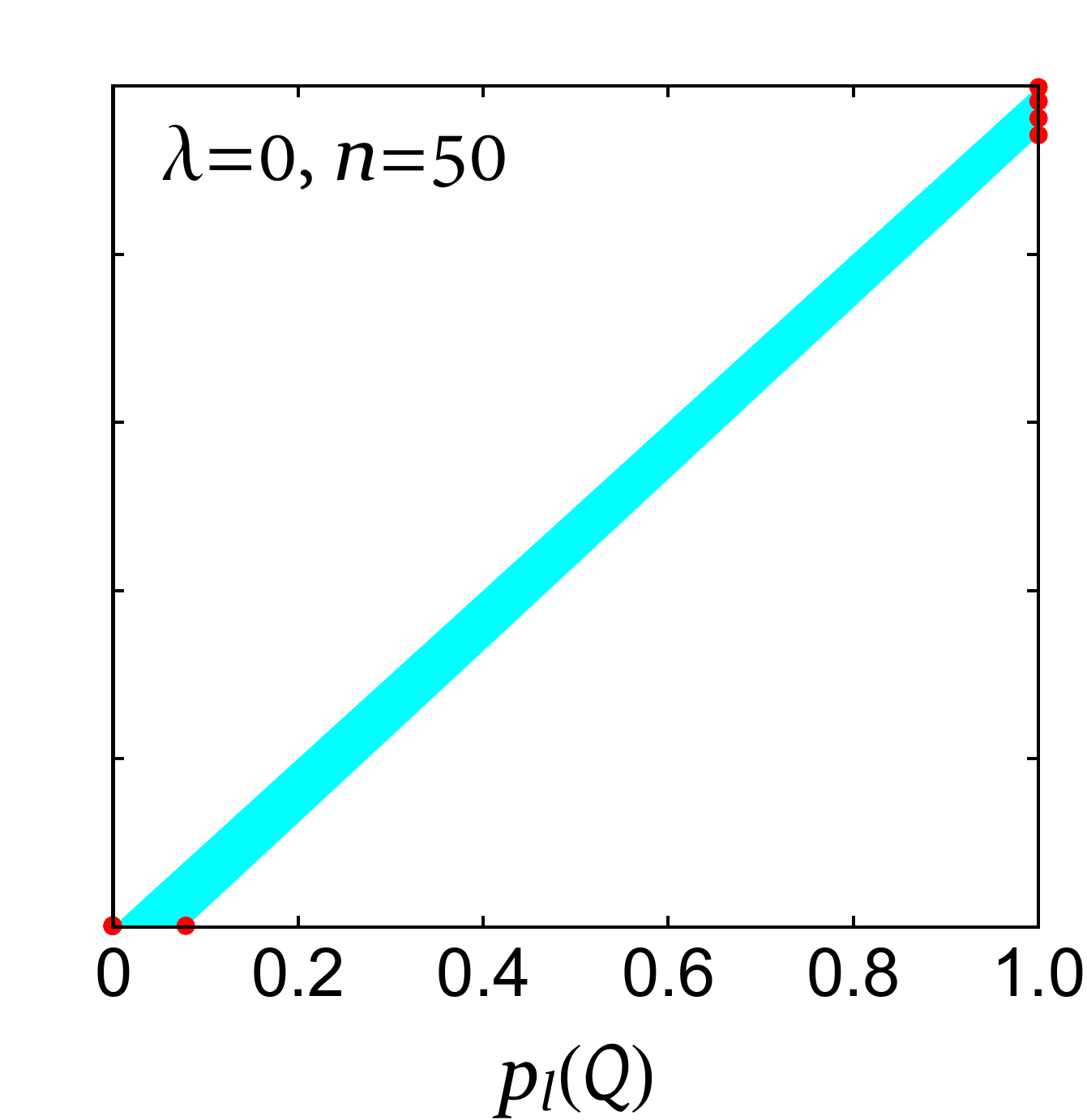}\!\!\!\!\!
\includegraphics[width=3.62cm]{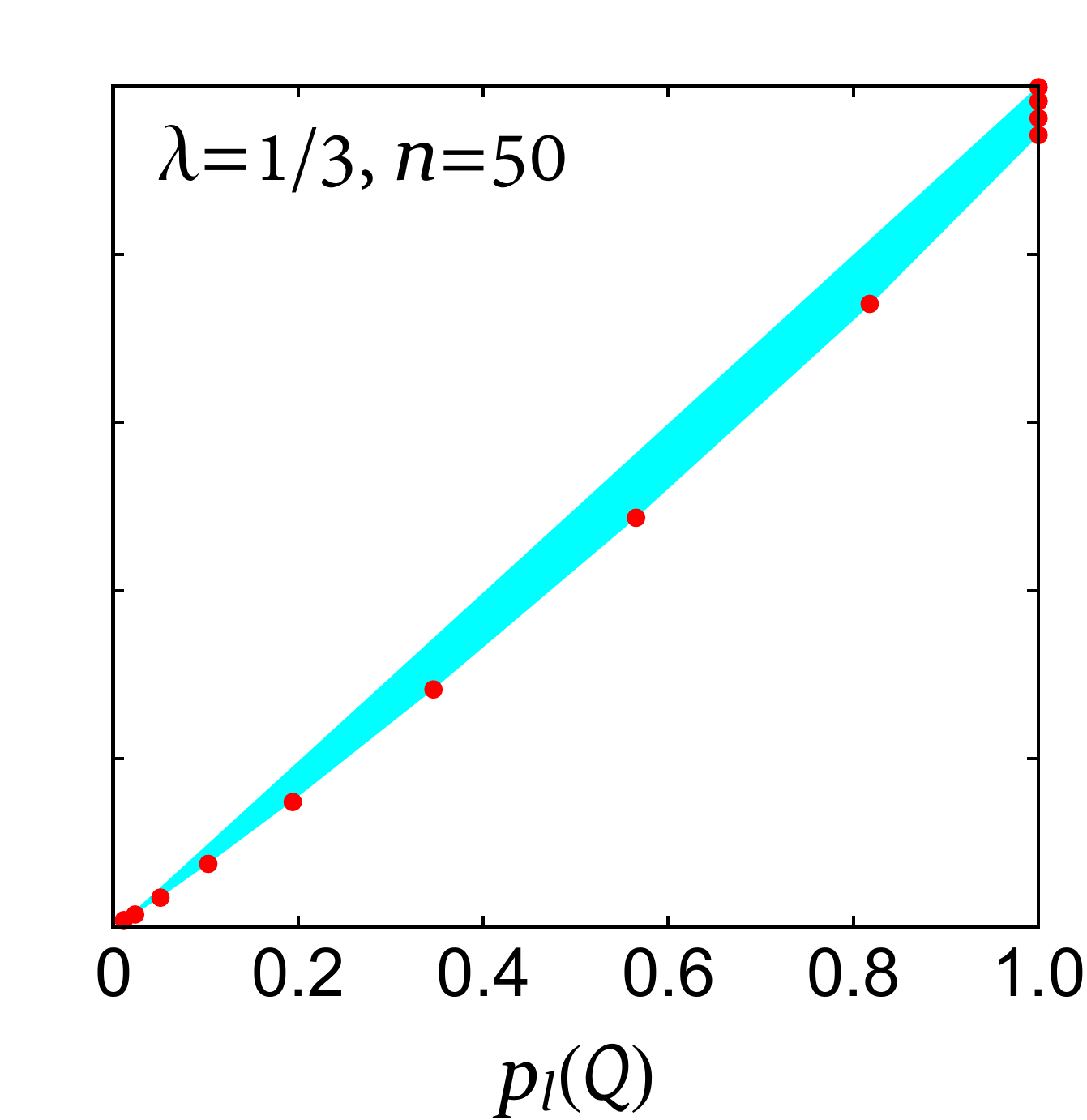}
	\caption{\label{fig:Rnklam} The region ${\cal R}_{l,n,\lambda}$ composed of points $\big(p_l (Q),f_l (Q)\big)$ as defined in \eref{eq:calR}.
The $n+2$ red points are given in \eref{eq:setpoints}.
Here we choose $l=3$ for all plots;
$n=10$        for the first  and second plots;
$n=50$        for the third  and forth  plots;
$\lambda=0$   for the first  and third plots;
$\lambda=1/3$ for the second and forth plots.
}
\end{center}
\end{figure}

According to this proposition, ${\cal R}_{l,n,\lambda}$ is a polygon, as illustrated in Fig.~\ref{fig:Rnklam}.
Now, the original complicated problem of determining the region ${\cal R}_{l,n,\lambda}$
has been transformed into the intuitive task of
studying the geometric relation between the $n+2$ points given in \eref{eq:setpoints}.


\subsection{\protect{Case with $\lambda=0$}}\Label{S8-2}
We show \eref{XMP9} by considering the case with $\lambda=0$.
In this case, \eqsref{eq:hzHomo}{eq:gzHomo} reduce to that
\begin{align}
h_z(l,n,\lambda=0)=
\begin{cases}
1 & z \leq l, \\
\frac{l+1}{n+1} & z = l+1, \\
0 & z \geq l+2,
\end{cases}\quad
g_z(l,n,\lambda=0)=
\begin{cases}
\frac{n-z+1}{n+1} &  z \leq l, \\
0   &  z \geq l+1.
\end{cases}
\end{align}
Hence the convex polygon ${\cal R}_{l,n,\lambda=0}$ has the following four extremal points:
\begin{align}
(0,0), \  \left(\frac{l+1}{n+1},0\right), \  \left(1,\frac{n-l+1}{n+1}\right), \  (1,1).
\end{align}
As a consequence, for $0<\delta\leq1$, $\zeta_0 (l,n,\delta)$ corresponds to the intersection of the
vertical line $p_k(Q)=\delta$ and the lower boundary of ${\cal R}_{l,n,\lambda=0}$, that is,
\begin{align}\Label{eq:zeta1(lambda=0)}
\begin{split}
\zeta_0 (l,n,\delta)
&= \max\left\{0, \frac{n-l+1}{n-l}\left(\delta-\frac{l+1}{n+1}\right)\right\}
= \begin{cases}
0                                                    &0<\delta\leq\frac{l+1}{n+1}, \\
\frac{n-l+1}{n-l}\left(\delta-\frac{l+1}{n+1}\right) &\frac{l+1}{n+1}<\delta\leq1.
\end{cases}
\end{split}
\end{align}
In addition, we have
\begin{align}\Label{eq:zeta2(lambda=0)}
\overline{\epsilon}_0^c(l,n,\delta)=
\frac{ \zeta_0 (l,n,\delta)}{\delta}
=
\left\{
\begin{array}{ll}
0  &0<\delta\leq\frac{l+1}{n+1}, \\
1-\frac{(l+1) (n+1-l)-\delta(n+1) }{\delta(n-l)(n+1)} \quad &
\frac{l+1}{n+1}<\delta\leq1 .
\end{array}
\right.
\end{align}

\subsection{Case with $\lambda>0$}\Label{S9}
In this subsection we discuss $\overline{\epsilon}_\lambda^{\rm c}(l,n,\delta)$ in the case $0<\lambda<1$.

As the first step, we clarify that the lower boundary of the
region ${\cal R}_{l,n,\lambda}$ defined in \eqref{eq:calR} consists of $n-l+1$ line segments,
as shown in the following lemmas (See \suppl{app:ProofLemmaZeta2}
for proofs).

\begin{lemma}\Label{lem:gzhzMono}
Suppose $0<\lambda<1$, $l,z\in\bbZ^{\geq 0}$, $n\in\bbZ^{\geq l+1}$.
Then
$h_z(l,n,\lambda)>0$;
$g_z(l,n,\lambda)\geq0$ for all $z\leq n+1$;
$g_z(l,n,\lambda)<0$ for all $z\geq n+2$.
In addition,
$h_z(l,n,\lambda)$ strictly decreases with $z$ for $l\leq z\leq n+1$;
$g_z(l,n,\lambda)$ strictly decreases with $z$ for $0\leq z\leq n+1$.
\end{lemma}

\begin{lemma}\Label{lem:SlopeMono}
Suppose $0<\lambda<1$, $l,z\in\bbZ^{\geq 0}$, $n\in\bbZ^{\geq l+1}$.
Then for $l \leq z\leq n+1$,
the point $\big(h_z(l,n,\lambda),g_z(l,n,\lambda)\big)$
is an extremal point of ${\cal R}_{l,n,\lambda}$ and
belongs to the
lower boundary of ${\cal R}_{l,n,\lambda}$.
\end{lemma}

Next, we shall determine the exact value of $\overline{\epsilon}_\lambda^{\rm c}(l,n,\delta)$.
By Lemma \ref{lem:gzhzMono}, we have $B_{n,l}=h_{n+1}(l,n,\lambda)\leq h_z(l,n,\lambda)\leq 1$ for $z\in\{0,\dots,n+1\}$.
For $B_{n,l}<\delta\leq 1$, let $\hat{z}$ be the largest integer $z$ such that
$h_z(l,n,\lambda)\geq\delta$.
For $z\in\{l,l+1,\dots,n\}$, define
\begin{align}
\kappa_z(l,n,\delta,\lambda)
             :=\;& \frac{\delta-h_{z+1}(l,n,\lambda)}{h_z(l,n,\lambda)-h_{z+1}(l,n,\lambda)}, \Label{eq:kappaz} \\
\tilde{\zeta}_\lambda(l,n,\delta,z)
             :=\;& [1-\kappa_z(l,n,\delta,\lambda)]g_{z+1}(l,n,\lambda)
                  +\kappa_z(l,n,\delta,\lambda)g_z(l,n,\lambda).  \Label{eq:tildezeta1}
\end{align}
The following theorem is an immediate consequence of \lrefs{lem:gzhzMono}{lem:SlopeMono} above.

\begin{theorem}\Label{thm:AdvFidelity}
Suppose $0<\lambda<1$, $0<\delta\leq1$,
$l\in\bbZ^{\geq 0}$, and $n\in\bbZ^{\geq l+1}$. Then we have
\begin{align}
&\overline{\epsilon}^{\rm c}_\lambda(l,n,\delta)
=\frac{\zeta_\lambda(l,n,\delta)}{\delta}
=
\begin{cases}
0                             & \delta \leq B_{n,l}, \\
\tilde{\zeta}_\lambda(l,n,\delta,\hat{z})/\delta >0 & \delta > B_{n,l}.
\end{cases}
\Label{True-plot}
\end{align}
\end{theorem}

In particular, when $\delta=h_z(l,n,\lambda)$ for $z\in\{l,l+1,\dots,n+1\}$, \tref{thm:AdvFidelity} implies that
\begin{align}\Label{eq:delta=Bzk}
\overline{\epsilon}^{\rm c}_\lambda(l,n,\delta=h_z(l,n,\lambda))  = \frac{g_z(l,n,\lambda)}{h_z(l,n,\lambda)},
\end{align}
which is nonincreasing in $z$.
This is because $h_z(l,n,\lambda)$ decreases monotonically with $z$ by \lref{lem:gzhzMono},
and $\overline{\epsilon}^{\rm c}_\lambda(l,n,\delta)$ is nondecreasing in $\delta$ for $0<\delta\leq1$ by \pref{lem-2-2} below.

As a counterpart of \pref{prop:epsiidMonoton} in Section \ref{sec:Property},
the following proposition clarifies the monotonicities of $\overline{\epsilon}_\lambda(l,n,\delta)$.
See \suppl{PF-lem-2-2} for a proof.

\begin{proposition}\Label{lem-2-2}
	Suppose $0\leq\lambda<1$, $0<\delta\leq1$, $l\in\bbZ^{\geq 0}$, and $n\in\bbZ^{\geq l+1}$.
	Then, $\overline{\epsilon}_\lambda(l,n,\delta)$
	is nonincreasing in $\delta$ for $0<\delta\leq1$,
	nonincreasing in $n$ for $n\in\bbZ^{\geq l+1}$, and
	nondecreasing in $l$ for $0\leq l\leq n-1$.
\end{proposition}

In the following, in order to show the remaining unproved results of Theorems \ref{thm:zeta2Asympt} and \ref{thm:zeta3Asympt}, we prepare two useful statements, which
provide several tight bounds for $\overline{\epsilon}^{\rm c}_\lambda(l,n,\delta)$.
\coref{cor:zeta2Bound1} follows from \eqref{eq:delta=Bzk}, \tref{thm:AdvFidelity}, Lemma \ref{lem:gzhzMono}, and the fact that  $\overline{\epsilon}^{\rm c}_\lambda(l,n,\delta) $ is nondecreasing in $\delta$ by \pref{lem-2-2}.
\lref{cor:zeta2Bound2} is proved in 
\suppl{app:ProofZeta2Asympt}.

\begin{corollary}\Label{cor:zeta2Bound1}
Suppose $0<\lambda<1$, $0<\delta<1$, $l\in\bbZ^{\geq 0}$, and
$n\in\bbZ^{\geq l+1}$, then we have
\begin{align}
\frac{g_{z^*+1}(l,n,\lambda)}{h_{z^*+1}(l,n,\lambda)}
\leq \overline{\epsilon}^{\rm c}_\lambda(l,n,\delta)
\leq \max \left\{0, \frac{g_{z_*}(l,n,\lambda)}{h_{z_*}(l,n,\lambda)}\right\},
\end{align}
where $z^*$ and $z_*$ are shorthand for $z^*(l,\delta,\lambda)$ and $z_*(l,\delta,\lambda)$ defined in \eref{eq:definez}, respectively.
\end{corollary}

\begin{lemma}\Label{cor:zeta2Bound2}
Suppose $0<\lambda<1$, $0<\delta<1$, $l\in\bbZ^{\geq 0}$, and
$n\in\bbZ^{\geq l+1}$. Then
\begin{align}
\overline{\epsilon}^{\rm c}_\lambda(l,n,\delta)
&\leq \max \left\{0, \frac{\lambda(n-z_*+1)}{\lambda(n-z_*+1)+z_*-l }
\right\}, \Label{eq:zeta2UB}\\
\overline{\epsilon}^{\rm c}_\lambda(l,n,\delta)
&\geq \frac{\lambda(n-z^*)}{\lambda(n-z^*)+z^*+1 }. \Label{eq:zeta2LB1}
\end{align}
If $\delta\leq1/2$ in addition, then
\begin{align}\label{eq:zeta2LB}
\overline{\epsilon}^{\rm c}_\lambda(l,n,\delta)
&\geq \frac{\lambda(n-z^*)}{\lambda(n-z^*)+z^*-l+1+\sqrt{\lambda l} }
\geq 1-\frac{z^*-l+1+\sqrt{\lambda l}}{\lambda n }.
\end{align}
\end{lemma}

\section{Proofs of Theorems \ref{thm:zeta3Asympt} and \ref{thm:zeta2Asympt}}\Label{sec:ProofThms}

\subsection{Proof of Theorem \ref{thm:zeta3Asympt}}\Label{PF-Th61}

We shall
first prove~\eqref{eq-zeta2krn} and then prove Eq.~\eqref{eq:zeta2deltarnN}.

\begin{proof}[Proof of \eqref{eq-zeta2krn}]

Since the case  $s=0$ follows from \eqref{eq:zeta2ninfty},
we focus on the case $0<s<1$.
In the following, we let $l=\lceil\nu s n\rceil$.
The proof is composed of two steps.

\noindent \textbf{Step 1:} The aim of this step is to show that  \eref{eq:zetaBoundsklarge} holds when $0<s<1$.

When $0<s<1$, in the limit $n\to+\infty$ we have
\begin{align}\label{eq:limz*nlarge}
z^*
= \frac{l-\Phi^{-1}(\delta)\sqrt{\lambda l }}{\nu}+O(1)
= sn - \frac{\Phi^{-1}(\delta)\sqrt{\lambda\nu s n }}{\nu}+O(1)
<n ,
\end{align}
where the first equality follows from \lref{lem:z-kinfity}.
Hence, the relation $B_{n,l}< B_{z^*,l}\leq\delta$ holds as $n\to+\infty$,
and we have
$\overline{\epsilon}_\lambda^{\rm c}(l,n,\delta)>0$ according to \tref{thm:AdvFidelity}.
This fact together with \coref{cor:zeta2Bound1} implies that
\begin{align}\Label{eq:zetaBoundsklarge}
\frac{g_{z^*+1}(\lceil\nu s n\rceil,n,\lambda)}{h_{z^*+1}(\lceil\nu s n\rceil,n,\lambda)}
\leq \overline{\epsilon}_\lambda^{\rm c}(\lceil\nu s n\rceil,n,\delta)
\leq \frac{g_{z_*}(\lceil\nu s n\rceil,n,\lambda)}{h_{z_*}(\lceil\nu s n\rceil,n,\lambda)}
\end{align}
as $n\to+\infty$.

\noindent \textbf{Step 2:} The aim of this step is to prove \eref{eq-zeta2krn}
in the case $0<s<1$.

In the limit $n\to\infty$, we have
\begin{align}
\begin{split}
\frac{z_*}{n-z_*+1}
&\stackrel{(a)}{=}
 \frac{\nu s n-\Phi^{-1}(\delta)\sqrt{\lambda \nu s n}+O(1)}{\nu n-\nu s n+\Phi^{-1}(\delta)\sqrt{\lambda \nu s n}+O(1)}
 = \frac{\nu s -\Phi^{-1}(\delta)\sqrt{\lambda \nu s }\frac{1}{\sqrt{n}}+O(\frac{1}{n})}{\nu -\nu s +\Phi^{-1}(\delta)\sqrt{\lambda \nu s }\frac{1}{\sqrt{n}}+O(\frac{1}{n})} \\
&= \left[\nu s -\frac{\Phi^{-1}(\delta)\sqrt{\lambda \nu s }}{\sqrt{n}}+O(\frac{1}{n})\right] \!
   \left[\frac{1}{\nu -\nu s } - \frac{\Phi^{-1}(\delta)\sqrt{\lambda \nu s }}{(\nu -\nu s )^2\sqrt{n}} +O(\frac{1}{n}) \right] \\
&= \frac{s}{1 - s } -\frac{\Phi^{-1}(\delta)\sqrt{\lambda  s}}{(1 - s )^2 \sqrt{\nu n}}+O(\frac{1}{n}).
\end{split}\Label{XMO}
\end{align}
where $(a)$ follows from \eref{eq:limz*nlarge}.
Then, in the limit $n\to\infty$ we have
\begin{align}\Label{eq:RHSgz*/hz*}
	\begin{split}
		&\text{RHS of \eqref{eq:zetaBoundsklarge}}
		=\left(1 +\frac{z_* B_{z_*-1,\lceil \nu s n\rceil} }
			{(n-z_*+1) B_{z_*,\lceil \nu sn\rceil}  }\right)^{-1}\\
		&\stackrel{(a)}{=}
		\left( 1 + \left[\frac{ s}{1 - s} -\frac{\Phi^{-1}(\delta)\sqrt{\lambda  s}}{ (1 -s)^2 \sqrt{\nu n}}+O(\frac{1}{n})\right]\!
			\left[1+\frac{\phi(\Phi^{-1}(\delta)) \nu}{\delta \sqrt{\lambda \nu sn}} +O(\frac{1}{n})\right]\right)^{-1} \\
		&=
		     \left( \frac{ 1}{1 - s} + \frac{ s}{1 - s}\frac{\phi(\Phi^{-1}(\delta)) \nu}{\delta \sqrt{\lambda \nu sn}}-\frac{\Phi^{-1}(\delta)\sqrt{\lambda  s}}{ (1 -s)^2 \sqrt{\nu n}}+O(\frac{1}{n})\right) ^{-1} \\
        &=
		      (1-s)
		\left( 1- \frac{s\phi(\Phi^{-1}(\delta)) \nu}{\delta \sqrt{\lambda \nu sn}}+\frac{\Phi^{-1}(\delta)\sqrt{\lambda  s}}{ (1 -s) \sqrt{\nu n}}+O(\frac{1}{n})
		\right) \\
		&=
		      1-s+\frac{\Phi^{-1}(\delta)\sqrt{\lambda  s}}{\sqrt{\nu n}} - \frac{s(1-s)\phi(\Phi^{-1}(\delta)) \nu}{\delta \sqrt{\lambda \nu sn}} +O(\frac{1}{n})\\
		&=
		      1-s -C(\lambda,s,\delta)\sqrt{\frac{1}{n}}+O\Big(\frac{1}{n}\Big).
	\end{split}
\end{align}
where $(a)$ follows from \eref{eq:Bz*/Bz*withk2} in \lref{lem:limBz/Bz}
and \eref{XMO}.

Using a similar reasoning for deriving \eref{eq:RHSgz*/hz*}, we can deduce that
\begin{align}\Label{eq:LHSgz*/hz*}
\text{LHS of \eqref{eq:zetaBoundsklarge}}
=1-s -C(\lambda,s,\delta)\sqrt{\frac{1}{n}}+O\Big(\frac{1}{n}\Big).
\end{align}
Then Eqs. \eqref{eq:zetaBoundsklarge}, \eqref{eq:RHSgz*/hz*}, and \eqref{eq:LHSgz*/hz*} together confirm \eref{eq-zeta2krn}.
\end{proof}

\begin{proof}[Proof of \eqref{eq:zeta2deltarnN}]
	In the following, we let $l=\lceil\nu s n\rceil$ and $\delta=\rme^{-rn}$.
The proof is composed of two steps.

\noindent \textbf{Step 1:} In this step we consider the case with $r > D(s\nu\|\nu)$.
In this case we have
\begin{align}
	\lim_{n\to\infty} \frac{\ln\delta}{\ln B_{n,l}}
	\stackrel{(a)}{\geq} \lim_{n\to\infty} \frac{\ln(\rme^{-r n})}{\ln \!\Big[\frac{1}{\rme\sqrt{\lceil\nu s n\rceil}}\rme^{-D(\lceil\nu s n\rceil/n\|\nu)n}\Big]}
	= \frac{r}{D(\nu s\|\nu)}
	> 1,
\end{align}
where $(a)$ follows from the reverse Chernoff bound given in \lref{lem:ChernoffRev}.
This fact together with \tref{thm:AdvFidelity} implies that $\overline{\epsilon}_\lambda^{\rm c}(l,n,\delta=\rme^{-r n})=0$
as $n\to\infty$, which confirms \eref{eq:zeta2deltarnN}.

\noindent \textbf{Step 2:} In this step we consider the case with $0 < r \leq D(s\nu\|\nu)$.

\noindent \textbf{Step 2.1:} The aim of this step is to show that
\begin{align}\label{z*nlargeStep2.1}
z^*=r t_\mathrm{D}\Big(\frac{\nu s }{r} ,\nu\Big) n +o(n),
\end{align}
where $t_\mathrm{D}(\frac{\nu s }{r} ,\nu)$ is defined in \eref{eq:t2rnueps}.

Eqs. \eqref{eq:ChernoffB} and \eqref{eq:ChernoffRev} imply that
\begin{align}
	\frac{1}{\rme \sqrt{\lceil\nu s n\rceil}}  \rme^{-z^* D(\frac{\lceil\nu s n\rceil}{z^*}\| \nu)}
	\leq B_{z^*, \lceil\nu s n\rceil}
	\leq \delta=\rme^{-rn}
	\leq B_{z_*, \lceil\nu s n\rceil}
	\leq \rme^{-z_* D(\frac{\lceil\nu s n\rceil}{z_*}\| \nu)}.
\end{align}
It follows that
\begin{align}
r= \lim_{n\to\infty} \frac{z_*}{n} D\left(\frac{\nu s n}{z^*}\Big\| \nu\right)
= \left(\lim_{n\to\infty} \frac{z_*}{n}\right) D\left(\frac{\nu s }{\lim_{n\to\infty} (z^*/n)}\bigg\| \nu\right),
\end{align}
and thus $\lim_{n\to\infty} (z_*/n)=r t_\mathrm{D}(\frac{\nu s }{r} ,\nu)$, which confirms \eref{z*nlargeStep2.1}.

\noindent \textbf{Step 2.2:} The aim of this step is to show
\eref{eq:zeta2deltarnN} by using \eref{z*nlargeStep2.1} and \lref{cor:zeta2Bound2}.

According to \lref{cor:zeta2Bound2} we have
\begin{align}\Label{eq:zeta2Step2.2}
	\frac{\lambda(n-z^*)}{\lambda(n-z^*)+z^*-l+1+\sqrt{\lambda l} }
	\leq
	\overline{\epsilon}_\lambda^{\rm c}(l,n,\delta)
	\leq \max\left\{ 0, \frac{\lambda(n-z_*+1)}{\lambda(n-z_*+1)+z_*- l}\right\}.
\end{align}
Then we have
\begin{align}\Label{eq:limLHSzeta2Step2.2}
\begin{split}
	\lim_{n\to\infty} \text{LHS of \eqref{eq:zeta2Step2.2}}
	&=
	\lim_{n\to\infty} \frac{\lambda(1-\frac{z^*}{n})}{\lambda(1-\frac{z^*}{n})+\frac{z^*}{n}+ \frac{1-\lceil\nu s n\rceil+\sqrt{\lambda \lceil\nu s n\rceil}}{n} }\\
	&\stackrel{(a)}{=}
	\frac{\lambda(1-r t_\mathrm{D}(\frac{\nu s }{r} ,\nu))}{\lambda(1-r t_\mathrm{D}(\frac{\nu s }{r} ,\nu))+r t_\mathrm{D}(\frac{\nu s }{r} ,\nu)- \nu s }
	=1-E_{\lambda, s}(r),
\end{split}
\end{align}
where $(a)$ follows from \eref{z*nlargeStep2.1}.
By a similar reasoning we can deduce that
\begin{align}
	\begin{split}
		\lim_{n\to\infty} \frac{\lambda(n-z_*+1)}{\lambda(n-z_*+1)+z_*- l}
		=1-E_{\lambda, s}(r)\geq 0
		\quad \forall 0 < r \leq D(s\nu\|\nu).
	\end{split}
\end{align}
Hence we have $\lim_{n\to\infty} \text{RHS of \eqref{eq:zeta2Step2.2}}=1-E_{\lambda, s}(r)$. This fact together with \eqref{eq:limLHSzeta2Step2.2} and \eqref{eq:zeta2Step2.2} confirms \eref{eq:zeta2deltarnN} in the case $0 < r \leq D(s\nu\|\nu)$.
\end{proof}

\subsection{Proof of Theorem \ref{thm:zeta2Asympt}}\Label{PF-Th51}
We shall prove
first prove Eq. \eqref{eq:zeta2ninfty} and then prove Eq. \eqref{eq:zeta2deltarn}.

\begin{proof}[Proof of \eref{eq:zeta2ninfty}]
	In the following, we let $l=\lceil\nu k_0\rceil$.
	The proof is composed of two steps.


\noindent \textbf{Step 1:} The aim of this step is to prove \eref{eq:zeta2ninfty}
in the case  $B_{z^*,l}<\delta$.

In this case, when $n$ goes to infinity, we have $B_{n,l}<\delta$ and $h_z(l,n,\lambda)\to B_{z,l}$ for any finite $z\in\bbZ^{\geq l+1}$.
So in this limit we have $\hat{z}=z_*$,
where $\hat{z}$ denotes the largest integer $z$ such that
$h_z(l,n,\lambda)\geq\delta$.
This fact together with \tref{thm:AdvFidelity} implies that
\begin{align}
\begin{split}
\overline{\epsilon}^{\rm c}_\lambda(l,n,\delta)
&= \frac{\tilde{\zeta}_\lambda(l,n,\delta,z_*)}{\delta}
=\frac{1}{\delta}\left[(1-\kappa_{z_*}(l,n,\delta,\lambda))g_{z^*}(l,n,\lambda)
   +\kappa_{z_*}(l,n,\delta,\lambda)g_{z_*}(l,n,\lambda) \right].\\
&=1-\frac{\delta z_*B_{z_*-1,l}- \delta z^*B_{z_*,l} + z^*B_{z_*,l}B_{z_*,l}-z_*B_{z_*-1,l}B_{z^*,l}}{\delta(n+1)[h_{z_*}(l,n,\lambda)-h_{z^*}(l,n,\lambda)]}.
\end{split}
\end{align}
It follows that
\begin{align}
\begin{split}
&\lim_{n\to\infty} n\overline{\epsilon}_\lambda(l,n,\delta)
=\lim_{n\to\infty} n[1-\overline{\epsilon}^{\rm c}_\lambda(l,n,\delta)]\\
&= \lim_{n\to\infty} \frac{\delta z_*B_{z_*-1,l}- \delta z^*B_{z_*,l} + z^*B_{z_*,l}B_{z_*,l}-z_*B_{z_*-1,l}B_{z^*,l}}{\delta[h_{z_*}(l,n,\lambda)-h_{z^*}(l,n,\lambda)]}\\
&= \frac{(\delta-B_{z^*,l}) z_*B_{z_*-1,l}+ (B_{z_*,l}-\delta) z^*B_{z_*,l} }{\delta(B_{z_*,l}-B_{z^*,l})}
=G(l,\delta,\lambda), \Label{eq:limn(1-zeta2)}
\end{split}\end{align}
which confirms \eref{eq:zeta2ninfty}.

\noindent \textbf{Step 2:} The aim of this step is to prove \eref{eq:zeta2ninfty}
in the case  $B_{z^*,l}=\delta$.

In this case, when $n\to\infty$, we have $B_{n,k}<\delta$ and $\hat{z}=z^*$.
Then by a similar reasoning that leads to \eref{eq:limn(1-zeta2)}, we can deduce that
\begin{align}
	\lim_{n\to\infty} n\overline{\epsilon}_\lambda(l,n,\delta)
    = \frac{(B_{z_*,l}-\delta) z^*B_{z_*,l} }{\delta(B_{z_*,l}-B_{z^*,l})}
	= G(l,\delta,\lambda),
\end{align}
which confirms \eref{eq:zeta2ninfty} again.
\end{proof}

\begin{proof}[Proof of \eref{eq:zeta2deltarn}]
In the following, we let $l=\lceil\nu k_0\rceil$, $\delta=\rme^{-r n}$.
The proof is composed of two steps.

\noindent \textbf{Step 1:} The aim of this step is to prove \eref{eq:zeta2deltarn}
in the case $r>\ln\lambda^{-1}$.

When $\delta=\rme^{-r n}$ with $r>\ln\lambda^{-1}$, we have
\begin{align}\label{eq:lndel/lnBnl}
\lim_{n\to\infty} \frac{\ln\delta}{\ln B_{n,l}}
\stackrel{(a)}{\geq} \lim_{n\to\infty} \frac{\ln(\rme^{-r n})}{\ln \!\Big[\frac{1}{\rme\sqrt{l}}\rme^{-D(l/n\|\nu)n}\Big]}
= \lim_{n\to\infty} \frac{r}{D(l/n\|\nu)+\frac{\ln(\rme\sqrt{l})}{n}}
\stackrel{(b)}{=} \frac{r}{\ln\lambda^{-1}}
> 1,
\end{align}
where $(a)$ follows from the reverse Chernoff bound given in \lref{lem:ChernoffRev},
and $(b)$ follows because $D(0\| \nu)=\ln\lambda^{-1}$.
\eref{eq:lndel/lnBnl} implies that $B_{n,l}>\delta$ when $n\to\infty$.
This fact together with \tref{thm:AdvFidelity} imply that $\overline{\epsilon}_\lambda^{\rm c}(l,n,\delta=\rme^{-r n})=0$
as $n\to\infty$, which confirms \eref{eq:zeta2deltarn}.

\noindent \textbf{Step 2:} The aim of this step is to prove \eref{eq:zeta2deltarn} in the case $0< r\leq \ln\lambda^{-1}$.

When $\delta=\rme^{-r n}$ with $0< r\leq\ln\lambda^{-1}$, we have
\begin{align}\Label{eq:zeta2thm233}
\frac{\lambda(n-z^*)}{\lambda(n-z^*)+z^*-l+1+\sqrt{\lambda l} }
\leq
\overline{\epsilon}_\lambda^{\rm c}(l,n,\delta)
\leq \max\left\{ 0, \frac{\lambda(n-z_*+1)}{\lambda(n-z_*+1)+z_*- l}\right\}
\end{align}
according to \lref{cor:zeta2Bound2}.
In the following we shall show \eref{eq:zeta2deltarn} by proving that
\begin{align}\Label{eq:LRHSzeta2thm233}
\lim_{n\to\infty} \text{LHS of \eqref{eq:zeta2thm233}}
=
\lim_{n\to\infty} \text{RHS of \eqref{eq:zeta2thm233}}
=\frac{\lambda\ln\lambda+\lambda r}{\lambda\ln\lambda-\nu r} .
\end{align}

Note that
\begin{align}\Label{eq:LHSzeta2thm233}
\begin{split}
\lim_{n\to\infty} \text{LHS of \eqref{eq:zeta2thm233}}
&= \lim_{\delta\to 0} \frac{\lambda(\frac{n}{\ln\delta}-\frac{z^*}{\ln\delta})}
                           {\lambda(\frac{n}{\ln\delta}-\frac{z^*}{\ln\delta})+\frac{z^*}{\ln\delta}+\frac{1-l+\sqrt{\lambda l} }{\ln\delta} } \\
&= \frac{\lambda(-\frac{1}{r}-\frac{1}{\ln\lambda})}{\lambda(-\frac{1}{r}-\frac{1}{\ln\lambda})+\frac{1}{\ln\lambda}}
=\frac{\lambda\ln\lambda+\lambda r}{\lambda\ln\lambda-\nu r} ,
\end{split}\end{align}
where the second equality follows from \lref{lem:z-kwithdelta} and the relation $\delta=\rme^{-rn}$.
Using a similar reasoning, we can deduce that
\begin{align}\Label{eq:RHSzeta2thm233}
\lim_{n\to\infty} \frac{\lambda(n-z_*+1)}{\lambda(n-z_*+1)+z_*- l}
=\frac{\lambda\ln\lambda+\lambda r}{\lambda\ln\lambda-\nu r},
\end{align}
which is positive when $0< r<\ln\lambda^{-1}$, and is equal to $0$ when $r=\ln\lambda^{-1}$.
These observations confirm \eref{eq:LRHSzeta2thm233} and complete the proof of \eref{eq:zeta2deltarn}.
\end{proof}

\section{Conclusion}\Label{S12}
We have studied one-sided hypothesis testing
under random sampling without replacement.
That is, when $n+1$ binary random variables $X_1,\ldots, X_{n+1}$
are subject to a permutation invariant distribution and
$n$ binary random variables $X_1,\ldots, X_{n}$ are observed,
we have proposed randomized tests
with a randomization parameter $\lambda>0$
for the upper confidence limit of the expectation of the
$n+1$-th random variable $X_{n+1}$
under a given significance level $\delta>0$.
We have shown that
our proposed randomized test significantly improves
deterministic test unlike random sampling with replacement.
For example, as shown in \eref{eq-zeta2krn} and \eref{eq:MMD}
of Theorem \ref{thm:zeta3Asympt},
the asymptotic upper confidence limit in our randomized test
is much better than 
that in the deterministic test in sampling without replacement.
This problem setting appears under a certain adversarial setting.
We can expect that our result will be applied to various areas
when the data generated in an adversarial way.

\begin{appendix}

\section{Useful lemmas}\Label{app:Usefullemma}
\subsection{Binomial probability}\Label{app:BinProb}
To show several important statements,
here we prepare several useful lemmas to characterize $B_{z,l}(p)$ and
$\Delta_{z,l}$, which are defined in \eqsref{eq:binomCFD}{eq:Deltazkmain} \main, respectively.
If there is no danger of confusions, we shall use $B_{z,l}$ as shorthand for $B_{z,l}(\nu)$, where
$\nu:=1-\lambda$ and $\lambda$ is the randomization parameter introduced in \eref{eq:deflambda} \main.

\begin{lemma}[Lemma 3.2 of Ref.~\cite{Binomial22}]\Label{lem:Bzkmono}
Suppose $0 \leq l \leq z$ and $0<p<1$. Then $B_{z, l}(p)$ is strictly increasing in $l$, strictly decreasing in $z$, and
nonincreasing in $p$.
\end{lemma}

The Chernoff bound states that
\begin{align}
	B_{z,l}\leq \rme^{-z D(\frac{l}{z}\| \nu )}\quad \forall l\leq \nu z,  \Label{eq:ChernoffB}
\end{align}
where $	D(p\|q)$ is the  relative entropy defined in \eref{eq:RelEntropy}  \main.
The following lemma provides a reverse Chernoff bound for $B_{z,l}$.

\begin{lemma}[Proposition 5.4 of Ref.~\cite{Binomial22}]\Label{lem:ChernoffRev}
	Suppose $l,z$ are positive integers and satisfy $l\leq z-1$, then
	\begin{align}
		B_{z,l} \geq \frac{1}{\rme \sqrt{l}} \rme^{-zD(\frac{l}{z}\|\nu )}	.  \Label{eq:ChernoffRev}
	\end{align}
\end{lemma}

The following lemma is a simple corollary of the Berry-Esseen theorem \cite{Berry41}
\begin{lemma}\Label{lem:BerryEsseen}
	Suppose $0<\lambda<1$, and integers $z\geq l\geq 1$. Then
	\begin{align}
		\left| B_{z,l} - \Phi\left(\frac{l - \nu z}{\sqrt{\nu\lambda z}}\right) \right|
		\leq \frac{0.5(1-2\nu\lambda)}{\sqrt{\nu\lambda z}}
		\leq \frac{0.5}{\sqrt{\nu\lambda l}} ,
	\end{align}
	where $\Phi(x)$ is the standard normal distribution function defined in \eref{eq:Phix} \main.
\end{lemma}

\begin{lemma}\Label{lem:Bzkz}
 Suppose $0<\lambda<1$ and integers $0\leq l\leq z$. Then $B_{z, l} / B_{z+1, l}$ is nondecreasing in $z$, and satisfies
\begin{align}
\frac{z-l+1}{(z+1)\lambda}
\leq \frac{B_{z,l}}{B_{z+1,l}} \leq \frac{1}{\lambda}.  \label{eq:Bzkz11}
\end{align}
In addition, if $l\leq \nu z$, then
\begin{align}
\frac{B_{z,l}}{B_{z+1,l}}
\leq \frac{z-l+1+\sqrt{\lambda l}}{(z+1)\lambda} . \label{eq:Bzkz22}
\end{align}
\end{lemma}

\begin{proof}[Proof of \lref{lem:Bzkz}]
If $l=0$, then $B_{z, l} / B_{z+1, l}=1/\lambda$ is nondecreasing in $z$.
If $l\geq 1$, then the monotonicity of $B_{z, l} / B_{z+1, l}$ with $z$ follows from Lemma 3.4 of Ref.~\cite{Binomial22}.

In addition, \eref{eq:Bzkz11} follows from Lemma 3.4 of Ref.~\cite{Binomial22}, and
\eref{eq:Bzkz22} follows from Proposition 4.10 of Ref.~\cite{Binomial22}.
\end{proof}

\begin{lemma}\Label{lem:Deltazk}
Suppose $0<\lambda<1$ and $z\geq l+1$; then $\Delta_{z,l}/\Delta_{z-1,l}$ is  nonincreasing in $z$ and nondecreasing in $l$.
\end{lemma}

\begin{proof}[Proof of Lemma \ref{lem:Deltazk}]
	The equality
	\begin{align}
		B_{z+1,l}=(1-\lambda)B_{z,l-1}+\lambda B_{z,l},
	\end{align}
	implies that
	\begin{align}\label{eq:Deltazkequal}
		\Delta_{z,l}=(1-\lambda)(B_{z,l}-B_{z,l-1})=(1-\lambda)b_{z,l}
		=\begin{cases}
			0 & l\geq z+1,\\
			(1-\lambda)\binom{z}{l}(1-\lambda)^l\lambda^{z-l} &l\leq z.
		\end{cases}
	\end{align}

	When $z\geq l+1$, \eref{eq:Deltazkequal} implies that
	\begin{align}
		\frac{ \Delta_{z,l}}{\Delta_{z-1,l}}=\frac{\binom{z}{l}\lambda}{\binom{z-1}{l}}=\frac{z\lambda}{z-l},
	\end{align}
	so $\Delta_{z,l}/\Delta_{z-1,l}$ is  nonincreasing in $z$ and nondecreasing in $l$.
\end{proof}


\subsection{Relative entropy}\Label{app:RelEntropy}
To show that $t_\mathrm{D}(\gamma,x)$, $\epsilon_D(s,r)$ and $s_D(\epsilon,r)$ are well defined in Section \ref{notation} \main, we prepare the following lemma.

\begin{lemma}\Label{lem:DpqMonoton}
	The relative entropy $D(p\|q)$ (defined in \eref{eq:RelEntropy}) \main is nondecreasing in $q$ and nonincreasing in $p$ when $0\leq p\leq q\leq 1$.
\end{lemma}

\begin{proof}[Proof of \lref{lem:DpqMonoton}]
	When $0\leq p\leq q\leq 1$ we have
	\begin{align}
		\frac{\partial D(p\|q)}{\partial q} &= \frac{q-p}{q(1-q)} \geq 0,  \\
		\frac{\partial D(p\|q)}{\partial p} &= \ln \frac{p}{q} + \ln \frac{1-q}{1-p} \leq 0 ,
	\end{align}
	which confirm the lemma.
\end{proof}

\subsection{Poissson distribution}\Label{app:Poissson}

To show that $t_\mathrm{P}(k,\delta)$  is well defined in \eref{eq:t1kdelta} \main, we prepare the following lemma.

\begin{lemma}\label{lem:monotPois}
	Suppose $k \in \bbZ^{\geq 0}$. Then $\mathrm{Pois}(k,x)$ (defined in \eref{eq:PoisFunc} \main) is monotonically decreasing in $x$ for $x\geq0$.
\end{lemma}

\begin{proof}[Proof of \lref{lem:monotPois}]
	For $x>0$ we have
	\begin{align}
		\frac{\partial \,\mathrm{Pois}(k,x)}{\partial x}
		=-\rme^{-x} \sum_{i=0}^{k} \frac{x^{i}}{i !}+\rme^{-x} \sum_{i=1}^{k} i \frac{x^{i-1}}{i !}
		=-\rme^{-x} \sum_{i=0}^{k} \frac{x^{i}}{i !}+\rme^{-x} \sum_{i=0}^{k-1} \frac{x^{i}}{i !}
		=-\rme^{-x} \frac{x^{k}}{k !}
		<0,
	\end{align}
	which confirms the lemma.
\end{proof}

\section{Properties of $z^*$ and  $B_{z_*,l}/B_{z^*,l}$}\Label{app:Propertyz*}

Given $0<\delta\leq1$, $z^*(l,\delta,\lambda)$ is the smallest integer $z$ that satisfies $B_{z,l}\leq \delta$, as defined in \eref{eq:definez} \main.
If there is no danger of confusions, we shall use $z^*$ as shorthand for $z^*(l,\delta,\lambda)$, and use $z_*:=z^*-1$. This appendix provides some
properties of $z^*$ and  $B_{z_*,l}/B_{z^*,l}$, which are useful for proving several important statements.

\subsection{Lower bound for $z^*$}\Label{app:Boundsz*}
\begin{lemma}[Corollary 1, \cite{Kaas80}]\label{lem:BzklamOneHalf}
	Suppose $z\in\bbZ^{\geq 0}$ and $0\leq p\leq1$, then $B_{z,\lceil zp \rceil}(p)\geq1/2$.
\end{lemma}

\begin{lemma}[Theorem 4.1, \cite{Nowako21}]\label{lem:irrational}
	Suppose $l\in\bbZ^{\geq 0}$ and $z\in\bbZ^{\geq l+1}$, then $B_{z,l}(l/z)\ne1/2$.
\end{lemma}

\begin{corollary}\label{cor:Bnk1-k/n>1/2}
	Suppose $l\in\bbZ^{\geq1}$ and $z\in\bbZ^{\geq l}$, then $B_{z,l}(l/z)>1/2$.
\end{corollary}

\begin{proof}
	First, in the case $z=l$, we have $B_{z,l}(l/z)=B_{z,z}(1)=1\geq1/2$.
	Second, in the case $z\geq l+1$, the corollary immediately follows from \lrefs{lem:BzklamOneHalf}{lem:irrational}.
\end{proof}

\begin{corollary}\label{cor:z*k/nu}
	Suppose $0<\lambda<1$ and $0<\delta\leq 1/2$. Then $z^*\geq l/\nu $.
\end{corollary}

\begin{proof}
	When $l=0$, the lemma holds clearly.
	
	Next, we shall consider $l\in\bbZ^{\geq1}$.
	In this case we have $\lfloor l/\nu \rfloor\geq l$
	and
	\begin{align}\label{eq:Coro10B2}
		B_{\lfloor l/\nu \rfloor, l}(\nu)
		\stackrel{(a)}{\geq} B_{\lfloor l/\nu \rfloor, l}\left(\frac{l}{\lfloor l/\nu \rfloor} \right)
		\stackrel{(b)}{>} 1/2,
	\end{align}
	where $(a)$ follows from \lref{lem:Bzkmono} and $\nu \leq l/\lfloor l/\nu \rfloor$, and
	$(b)$ follows from \coref{cor:Bnk1-k/n>1/2} above.
	Therefore,
	\begin{align}
		z^*
		=\min\{z\in\bbZ^{\geq l}|B_{z,l}(\nu)\leq \delta\}
		\stackrel{(a)}{\geq} \min\{z\in\bbZ^{\geq l}|B_{z,l}(\nu)\leq 1/2\}
		\stackrel{(b)}{\geq} \lfloor l/\nu \rfloor+1
		\geq l/\nu,
	\end{align}
	which confirms the lemma.
	Here $(a)$ follows from the assumption $\delta\leq 1/2$, and
	$(b)$ follows from \eref{eq:Coro10B2}.
\end{proof}

\subsection{Limits of $z^*$}
\begin{lemma}\Label{lem:z-kwithdelta}
	Suppose $l\in\bbZ^{\geq 0}$, $0<\lambda<1$ and $0<\delta<1$, then
	\begin{align}
		\lim_{\delta\to 0}\, \frac{z^*}{\ln\delta} &=\frac{1}{\ln \lambda}. \Label{eq:z*highpresion}
	\end{align}
\end{lemma}

\begin{proof}[Proof of Lemma \ref{lem:z-kwithdelta}]
	By large deviation principle, we have
	\begin{align}
		\lim_{z \to \infty}\frac{1}{z}\ln (B_{z,l}) = -D( 0\|1-\lambda)
		=\ln \lambda.
	\end{align}
	Since we have $z^*\to \infty$ when $\delta\to 0$, the above relation implies that
	\begin{align}
		\lim_{\delta\to 0}\, \frac{\ln\delta}{z^*}
		=\lim_{\delta \to 0}\frac{\ln (B_{z^*,l})}{z^*}
		=\lim_{z^* \to \infty}\frac{\ln (B_{z^*,l})}{z^*}
		=\ln \lambda,
	\end{align}
	which confirms the lemma.
\end{proof}

\begin{lemma}\Label{lem:z-kinfity}
	Suppose $l\in\bbZ^{\geq 0}$, $0<\lambda<1$ and $0<\delta<1$. Then in the limit $l\to+\infty$ we have
	\begin{align}
		z^* &= \frac{l-\Phi^{-1}(\delta)\sqrt{\lambda l}}{\nu}+O(1),
		\Label{eq:z=withkinfity}
	\end{align}
	where $z^*$ is the smallest integer $z$ such that $B_{z,l}\leq \delta$, and $\Phi(x)$ is defined in \eref{eq:Phix} \main.
\end{lemma}

\begin{proof}[Proof of Lemma \ref{lem:z-kinfity}]
	By \lref{lem:BerryEsseen}, we have
	\begin{align}\label{eq:B24}
		\Phi\bigg(\frac{l - \nu z}{\sqrt{\nu \lambda z}}\bigg) - \frac{0.5}{\sqrt{\nu \lambda l}}
		\leq B_{z,l}
		\leq \Phi\bigg(\frac{l - \nu z}{\sqrt{\nu \lambda z}}\bigg) + \frac{0.5}{\sqrt{\nu \lambda l}}.
	\end{align}
	When $l\to \infty$, we have $\delta-\frac{0.5}{\sqrt{\nu \lambda l}}>0$, and
	\begin{align}\label{eq:B25}
		z^*
		&=\min\{z\in\bbZ^{\geq l}|B_{z,l}\leq \delta\}
		\leq
		\bigg\lceil \bigg\{z \in \mathbb{R}^+  \,\bigg|\,
		\Phi\bigg(\frac{l - \nu z}{\sqrt{\nu \lambda z}}\bigg) + \frac{0.5}{\sqrt{\nu \lambda l}} =\delta  \bigg\}\bigg\rceil
		\nonumber \\
		&= \bigg\lceil \bigg\{z \in \mathbb{R}^+  \,\bigg|\,
		\frac{l - \nu z}{\sqrt{\nu \lambda z}}=  \Phi^{-1}\bigg( \delta-\frac{0.5}{\sqrt{\nu \lambda l}}\bigg)  \bigg\}\bigg\rceil
		\nonumber \\
		&= \left\lceil \left(
		- \sqrt{\frac{\lambda}{4\nu}} \, \Phi^{-1}\bigg(\delta-\frac{0.5}{\sqrt{\nu \lambda l}}\bigg)
		+\sqrt{\frac{l}{\nu}+\frac{\lambda}{4\nu}\bigg[\Phi^{-1}\bigg(\delta-\frac{0.5}{\sqrt{\nu \lambda l}}\bigg)\bigg]^2}
		\right)^2 \right\rceil
		\nonumber \\
		&= \left\lceil \left(
		- \sqrt{\frac{\lambda}{4\nu}} \, \Phi^{-1}(\delta)+O\Big(\frac{1}{\sqrt{l}}\Big)
		+\sqrt{\frac{l}{\nu}+O(1)}
		\right)^2 \right\rceil
		=\frac{l-\Phi^{-1}(\delta)\sqrt{\lambda l}}{\nu}+O(1),
	\end{align}
	where the inequality follows from the upper bound in \eref{eq:B24}.
	
	In addition, when $l\to \infty$, we have $\delta+\frac{0.5}{\sqrt{\nu \lambda l}}<1$, and
	\begin{align}\label{eq:B26}
		z^*
		&=\min\{z\in\bbZ^{\geq l}|B_{z,l}\leq \delta\}
		\geq
		\bigg\{z \in \mathbb{R}^+  \,\bigg|\,
		\Phi\bigg(\frac{l - \nu z}{\sqrt{\nu \lambda z}}\bigg) -\frac{0.5}{\sqrt{\nu \lambda l}} =\delta  \bigg\}
		\nonumber \\
		&= \bigg\{z \in \mathbb{R}^+  \,\bigg|\,
		\frac{l - \nu z}{\sqrt{\nu \lambda z}}=  \Phi^{-1}\bigg( \delta+\frac{0.5}{\sqrt{\nu \lambda l}}\bigg)  \bigg\}
		\nonumber \\
		&= \left(
		- \sqrt{\frac{\lambda}{4\nu}} \, \Phi^{-1}\bigg(\delta+\frac{0.5}{\sqrt{\nu \lambda l}}\bigg)
		+\sqrt{\frac{l}{\nu}+\frac{\lambda}{4\nu}\bigg[\Phi^{-1}\bigg(\delta+\frac{0.5}{\sqrt{\nu \lambda l}}\bigg)\bigg]^2}
		\right)^2
		\nonumber \\
		&= \left(
		- \sqrt{\frac{\lambda}{4\nu}} \, \Phi^{-1}(\delta)+O\Big(\frac{1}{\sqrt{l}}\Big)
		+\sqrt{\frac{l}{\nu}+O(1)}
		\right)^2
		=\frac{l-\Phi^{-1}(\delta)\sqrt{\lambda l}}{\nu}+O(1),
	\end{align}
	where the inequality follows from the lower bound in \eref{eq:B24}.
	\eqsref{eq:B25}{eq:B26} together confirms the lemma.
\end{proof}

\subsection{Limit of $B_{z_*,l}/B_{z^*,l}$}

\begin{lemma}\Label{lem:limBz/Bz}
	Suppose $l\in\bbZ^{\geq 0}$, $0<\lambda<1$ and $0<\delta< 1$. Then in the limit
	$l\to+\infty$ we have
	\begin{align}
		\frac{ B_{z_*,l}}{B_{z^*,l}}
		&= 1+\frac{\phi(\Phi^{-1}(\delta)) \nu}{\delta \sqrt{\lambda l}} +O(l^{-1}), \Label{eq:Bz*/Bz*withk1} \\
		\frac{ B_{z_*-1,l}}{B_{z_*,l}}
		&= 1+\frac{\phi(\Phi^{-1}(\delta)) \nu}{\delta \sqrt{\lambda l}} +O(l^{-1}), \Label{eq:Bz*/Bz*withk2} \\
		\frac{ B_{z^*,l}}{B_{z^*+1,l}}
		&= 1+\frac{\phi(\Phi^{-1}(\delta)) \nu}{\delta \sqrt{\lambda l}} +O(l^{-1}). \Label{eq:Bz*/Bz*withk3}
	\end{align}
\end{lemma}

\begin{proof}[Proof of Lemma \ref{lem:limBz/Bz}]
	To consider the normal approximation, we define
	\begin{align}\Label{eq:x1}
		x_1:=\frac{l-z^*\nu}{\sqrt{z^*\lambda\nu} },\qquad
		x_2:=\frac{l-z_*\nu}{\sqrt{z_*\lambda\nu} }.
	\end{align}
	By \lref{lem:z-kinfity}, in the limit $l\to+\infty$ we have $z^*=l/\nu+O(\sqrt{l})$, and thus
	\begin{align}
		x_1
		= \frac{l-z^*\nu}{\sqrt{z^*\lambda\nu} }
		= \frac{l-\nu [l/\nu+O(\sqrt{l})] }{\sqrt{\lambda\nu[l/\nu+O(\sqrt{l})]} }
		= \frac{O(\sqrt{l}) }{\sqrt{\lambda l+O(\sqrt{l})} }
		= O(1).
	\end{align}
	In addition, we have
	\begin{align}\Label{eq:x1-x2}
		x_2-x_1
		&= \frac{l-z_*\nu}{\sqrt{z_*\lambda\nu} } -x_1
		= \frac{1}{\sqrt{z_*}}\sqrt{\frac{\nu}{\lambda}}
		+ \frac{\sqrt{z^*}}{\sqrt{z_*}} \left(\frac{l-z^*\nu}{\sqrt{z^*\lambda\nu} } \right) -x_1 \nonumber\\
		&= \frac{1}{\sqrt{z_*}}\sqrt{\frac{\nu}{\lambda}} + \left( \sqrt{1+\frac{1}{z_*}} -1 \right) x_1 \nonumber\\
		&\stackrel{(a)}{=} \frac{1}{\sqrt{\frac{l}{\nu}+O(\sqrt{l})}}\sqrt{\frac{\nu}{\lambda}}+ \left( \sqrt{1+O(\frac{1}{l})} -1 \right) x_1 \nonumber\\
		&\stackrel{(b)}{=} \left[1+O\Big(\frac{1}{\sqrt{l}}\Big)\right] \frac{\nu}{\sqrt{\lambda l}} + O(l^{-1})
		= \frac{\nu}{\sqrt{\lambda l}}+O(l^{-1}),
	\end{align}
	where $(a)$ follows because $z_*=l/\nu+O(\sqrt{l})$ and $1/z_*=O(1/l)$ (see \lref{lem:z-kinfity}),
	and $(b)$ follows because $x_1=O(1)$.
	
	According to \lref{lem:BerryEsseen} we have
	\begin{align}
		B_{z^*,l} &=
		\Phi\bigg(\frac{l-z^*\nu}{\sqrt{z^*\lambda\nu} }\bigg)+O\Big(\frac{1}{\sqrt{l}}\Big) =\Phi(x_1)+O\Big(\frac{1}{\sqrt{l}}\Big).  \Label{eq:EdgeSeries}
	\end{align}
	In addition, considering the Edgeworth series \cite[(2.1)]{Edgeworth}, we have
	\begin{align}
		\frac{ B_{z_*,l}}{B_{z^*,l}}
		&= \frac{\Phi(x_2)}{\Phi(x_1)}+ O(l^{-1})
		= \frac{\Phi(x_1)+ \phi(x_3) (x_2-x_1) }{\Phi(x_1)}+ O(l^{-1})\nonumber\\
		&= 1+\frac{\phi(x_1) \Big[\frac{\nu}{\sqrt{\lambda l}}+O(l^{-1})\Big] }{\Phi(x_1)}+ O(l^{-1})
		= 1+\frac{\phi(x_1) \nu }{\Phi(x_1)\sqrt{\lambda l}} +O(l^{-1}),\Label{eq:NJ2}
	\end{align}
	where $x_3\in(\min(x_1,x_2),\max(x_1,x_2))$, and the third equality follows from \eref{eq:x1-x2}.

	Since $B_{z^*,l}\leq \delta<B_{z_*,l}$,
	\eref{eq:NJ2} implies that $B_{z^*,l}=\delta+O(\frac{1}{\sqrt{l}})$.
	This fact together with \eref{eq:EdgeSeries} imply that
	\begin{align}
		\Phi(x_1)=\delta+O\Big(\frac{1}{\sqrt{l}}\Big), \qquad
		x_1  =\Phi^{-1}(\delta)+O\Big(\frac{1}{\sqrt{l}}\Big). \Label{eq:x1=Phidelta}
	\end{align}
	Then by plugging the above two equations into \eref{eq:NJ2} we get \eref{eq:Bz*/Bz*withk1}.
	
	Eqs.~\eqref{eq:Bz*/Bz*withk2} and \eqref{eq:Bz*/Bz*withk3} follow from \eqref{eq:Bz*/Bz*withk1} and $B_{z^*,l}=\delta+O(\frac{1}{\sqrt{l}})$, $B_{z_*,l}=\delta+O(\frac{1}{\sqrt{l}})$.
\end{proof}

\section{Derivation of \eref{XMY}}\Label{S3}
To show \eref{XMY},
we discuss $\overline{\epsilon}_\lambda^{\rm c}(l,n,\delta):=
1-\overline{\epsilon}_\lambda(l,n,\delta)$
instead of $\overline{\epsilon}_\lambda(l,n,\delta)$.
That is, we shall show
\begin{align}
	\overline{\epsilon}_\lambda^{\rm c}(0,n,\delta_z)
	=\frac{(n-z+1)\lambda}{z+(n-z+1)\lambda}
	\quad \forall z\in\{0,1,\dots,n+1\}.
	\Label{XMP3}
\end{align}

Let $Q_{z}\in {\cal Q}_{n+1}$ be the permutation invariant distribution on
$[n+1]$ such that $\sum_{i=1}^{n+1}Y_i=z$.
Then by simple calculation we have
\begin{align}
	Q_z(L=0)&=
	\left\{
	\begin{array}{ll}
		\lambda^{n} & \hbox{ when } z =n+1, \\
		\frac{(n+1-z)\lambda^z+ z \lambda^{z-1}}{n+1} =\delta_z
		& \hbox{ when } 1\le z \le n, \\
		1  & \hbox{ when } z =0,
	\end{array}
	\right.\\
	Q_z(L=0,Y_{n+1}=0)&=
	\left\{
	\begin{array}{ll}
		0 & \hbox{ when } z =n+1, \\
		\frac{(n+1-z)\lambda^z}{n+1} & \hbox{ when } 1\le z \le n, \\
		1  & \hbox{ when } z =0,
	\end{array}
	\right.
\end{align}
and thus
\begin{align}\Label{E13}
	\frac{Q_z(L=0,Y_{n+1}=0)}{Q_z(L=0)}
	&=
	\left\{
	\begin{array}{ll}
		0 &  z =n+1, \\
		\frac{(n+1-z)\lambda}{(n+1-z)\lambda+ z }\quad
		&  1\le z \le n, \\
		1  &  z =0,
	\end{array} \right\}
	=
	\frac{(n-z+1)\lambda}{z+(n-z+1)\lambda} .
\end{align}

In addition, we have
\begin{align}
	\begin{split}
		\overline{\epsilon}_\lambda^{\rm c}(0,n,\delta)
		=&
		\min_{ \{p_j\}_{j=0}^{n+1}}
		\Bigg\{
		\frac{\sum_{j=0}^{n+1} p_j Q_j(L=0,Y_{n+1}=0)
		}{\sum_{j=0}^{n+1} p_j Q_j(L=0)}
		\Bigg|
		\sum_{j=0}^{n+1} p_j Q_j(L=0)
		\ge \delta \Bigg\},
	\end{split}
\end{align}
where $p_j$ form a probability distribution on $[n+1]$.

Note that
$Q_z(L=0)$, $Q_z(L=0,Y_{n+1}=0)$, and $\frac{Q_z(L=0,Y_{n+1}=0)}{Q_z(L=0)}$ all
strictly decrease with $z$ for $z\in\{0,1,\dots,n+1\}$.
Therefore, when $\delta=\delta_z$,
the above minimum is realized with
\begin{align}
	p_j= \left\{
	\begin{array}{ll}
		1 & \quad\hbox{ when } j=z,  \\
		0 & \quad\hbox{ otherwise. }
	\end{array}
	\right.
\end{align}
This fact together with \eref{E13} confirms \eref{XMP3}, and completes the proof of \eref{XMY}.

\section{Proof of \pref{prop:epsiidMonoton}  \main}\Label{PFepsiidMonoton}
According to \eref{XMR} \main we have
\begin{align}
	\overline{\epsilon}_{\iid}(k,n,\delta)
	=
	\max \{0\leq\epsilon\leq1 \,| B_{n,k}(\epsilon)\geq \delta \},
\end{align}
where $B_{n,k}(\epsilon)$ is defined in \eref{eq:binomCFD} \main.

\noindent
{\bf Monotonicity for $\delta$:}
The monotonicity of $\overline{\epsilon}_{\iid}(k,n,\delta)$ with $\delta$ follows immediately from the definition \eref{eq:epsiidDef} \main.

\noindent
{\bf Monotonicity for $n$:} We have
\begin{align}
	\begin{split}
		\overline{\epsilon}_{\iid}(k,n,\delta)
		&=
		\max \{0\leq\epsilon\leq1 \,| B_{n,k}(\epsilon)\geq \delta \} \\
		&\geq
		\max \{0\leq\epsilon\leq1 \,| B_{n+1,k}(\epsilon)\geq \delta \}
		=	
		\overline{\epsilon}_{\iid}(k,n+1,\delta),
	\end{split}
\end{align}
where the inequality follows because $B_{n,k}(\epsilon)\geq B_{n+1,k}(\epsilon)$ (according to \lref{lem:Bzkmono}).
Therefore, $\overline{\epsilon}_{\iid}(k,n,\delta)$ is
nonincreasing in $n$ for $n\geq k+1$.

\noindent
{\bf Monotonicity for $k$:} We have
\begin{align}
	\begin{split}
		\overline{\epsilon}_{\iid}(k,n,\delta)
		&=
		\max \{0\leq\epsilon\leq1 \,| B_{n,k}(\epsilon)\geq \delta \} \\
		&\leq
		\max \{0\leq\epsilon\leq1 \,| B_{n,k+1}(\epsilon)\geq \delta \}
		=	
		\overline{\epsilon}_{\iid}(k+1,n,\delta),
	\end{split}
\end{align}
where the inequality follows because $B_{n,k}(\epsilon)\leq B_{n,k+1}(\epsilon)$ (according to \lref{lem:Bzkmono}).
Therefore, $\overline{\epsilon}_{\iid}(k,n,\delta)$ is
nondecreasing in $k$ for $0\leq k\leq n-1$.

\section{Proof of Lemma \ref{LLP} \main}\Label{PF-Th68}
\begin{proof}[Proof of \eqref{XMA6T} \main]
	First, we show \eqref{XMA6T} \main by showing the relation
	\begin{align}
		\max\{0\leq\epsilon\leq1 \,| B_{n,k}(\epsilon)\geq 1/2\}
		>   \frac{k}{n} \Label{eLLO}
	\end{align}
	because \eqref{XMR} \main guarantees the equation
	$\max \{0\leq\epsilon\leq1 \,| B_{n,k}(\epsilon)\geq \delta \}
	=\overline{\epsilon}_{\iid}(k,n,\delta)$
	and
	the relation $\max\{0\leq\epsilon\leq1 \,| B_{n,k}(\epsilon)\geq \delta \}
	\ge
	\max\{0\leq\epsilon\leq1 \,| B_{n,k}(\epsilon)\geq \frac{1}{2} \}$
	holds for
	$0 < \delta \le \frac{1}{2}$.
	For this aim, we shall consider two cases depending on the value of $k$.
	
	\begin{enumerate}
		\item[1)] $k=0$.
		In this case we have
		\begin{align}
			\max\{0\leq\epsilon\leq1 \,| B_{n,k}(\epsilon)\geq 1/2\}
			=\max\{0\leq\epsilon\leq1 \,| (1-\epsilon)^n\geq 1/2\}
			>0 ,
		\end{align}
		which confirms \eref{eLLO}.
		
		\item[2)] $k\in\bbZ^{\geq 1}$.
		In this case we have $B_{n,k}(k/n)> 1/2$ according to
		\coref{cor:Bnk1-k/n>1/2}.
		Since $B_{n,k}(\epsilon)$ is continuous in $\epsilon$,
		there exists a
		positive number $\epsilon_0> \frac{k}{n}$ such that $B_{n,k}(\epsilon_0) \ge 1/2$.
		Hence, we find that
		\begin{align}
			\max\{0\leq\epsilon\leq1 \,| B_{n,k}(\epsilon)\geq 1/2\}
			\ge \epsilon_0
			>   \frac{k}{n} ,
		\end{align}
		which implies \eref{eLLO}.
	\end{enumerate}
	In conclusion, \eref{eLLO} holds for all $k\in\bbZ^{\geq 0}$.
\end{proof}
\begin{proof}[Proof of \eqref{XMA6C} \main]
	We show \eqref{XMA6C} \main by using \eqref{XMA6T} \main.
	When $0 < \delta \leq 1/2$ we have
	\begin{align}\Label{CAY}
		\begin{split}
			\bar{\epsilon}_{\lambda}(l, n, \delta)
			\geq \bar{\epsilon}_{\lambda}(l, n, 1/2)
			&\stackrel{(a)}{\geq}
			\max _{Q \in \mathcal{Q}_{\iid,n+1}}\left\{Q\left(Y_{n+1}=1 | L \leq l\right) |\, Q(L \leq l) \geq 1/2\right\} \\
			& \stackrel{(b)}{=} \max\{0\leq\epsilon\leq1 \,| B_{n,l}(\nu\epsilon)\geq1/2\} .
		\end{split}
	\end{align}
	Here $(a)$ follows because the optimization constraints in the RHS
	are
	stronger than those in $\bar{\epsilon}_{\lambda}(l, n, 1/2)$;
	$(b)$ follows from the fact that $\mathcal{Q}_{\iid,n+1}=\{P_{\theta}^{n+1}\}_{\theta \in [0,1]}$.
	
	In the following, we consider two cases depending on the relation between $\nu n$ and $l$.
	First, when $\nu n\leq l$, we have
	\begin{align}
		B_{n,l}(\nu) \stackrel{(a)}{\geq} B_{n,l}(l/n) \stackrel{(b)}{\geq} 1/2,
	\end{align}
	where $(a)$ follows from $l/n\geq \nu$ and \lref{lem:Bzkmono}; and  $(b)$ follows from \lref{lem:BzklamOneHalf}.
	This together with \eref{CAY} imply that $\bar{\epsilon}_{\lambda}(l, n, 1/2)=1$, which confirms \eqref{XMA6C} \main.

	Second, when $\nu n> l$, we have
	\begin{align}
		\text{RHS of \eqref{CAY}} & = \frac{1}{\nu}\max\{0\leq\epsilon\leq \nu \,| B_{n,l}(\epsilon)\geq1/2\}  \stackrel{(a)}{>}    \frac{l}{\nu n},
	\end{align}
	which confirms \eqref{XMA6C} \main again. Here
	$(a)$ follows from \eqref{XMA6T} \main, i.e., \eqref{eLLO}.
\end{proof}

\begin{proof}[Proof of \eqref{XMA6J} \main]
	When $\frac{k}{n}< \delta < \frac{k+1}{n+1}$, \eref{XMP9} \main implies \eqref{XMA6J} \main.
	When $\frac{k+1}{n+1}\leq \delta\leq 1/2$, we have
	\begin{align*}
		\overline{\epsilon}_0(k,n,\delta)
		&=
		\frac{(k+1) (n+1-k) -\delta(n+1)}{\delta(n-k)(n+1)}
		\geq
		\frac{(k+1) (n+1-k) -\frac{1}{2}(n+1)}{\delta(n-k)(n+1)}  \\
		&=
		\frac{k(n-k) +\frac{1}{2}(n+1)}{\delta(n-k)(n+1)}
		>
		\frac{k }{\delta n} ,
	\end{align*}
	which confirms \eqref{XMA6J} \main again. Here the last inequality follows because the following conditions are equivalent,
	\begin{align}
		\frac{k(n-k) +\frac{1}{2}(n+1)}{\delta(n-k)(n+1)}
		>\frac{k }{\delta n}
		\quad&\Leftrightarrow\quad
		nk(n-k) +\frac{n}{2}(n+1) > k(n-k)(n+1)
		\nonumber\\
		\quad&\Leftrightarrow\quad
		\frac{n}{2}(n+1) > k(n-k), \label{eq:11condition}
	\end{align}
	and that \eref{eq:11condition} holds for $\frac{k+1}{n+1}\leq 1/2$.
\end{proof}

\section{Proof of Proposition \ref{P4-3} \main}\Label{PP4-3}
Remember that
${\cal Q}_{{\rm iid},n+1} =\{P_{\theta}^{n+1}\}_{\theta \in [0,1]}$.
The central limit theorem guarantees that
\begin{align}
\lim_{n \to \infty}P_{s +\frac{\tau}{\sqrt{n}}}^{n+1}(K \le \lceil s n\rceil)=
\Phi(-\frac{\tau}{\sqrt{s(1-s)}}).
\end{align}
Hence, the condition
$\lim_{n \to \infty}P_{s +\frac{\tau}{\sqrt{n}}}^{n+1}(K \le \lceil s n\rceil) \le \delta$
is equivalent to the condition $\tau \le - \Phi^{-1}(\delta) \sqrt{s(1-s) } $.
The above relation shows Item (a) for $\overline{\epsilon}_{{\rm iid}}(\lceil s n\rceil,n,\delta)$.

We have
\begin{align}
\lim_{n \to \infty}\frac{-1}{n}\ln P_{\tau}^{n+1}(K \le \lceil s n\rceil)
=D( s\|\tau).
\end{align}
Solving the equation $D( s\|\tau)=r $ for $\tau$, we have $\tau=
\epsilon_D(s,r)$.
The above relation shows Item (b) for
$\overline{\epsilon}_{{\rm iid}}(\lceil s n\rceil,n,\delta)$.

\section{Proof of Lemma \ref{L-XMLA} \main}	\Label{P-XMLA}
First, we notice that the derivative of $\kappa(\lambda)$ is calculated as
\begin{align*}
\begin{split}
&\kappa'(\lambda) 
=\frac{((1-\theta_0)\alpha -(1+\theta_0))\beta \lambda^{-1/2}-(1-\theta_0)\beta\lambda^{-3/2}
-\gamma (1-\lambda)^{-1/2}
}{2(1-\theta_0(1-\lambda))^{3/2}} \\
&=\frac{\Phi^{-1}(\delta)(1-\theta_0)
\Big(
(1 +(1+\theta_0)\psi(\delta) )\lambda^{-1/2}+(1-\theta_0)\psi(\delta)\lambda^{-3/2}\Big)
-\frac{e-t}{\sqrt{\theta_0}}(1-\lambda)^{-1/2}
}{2(1-\theta_0(1-\lambda))^{3/2}} ,
\end{split}
\end{align*}
which confirms the equivalence between $\kappa'(\lambda)=0$ and 
Eq. \eqref{XMLA}\main. 
We notice that
\eref{XMLA} \main 
is equivalent to the cubic equation  $f_1(\lambda)+f_2(\lambda)=0$, where
		the functions $f_1(x)$ and $f_2(x)$ are defined as follows,
		\begin{align}
			f_1(x)&:=\frac{(e-t)^2}{\theta(\Phi^{-1}(\delta)(1-\theta_0))^2} x^3, \\
			f_2(x)&:=(x-1)
			\Big[(1 +(1+\theta_0)\psi(\delta) )x+(1-\theta_0)\psi(\delta)\Big]^2.
		\end{align}
		The cubic equation $f_2(x)=0$ has two real roots $x_0=\frac{-(1-\theta_0)\psi(\delta)}{1 +(1+\theta_0)\psi(\delta)}$ and $x_1=1$, where $x_0$ is a double root.
		
		The remaining proof is composed of three steps.
		The results of Step 1 and Step 3 will together confirm the fact that the equation \eqref{XMLA} \main has but only one solution in $(0,1)$.
		
		\noindent \textbf{Step 1:} The aim of this step is to show that the equation $f_1(x)+f_2(x)=0$ has at least one solution in $(0,1)$.
		
		We have
		\begin{align}
			f_1(0)+f_2(0)&=-\Big[(1-\theta_0)\psi(\delta)\Big]^2<0, \Label{eq:f1+f2<0} \\
			f_1(1)+f_2(1)&=\frac{(e-t)^2}{\theta(\Phi^{-1}(\delta)(1-\theta_0))^2}>0.
		\end{align}
		Since $f_1(x)+f_2(x)$ is continuous in $x$, the above inequalities implies that the equation $f_1(x)+f_2(x)=0$ has at least one solution in $(0,1)$.
		
		\noindent \textbf{Step 2:} The aim of this step is to show the double root $x_0<0$.
		
		First note that the numerator $-(1-\theta_0)\psi(\delta)>0$ for $\delta<1/2$.
		Second, the denominator  satisfies
		\begin{align}
			1 +(1+\theta_0)\psi(\delta)
			\stackrel{(a)}{<} 1 +\psi(\delta)
			=1+\frac{\phi(\Phi^{-1}(\delta)) }{\delta  \Phi^{-1}(\delta)}
			\stackrel{(b)}{=}
			1-\frac{\phi(y) }{y\Phi(-y)}
			\stackrel{(c)}{\leq} 0,
		\end{align}
		which confirms that $x_0<0$.
		Here $(a)$ follows because $\psi(\delta)<0$ when $\delta<1/2$;
		in $(b)$ we change the variable by letting $y=-\Phi^{-1}(\delta)$; and $(c)$ follows because $\Phi(-y)\leq\phi(y)/y$ for $y>0$, witch was first proved by Gordon \cite{Gordon41}.
		
		\noindent \textbf{Step 3:} The aim of this step is to show that the equation $f_1(x)+f_2(x)=0$ has at most one solution in $(0,\infty)$.
		
		Since the double root $x_0<0$, there are only two cases for the monotonicity of $f_2(x)$ in $(0,\infty)$:
		
		\textbf{Case A:} $f_2(x)$ increases monotonically with $x$ for $x\in(0,\infty)$. In this case, $f_1(x)+f_2(x)$ is also nondecreasing in $x$ for $x\in(0,\infty)$. Hence, the equation $f_1(x)+f_2(x)=0$ has at most one solution in $(0,\infty)$.
		
		\textbf{Case B:} In $(0,\infty)$,  $f_2(x)$ first monotonically decreases and then monotonically increases.
		In this case, in the interval $(0,\infty)$, the cubic function $f_1(x)+f_2(x)$ also first monotonically decreases and then monotonically increases.
		This fact together with \eref{eq:f1+f2<0} implies that
		the equation $f_1(x)+f_2(x)=0$ has at most one solution in $(0,\infty)$.

\section{Proof of Proposition \ref{P3-2} \main}\Label{PP3-2}
Remember that
${\cal Q}_{{\rm iid},n+1} =\{P_{\theta}^{n+1}\}_{\theta \in [0,1]}$.
The law of small number guarantees \cite{PP} that
\begin{align}
\lim_{n \to \infty}P_{\frac{\tau}{n}}^{n+1}(K \le k_0)=
\mathrm{Pois}(k_0,\tau).
\end{align}
Hence, the condition
$\mathrm{Pois}(k_0,\tau)=\lim_{n \to \infty}P_{\frac{\tau}{n}}^{n+1}(K \le k_0) \le \delta$
is equivalent to the condition $\tau \le t_\mathrm{P}(k_0,\delta)$.
The above relation shows Item (a) for $\overline{\epsilon}_{{\rm iid}}(k_0,n,\delta)$.

The large deviation principle \cite{DembZ10} guarantees that
\begin{align}
\lim_{n \to \infty}\frac{-1}{n}\ln P_{\tau}^{n+1}(K \le k_0)=D( 0\|\tau)
=-\ln (1-\tau).
\end{align}
Solving the equation $-\ln (1-\tau)=r $ for $\tau$, we have $\tau=1-\rme^{-r}$.
The above relation shows Item (b) for $\overline{\epsilon}_{{\rm iid}}(k_0,n,\delta)$.

\section{Proof of \lref{MLT} \main}\Label{PF-MLT}
\begin{proof}[Proof of \lref{MLT} \main]
	First, by \eqref{eq:NsmalldeltaT} \main we have
	\begin{align}\Label{eq:Nklimepsdel}
		\lim_{\epsilon \to 0}\lim_{\delta\to 0} \frac{\epsilon N_{\lambda,{\rm c}}(\lceil \nu k_0\rceil,\epsilon,\delta)}{\ln \delta^{-1}}
		=\frac{1}{\lambda \ln \lambda^{-1}}.
	\end{align}

	In the following, we shall prove that
	\begin{align}\Label{eq:Nkhighpresion2}
		\lim_{\delta\rightarrow 0}\,\lim_{\epsilon\rightarrow 0} \frac{\epsilon N_{\lambda,{\rm c}}(\lceil \nu k_0\rceil,\epsilon,\delta)}{\ln\delta^{-1}}
		=\frac{1}{\lambda \ln\lambda^{-1}},
	\end{align}
	which together with \eref{eq:Nklimepsdel} confirm the lemma.
	
	For convenience, we use $l=\lceil \nu k_0\rceil$.
	By \lref{lem:Nk} below, when $\delta\leq1/2$ we have
	\begin{align}\label{eq:boundepsN/lndel}
		\frac{(1-\epsilon)(z_*-l)}{\lambda\ln\delta^{-1}}+\frac{\epsilon(z_*-1)}{\ln\delta^{-1}}
		\leq \frac{\epsilon N_{\lambda,{\rm c}}(l,\epsilon,\delta)}{\ln\delta^{-1}}
		\leq \frac{z^*-l+1+\sqrt{\lambda l}}{\lambda\ln\delta^{-1}},
	\end{align}
	where the upper bound follows from \eref{eq:NLB2}, and the lower bound follows from \eref{eq:NkLB1}.
	
	Then \lref{lem:z-kwithdelta} implies that
	\begin{align}
		&\lim_{\delta\to 0}\,\lim_{\epsilon\to 0}\, \text {LHS of \eqref{eq:boundepsN/lndel}}
		= \lim_{\delta\to 0} \frac{z^*-l-1}{\lambda \ln\delta^{-1}}
		= \frac{1}{\lambda \ln\lambda^{-1}},
		\\
		&\lim_{\delta\to 0}\,\lim_{\epsilon\to 0}\, \text {RHS of \eqref{eq:boundepsN/lndel}}
		= \lim_{\delta\to 0} \frac{z^*-l+1+\sqrt{\lambda l}}{\lambda \ln\delta^{-1}}
		= \frac{1}{\lambda \ln\lambda^{-1}},
	\end{align}
	which confirm \eref{eq:Nkhighpresion2} and completes the proof.
\end{proof}

\begin{lemma}\Label{lem:Nk}
	Suppose $0<\lambda,\epsilon,\delta< 1$ and $l\in\bbZ^{\geq 0}$. Then
	\begin{align}
		N_{\lambda,{\rm c}}(l,\epsilon,\delta)
		&\geq \Bigl(\frac{1}{\epsilon}-1\Bigr) \frac{z_*B_{z_*-1,l}}{B_{z_*,l}} +z_*-1 
		\geq \Bigl(\frac{1}{\epsilon}-1\Bigr) \frac{z_*-l}{\lambda} +z_*-1. \Label{eq:NkLB1}
	\end{align}	
	If $\delta\leq1/2$ in addition, then
	\begin{align}
		N_{\lambda,{\rm c}}(l,\epsilon,\delta)
		&\leq \biggl\lceil\frac{z^*-l+1+\sqrt{\lambda l}}{\lambda}\Bigl(\frac{1}{\epsilon}-1\Bigr)+z^*\biggr\rceil
		< \frac{z^*-l+1+\sqrt{\lambda l}}{\lambda \epsilon}.\Label{eq:NLB2}
	\end{align}	
\end{lemma}

\begin{proof}[Proof of \lref{lem:Nk}]
	The proof is composed of two steps.
	
	\noindent \textbf{Step 1:} Proof of \eqref{eq:NkLB1}.
	When $0<\delta<1$ we have
	\begin{align}
		\begin{split}
			&N_{\lambda,{\rm c}}(l,\epsilon,\delta)
			= \min\{ n \geq l+1 \,|\, \overline{\epsilon}_\lambda^{\rm c}(l,n,\delta) \geq 1-\epsilon \} \\
			&\geq  \min\bigg\{ n \,\bigg| \max \left\{0, \frac{g_{z_*}(l,n,\lambda)}{h_{z_*}(l,n,\lambda)}\right\} \geq 1-\epsilon \bigg\}\\
			&=\biggl\lceil \frac{z_*B_{z_*-1,l}}{B_{z_*,l}}\Bigl(\frac{1}{\epsilon}-1\Bigr)+z_*-1\biggr\rceil
			\geq \frac{z_*-l}{\lambda}\Bigl(\frac{1}{\epsilon}-1\Bigr)+z_*-1,
		\end{split}
	\end{align}
	which confirms \eqref{eq:NkLB1}.
	Here the first inequality follows from \coref{cor:zeta2Bound1} in the main text, and the second equality follows from \lref{lem:Bzkz}. 

	\noindent \textbf{Step 2:} Proof of \eref{eq:NLB2}.
	When $0<\delta\leq1/2$ we have
	\begin{align}
		\begin{split}
			&N_{\lambda,{\rm c}}(l,\epsilon,\delta)
			= \min\{ n \geq l+1 \,|\, \overline{\epsilon}_\lambda^{\rm c}(l,n,\delta) \geq 1-\epsilon \} \\
			&\leq  \min\bigg\{ n \geq l+1 \,\bigg|\, \frac{g_{z^*+1}(l,n,\lambda)}{h_{z^*+1}(l,n,\lambda)} \geq 1-\epsilon \bigg\}
			=  \biggl\lceil\frac{(z^*+1)B_{z^*,l}}{B_{z^*+1,l}}\Bigl(\frac{1}{\epsilon}-1\Bigr)+z^*\biggr\rceil \\
			& \leq \biggl\lceil\frac{z^*-l+1+\sqrt{\lambda l}}{\lambda}\Bigl(\frac{1}{\epsilon}-1\Bigr)+z^*\biggr\rceil
			< \frac{z^*-l+1+\sqrt{\lambda l}}{\lambda \epsilon},
		\end{split}
	\end{align}
	which confirms \eref{eq:NLB2}.
	Here the first inequality follows from \coref{cor:zeta2Bound1} in the main text, and
	the second inequality follows from \lref{lem:Bzkz} and the fact that $z^*\geq l/\nu$ when $0<\delta\leq1/2$ (see \coref{cor:z*k/nu}).
\end{proof}

\section{Proofs of Equations~(\ref{eq:hzHomo}) and (\ref{eq:gzHomo}) \main}\Label{app:ProofhzgzHomo}
For $z\in\{0,1,\dots,n+1\}$, we have
\begin{align}
	Q(K=z,Y_{n+1}=0|Z=z) &= \frac{n -z+1}{n+1}, \\
	Q(K=z-1,Y_{n+1}=1|Z=z) &= \frac{z}{n+1}.
\end{align}
In the following, we shall show Eqs.~\eqref{eq:hzHomo} and \eqref{eq:gzHomo} \main, respectively.

\noindent \textbf{Step 1:} Proof of \eref{eq:hzHomo} \main.
For $0\leq v\leq z-1$, we have
\begin{align}
	Q(L=v|Z=z) =\,& Q(K=z,Y_{n+1}=0|Z=z)  {z \choose v} (1-\lambda)^v \lambda^{z-v} \nonumber \\
	&+Q(K=z-1,Y_{n+1}=1|Z=z)  {z-1 \choose v} (1-\lambda)^v \lambda^{z-1-v} \nonumber \\
	=\,& \frac{n -z+1}{n+1}  b_{z,v}(\nu)
	+\frac{z}{n+1}  b_{z-1,v}(\nu),
\end{align}
where $b_{z-1,v}(\nu)$ is defined in \eref{eq:binomCFD} \main.
Hence, for $0\leq l \leq z-1$, we have
\begin{align}
	Q(L\leq l |Z=z) &= \sum_{v=0}^l Q(L=v|Z=z)
	= \frac{n -z+1}{n+1}  \sum_{v=0}^l b_{z,v}(\nu)
	+\frac{z}{n+1}  \sum_{v=0}^l b_{z-1,v}(\nu) \nonumber \\
	&= \frac{n -z+1}{n+1}B_{z,l} +\frac{z}{n+1} B_{z-1,l}.
\end{align}
In addition, when $l \geq z$, it is obvious that $P(L \le l|Z=z)=1$.
This completes the proof of \eref{eq:hzHomo} \main.

\noindent \textbf{Step 2:} Proof of \eref{eq:gzHomo} \main.
For $0\leq v\leq z-1$, we have
\begin{align}
	Q(L=v,Y_{n+1}=0|Z=z)
	&= Q(K=z,Y_{n+1}=0|Z=z)  {z \choose v} (1-\lambda)^v \lambda^{z-v} \nonumber \\
	&=\frac{n -z+1}{n+1}  b_{z,v}(\nu),
\end{align}
Hence, when $0\leq l \leq z-1$, we have
\begin{align}
	Q(L\leq l ,Y_{n+1}=0 |Z=z)
	&= \sum_{v=0}^l Q(L=v,Y_{n+1}=0|Z=z)  \nonumber \\
	&= \frac{n -z+1}{n+1}  \sum_{v=0}^l b_{z,v}(\nu)
	= \frac{n -z+1}{n+1} B_{z,l}.
\end{align}
In addition, when $l \geq z$, we have
\begin{align}
	Q(L\leq l ,Y_{n+1}=0 |Z=z)
	=	Q(K=z,Y_{n+1}=0|Z=z)
	= \frac{n -z+1}{n+1}.
\end{align}
This completes the proof of \eref{eq:gzHomo} \main.

\section{Proofs of \lrefs{lem:gzhzMono}{lem:SlopeMono} \main}\Label{app:ProofLemmaZeta2}
\begin{proof}[Proof of \lref{lem:gzhzMono}]
	According to \lref{lem:Bzkmono}, $B_{z,l}$ decreases monotonically in $z$ for $z\geq l$.
	
	For $z\in\{l+1,\dots,n\}$, we have
	\begin{align}
		h_n(l,n,\lambda)=B_{n,l}
		\leq B_{z,l}< h_z(l,n,\lambda)
		< B_{z-1,l}\leq B_{l,l}= h_l(l,n,\lambda),
	\end{align}
	where the second and third inequality follows because
	\begin{align}
		h_z(l,n,\lambda)= B_{z,l} + \frac{z}{n+1} (B_{z-1,l}-B_{z,l}).
	\end{align}
	Hence $h_z(l,n,\lambda)$ decreases monotonically in $z$ for $l\leq z\leq n+1$.
	
	The monotonicity of $g_z(l,n,\lambda)$ with $z$ follows from that
	$\frac{n-z+1}{n+1}$ decreases monotonically in $z$ for $0\leq z\leq n+1$, and that
	$B_{z,l}$ decreases monotonically  in $z$ for $l\leq z\leq n+1$.
	
	By \eref{eq:gzHomo} \main,
	$g_z(l,n,\lambda)\geq0$ for $z\leq n+1$, and
	$g_z(l,n,\lambda)<0$ for $z\geq n+2$.
	In addition, By \eref{eq:hzHomo} \main we have
	\begin{align}
		h_z(l,n,\lambda)=B_{z,l}+
		\frac{z }{n+1}(B_{z-1,l}-B_{z,l})>0,
	\end{align}
	which completes the proof.
\end{proof}

\begin{proof}[Proof of \lref{lem:SlopeMono} \main]
	Due to \lref{lem:gzhzMono} \main, it suffices to prove that $[h_{z+1}(l,n,\lambda)-h_z(l,n,\lambda)]/[g_{z+1}(l,n,\lambda)-g_z(l,n,\lambda)]$ is strictly increasing in $z$ for $l\leq z\leq n$. 	
	We have
	\begin{align}
		(n+1)[h_{z+1}(l,n,\lambda)-h_z(l,n,\lambda)]&=(n-z)B_{z+1,l}+(2z-n)B_{z,l}-zB_{z-1,l},\\
		(n+1)[g_{z+1}(l,n,\lambda)-g_z(l,n,\lambda)]&=(n-z)B_{z+1,l}-(n-z+1)B_{z,l},
	\end{align}
	which imply that
	\begin{align}
		\begin{split}
			&\frac{h_{z+1}(l,n,\lambda)-h_z(l,n,\lambda)}{g_{z+1}(l,n,\lambda)-g_z(l,n,\lambda)}-1
			=\frac{z(B_{z-1,l} -B_{z,l})-B_{z,l}}{(n-z)(B_{z,l} -B_{z+1,l}) +B_{z,l}}\\
			&=\frac{z\Delta_{z-1,l}-B_{z,l}}{(n-z)\Delta_{z,l} +B_{z,l}}
			=\frac{z\Delta_{z-1,l}}{(n-z)\Delta_{z,l} +B_{z,l}}\biggl(1-\frac{B_{z,l}}{z\Delta_{z-1,l}}\biggr).
		\end{split}
	\end{align}
	Note that
	\begin{equation}
		\frac{\Delta_{z-1,l}}{B_{z,l}}=\frac{B_{z-1,l}}{B_{z,l}}-1
	\end{equation}
	is nondecreasing in $z$ according to \lref{lem:Bzkz}, so  the function $\bigl(1-\frac{B_{z,l}}{z\Delta_{z-1,l}}\bigr)$ is strictly increasing in $z$.
	In addition, $\Delta_{z-1,l}/\Delta_{z,l}$ is nondecreasing in $z$ according to
	\lref{lem:Deltazk},
	so  the function $z\Delta_{z-1,l}/[(n-z)\Delta_{z,l} +B_{z,l}]$ is strictly increasing in $z$, because
	\begin{equation}
		\frac{(n-z)\Delta_{z,l} +B_{z,l}}{z\Delta_{z-1,l}}
		=\left(\frac{n}{z}-1\right)\frac{\Delta_{z,l}}{\Delta_{z-1,l}}+\frac{B_{z,l}}{z\Delta_{z-1,l}}
	\end{equation}
	is strictly decreasing in $z$.
	Therefore, $[h_{z+1}(l,n,\lambda)-h_z(l,n,\lambda)]/[g_{z+1}(l,n,\lambda)-g_z(l,n,\lambda)]$ is strictly increasing in $z$ for $l\leq z\leq n$, which completes the proof.
\end{proof}

\section{Proof of \pref{lem-2-2} \main}\Label{PF-lem-2-2}
We show \pref{lem-2-2} \main by showing the following lemma.
\begin{lemma}\label{lem:zeta2monoto}
	Suppose $0\leq\lambda<1$, $0<\delta\leq1$, $l\in\bbZ^{\geq 0}$, and $n\in\bbZ^{\geq l+1}$.
	Then $\overline{\epsilon}_\lambda^{\rm c}(l,n,\delta)$ is
	nondecreasing in $\delta$ for $0<\delta\leq1$,
	nondecreasing in $n$ for $n\in\bbZ^{\geq l+1}$, and
	nonincreasing in $l$ for $0\leq l\leq n-1$.
\end{lemma}

\noindent
{\bf Monotonicity for $\delta$:}\quad
The monotonicity of $\overline{\epsilon}_\lambda^{\rm c}(l,n,\delta)$ with $\delta$ follows immediately from the formula \eqref{AYU} \main for $\overline{\epsilon}_\lambda^{\rm c}$.

\noindent
{\bf Monotonicity for $n$:}\quad
To prove the monotonicity of $\overline{\epsilon}_\lambda^{\rm c}(l,n,\delta)$ with $n$,
we consider two cases depending on the value of $\lambda$.
First, when $\lambda=0$, the monotonicity follows from \eqsref{eq:zeta1(lambda=0)}{eq:zeta2(lambda=0)} \main.
Second, when $\lambda>0$, we show the monotonicity by showing that the lower boundary of ${\cal R}_{n,l,\lambda}$ is not higher than that of ${\cal R}_{n+1,l,\lambda}$ for $n\geq l+1$. This fact can be shown as follows.

Since $g_z(l,n,\lambda)$ and $h_z(l,n,\lambda)$ are defined in \eqsref{eq:gzHomo}{eq:hzHomo} \main, respectively, we have the relations;
\begin{align}
	g_z(l,n,\lambda) \leq g_z(l,n+1,\lambda), \qquad
	h_z(l,n,\lambda) \geq h_z(l,n+1,\lambda)
\end{align}
for $z=0,1,\dots,n+1$.
Let $\tau_{n,l}(x) $ be the graph to show the lower boundary of ${\cal R}_{n,l,\lambda}$.
For $x \le h_{n+1}(l,n,\lambda)$, we define
$\tau_{n,l}(x) $ to be $0$.
We denote the line $\kappa[A_1,A_2]$ to pass through two points $A_1$ and $A_2$.
We denote the point $(h_{z}(l,n,\lambda),g_{z}(l,n,\lambda))$
by $A_z(l,n)$.
Since Lemma \ref{lem:gzhzMono} \main guarantees that $h_z(l,n,\lambda)$
is monotonically decreasing for $z$, for $x \le h_{z}(l,n,\lambda)$,
there exists an integer $z'\ge z+1$ such that $x \in [h_{z'}(l,n,\lambda), h_{z'-1}(l,n,\lambda)]$.
Then, we have
\begin{align}
	\tau_{n,l}(x) \le
	\kappa[A_{z}(l,n),A_{z'}(l,n)](x).
	\Label{aa}
\end{align}
Since Lemma \ref{lem:gzhzMono} \main guarantees that  $g_z(l,n,\lambda)$
is monotonically decreasing for $z$,
we have
\begin{align}
	g_z(l,n+1,\lambda) \ge g_z(l,n,\lambda)  \ge g_{z'}(l,n,\lambda) .
	\Label{ab}
\end{align}
Combining \eqref{aa} and \eqref{ab}, we have
\begin{align}
	\tau_{n,l}(h_{z}(l,n+1,\lambda)) \le g_z(l,n+1,\lambda).
\end{align}
Therefore, the lower boundary of ${\cal R}_{n,l,\lambda}$ is not higher than that of ${\cal R}_{n+1,l,\lambda}$ for $n\geq l+1$.
(see the left plot of Fig.~\ref{fig:Monoton_nk} for an illustration),

\begin{figure}
	\begin{center}
		\includegraphics[width=6.5cm]{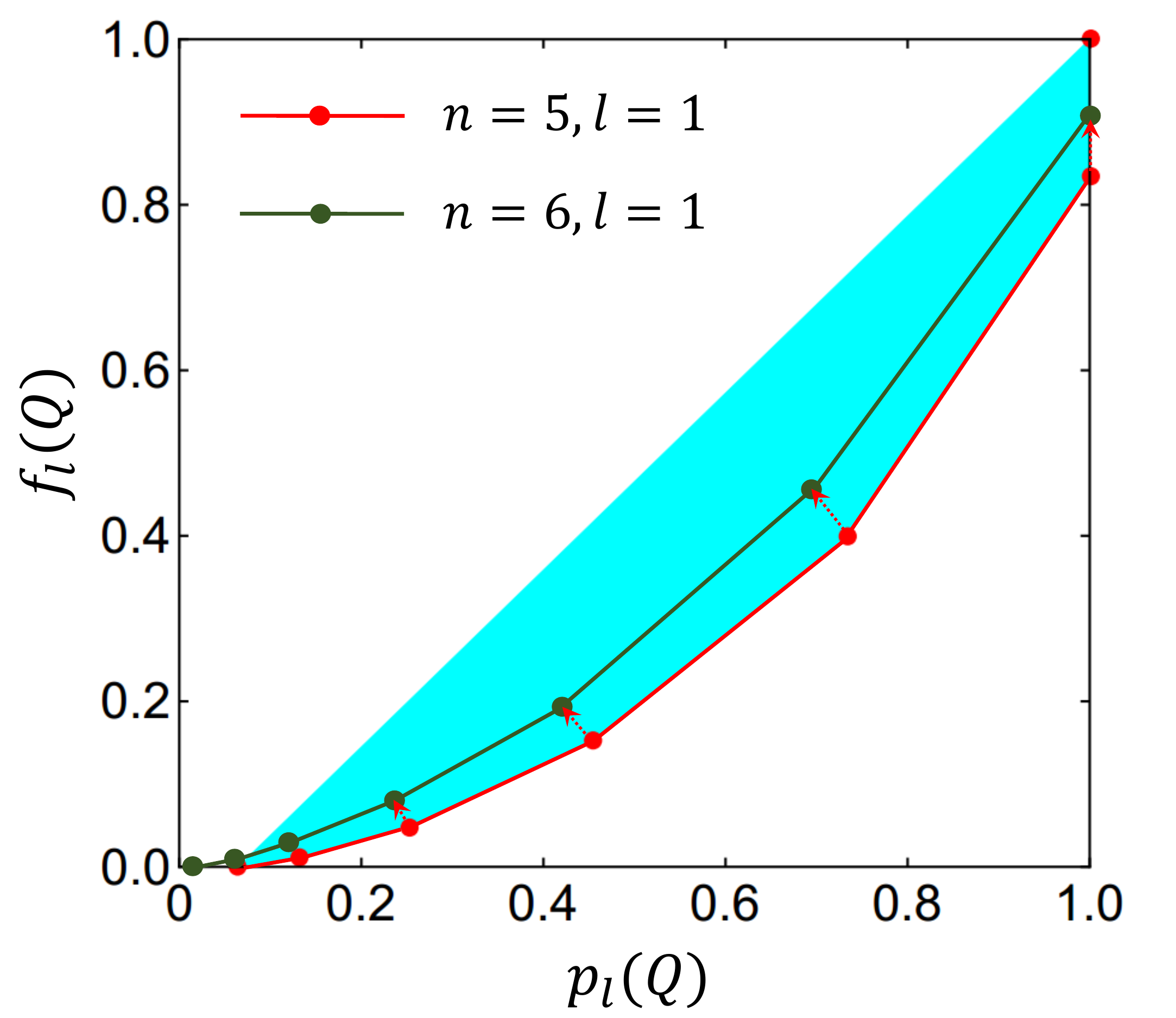}\quad
		\includegraphics[width=6.5cm]{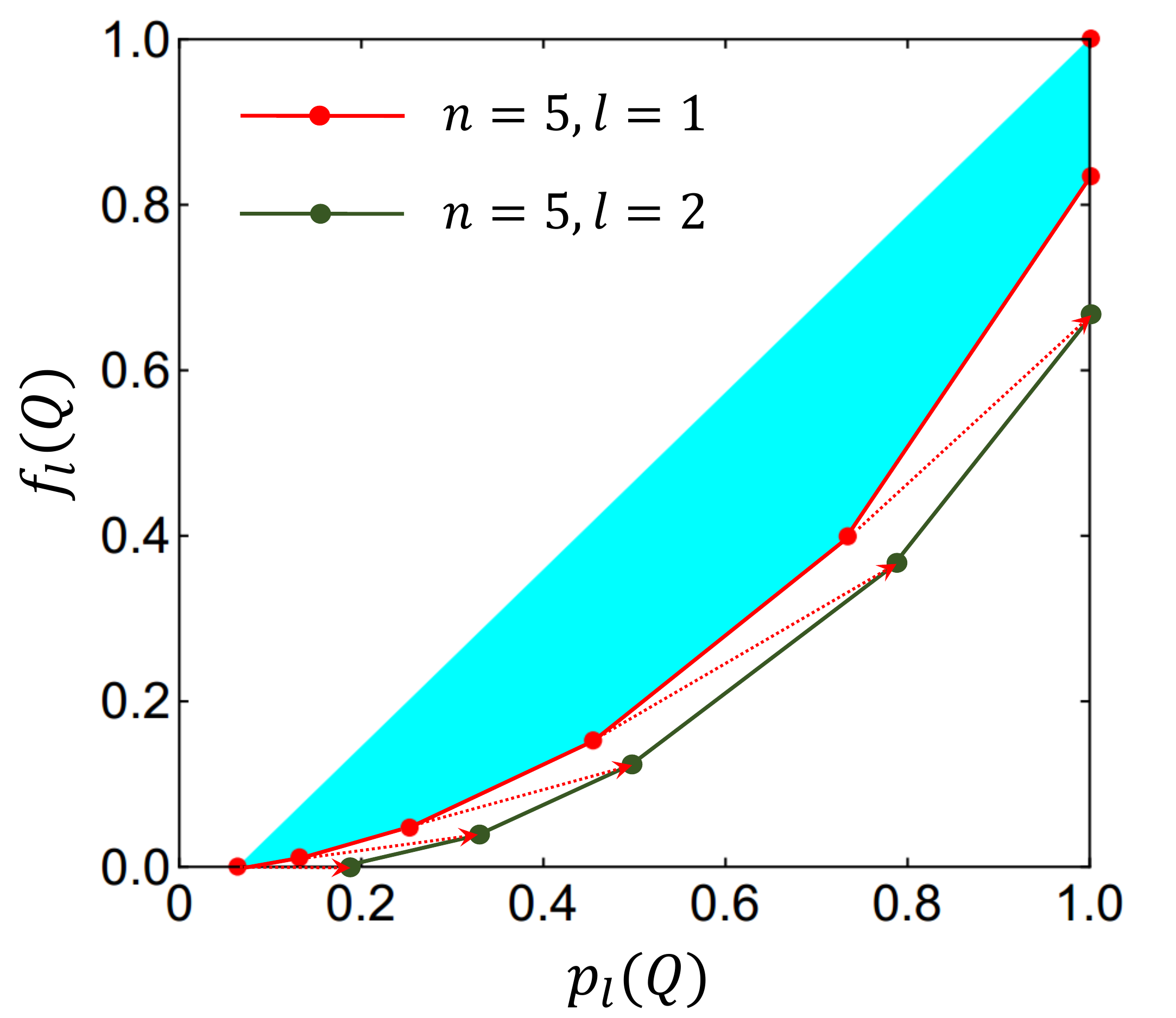}
		\caption{\label{fig:Monoton_nk}
			(a) Left plot: the lower boundary of ${\cal R}_{l,n,\lambda}$ is not higher than that of ${\cal R}_{l,n+1,\lambda}$ for $n\geq l+1$.
			Here the red solid lines are the lower boundary of ${\cal R}_{l=1,n=5,\lambda=1/\rme}$;
			the red points are the extremal points of ${\cal R}_{l=1,n=5,\lambda=1/\rme}$;
			the green solid lines are the lower boundary of ${\cal R}_{l=1,n=6,\lambda=1/\rme}$;
			the green points are the extremal points of ${\cal R}_{l=1,n=6,\lambda=1/\rme}$;
			(b) Right plot: the lower boundary of ${\cal R}_{l+1,n,\lambda}$ is not higher than that of ${\cal R}_{l,n,\lambda}$ for $0\leq l \leq n-2$.
			Here the red solid lines are the lower boundary of ${\cal R}_{l=1,n=5,\lambda=1/\rme}$;
			the red points are the extremal points of ${\cal R}_{l=1,n=5,\lambda=1/\rme}$;
			the green solid lines are the lower boundary of ${\cal R}_{l=2,n=5,\lambda=1/\rme}$;
			the green points are the extremal points of ${\cal R}_{l=2,n=5,\lambda=1/\rme}$;
		}
	\end{center}
\end{figure}

\noindent
{\bf Monotonicity for $l$:}\par
\noindent \textbf{Step 1:} \quad
To prove the monotonicity of $\overline{\epsilon}_\lambda^{\rm c}(l,n,\delta)$ with $l$, we still consider two cases depending on the value of $\lambda$.
First, when $\lambda=0$, the monotonicity follows from \eqsref{eq:zeta1(lambda=0)}{eq:zeta2(lambda=0)} \main.
Second, when $\lambda>0$, we show the monotonicity by showing that
the lower boundary of ${\cal R}_{n,l+1,\lambda}$
is not higher than that of ${\cal R}_{n,l,\lambda}$ for $0\leq l\leq n-2$.
This fact can be shown as follows.

As shown later, we have the following relations.
\begin{align}
	& g_{z_1}(l,n,\lambda) \leq g_{z_1}(l+1,n,\lambda), \quad
	h_{z_1}(l,n,\lambda) \leq h_{z_1}(l+1,n,\lambda),
	\Label{eq:ghzwithk}\\
	& h_{z+1}(l+1,n,\lambda) \geq h_z(l,n,\lambda),
	\Label{eq:hz+1k+1geqhzk}\\
	&\frac{g_{z+1}(l,n,\lambda)-g_z(l,n,\lambda)}{h_{z+1}(l,n,\lambda)-h_z(l,n,\lambda)}
	\geq
	\frac{g_{z+1}(l,n,\lambda)-g_{z+1}(l+1,n,\lambda)}{h_{z+1}(l,n,\lambda)-h_{z+1}(l+1,n,\lambda)}
	\Label{eq:slopwithk}
\end{align}
for $ z_1=0,1,\dots,n+1$, and
$z=l,l+1,\dots,n$.

\noindent \textbf{Step 2:} \quad
In the following, using \eqref{eq:ghzwithk}, \eqref{eq:hz+1k+1geqhzk}, and \eqref{eq:slopwithk}, we show the following inequality by induction;
\begin{align}
	\tau_{n,l+1}(h_{z}(l,n,\lambda)) \le g_z(l,n,\lambda).
	\Label{AA1}
\end{align}

First, we show \eqref{AA1} with $z=n+1$.
Since \eqref{eq:hz+1k+1geqhzk} guarantees that $h_{n+1}(l,n,\lambda) \le h_{n+1}(l+1,n,\lambda)$,
we have
\begin{align}
	\tau_{n,l+1}(h_{n+1}(l,n,\lambda)) =0
	\stackrel{(a)}{=} g_{n+1}(l,n,\lambda),
\end{align}
where $(a)$ follows from the definition of $g_{z}(l,n,\lambda)$.
Hence, \eqref{AA1} hold with $z=n+1$.

Next, we show \eqref{AA1} with $z=z_0$
by assuming \eqref{AA1} hold with $z=z_0+1$.
The relations \eqref{eq:ghzwithk}, \eqref{eq:hz+1k+1geqhzk}, and \eqref{eq:slopwithk} guarantee that
\begin{align}
	\kappa[A_{z_0+1}(l,n), A_{z_0+1}(l+1,n)](h_{z_0}(l,n,\lambda))
	\le
	g_{z_0}(l,n,\lambda) \Label{AML2}\\
	h_{z_0}(l,n,\lambda) \in [h_{z_0+1}(l,n,\lambda),h_{z_0+1}(l+1,n,\lambda)]
	\Label{AML3}.
\end{align}
Hence, we have
\begin{align}
	\begin{split}
		& \tau_{n,l+1} (h_{z_0}(l,n,\lambda)) \\
		\stackrel{(a)}{\le} &
		\kappa[(h_{z_0+1}(l,n,\lambda) ,\tau_{n,l+1}(h_{z_0+1}(l,n,\lambda)) )
		, A_{z_0+1}(l+1,n)](h_{z_0}(l,n,\lambda)) \\
		\stackrel{(b)}{\le} &
		\kappa[A_{z_0+1}(l,n), A_{z_0+1}(l+1,n)](h_{z_0}(l,n,\lambda)) \Label{AML}
	\end{split}
\end{align}
where
$(a)$ follows from \eqref{AML3},
and $(b)$ follows from \eqref{AA1} hold with $z=z_0+1$.
The combination of
\eqref{AML} and \eqref{AML2} implies \eqref{AA1} with $z=z_0$.

Hence, we obtain \eqref{AA1} for $z=1, \ldots, n+1$, i.e.,
the lower boundary of ${\cal R}_{n,l+1,\lambda}$
is not higher than that of ${\cal R}_{n,l,\lambda}$ for $0\leq l\leq n-2$.
(see the right plot of Fig.~\ref{fig:Monoton_nk} for an illustration).

\noindent \textbf{Step 3:} \quad
In the following, we shall prove Eqs.~\eqref{eq:ghzwithk},~\eqref{eq:hz+1k+1geqhzk}, and \eqref{eq:slopwithk} in three steps, respectively.


\noindent \textbf{Step 3-1:} Proof of \eref{eq:ghzwithk}.
\eref{eq:ghzwithk} follows from Eqs.~\eqref{eq:hzHomo}, \eqref{eq:gzHomo} \main, and \lref{lem:Bzkmono}.

\noindent \textbf{Step 3-2:} Proof of \eref{eq:hz+1k+1geqhzk}.
For $z=l$, we have $h_{z+1}(l+1,n,\lambda)=h_z(l,n,\lambda)=1$ according to \eref{eq:hzHomo} \main.
For $z=l+1,\dots,n$, we have
\begin{align}
	\begin{split}
		&(n+1)[h_{z+1}(l+1,n,\lambda)-h_z(l,n,\lambda)] \\
		=& [(n-z) B_{z+1,l+1}+(z+1) B_{z,l+1}]-[(n-z+1) B_{z,l}+z B_{z-1,l}] \\
		=& (n-z) (B_{z+1,l+1}-B_{z,l}) +z(B_{z,l+1}-B_{z-1,l})+(B_{z,l+1}-B_{z,l}) \\
		\stackrel{(a)}{=} & (n-z) [(1-\lambda)B_{z,l}+\lambda B_{z,l+1}-B_{z,l}] \\
		&+z[(1-\lambda)B_{z-1,l}+\lambda B_{z-1,l+1}-B_{z-1,l}]+(B_{z,l+1}-B_{z,l}) \\
		=& (n-z) \lambda (B_{z,l+1}-B_{z,l}) +z\lambda(B_{z-1,l+1}-B_{z-1,l})+(B_{z,l+1}-B_{z,l})
		\stackrel{(b)}{\ge} 0 ,
	\end{split}
\end{align}
Here $(a)$ follows from
the basic relation $B_{z+1,l}=(1-\lambda)B_{z,l-1}+\lambda B_{z,l}$,
and $(b)$  follows from \lref{lem:Bzkmono}.

\noindent \textbf{Step 3-3:} Proof of \eref{eq:slopwithk}.
To prove \eref{eq:slopwithk}, we shall consider three cases depending on the value of $z$.
First, for $z=l$,
we have $h_{z+1}(l+1,n,\lambda)=h_z(l,n,\lambda)=1$ and $g_{z+1}(l+1,n,\lambda)\leq g_z(l,n,\lambda)$, so \eref{eq:slopwithk} holds.
Second, for $z=n$, the LHS of \eqref{eq:slopwithk} is positive, and the RHS of \eqref{eq:slopwithk} equals 0, so \eref{eq:slopwithk} holds.

Lastly, we consider the case $z\in\{l+1,\dots,n-1\}$.
To prove \eref{eq:slopwithk} in this case, it suffices to show
\begin{align}
	\label{eq:invslopwithk}
	\frac{h_{z+1}(l,n,\lambda)-h_z(l,n,\lambda)}{g_{z+1}(l,n,\lambda)-g_z(l,n,\lambda)}-1
	\leq
	\frac{h_{z+1}(l,n,\lambda)-h_{z+1}(l+1,n,\lambda)}{g_{z+1}(l,n,\lambda)-g_{z+1}(l+1,n,\lambda)}-1
\end{align}
for $z=l,l+1,\dots,n$.
Note that
\begin{align}\label{eq:LHSF14}
	\begin{split}
		\text{LHS of \eqref{eq:invslopwithk}}
		=&\frac{z(B_{z-1,l} -B_{z,l})-B_{z,l}}{(n-z)(B_{z,l} -B_{z+1,l}) +B_{z,l}}
		\leq \frac{z(B_{z-1,l} -B_{z,l})}{(n-z)(B_{z,l} -B_{z+1,l})} \\
		=&\frac{z\Delta_{z-1,l}}{(n-z)\Delta_{z,l}}
		=\frac{z-l}{(n-z)\lambda}.
	\end{split}
\end{align}
where the last equality follows because
\begin{align}
	\frac{\Delta_{z-1,l}}{ \Delta_{z,l}}
	=\frac{\binom{z-1}{l}}{\binom{z}{l}\lambda}
	=\frac{z-l}{z\lambda}.
\end{align}
In addition,
\begin{align}\label{eq:RHSF14}
	\begin{split}
		\text{RHS of \eqref{eq:invslopwithk}}
		=&\frac{(z+1)(B_{z,l+1}-B_{z,l})}{(n-z)(B_{z+1,l+1}-B_{z+1,l})}
		=\frac{(z+1)b_{z,l+1}}{(n-z)b_{z+1,l+1}} \\
		=&\frac{z+1}{n-z} \cdot \frac{z-l}{\lambda(z+1)}
		=\frac{z-l}{(n-z)\lambda}.
	\end{split}
\end{align}
Then \eqsref{eq:LHSF14}{eq:RHSF14} together confirm \eref{eq:invslopwithk}, and completes the proof of \eref{eq:slopwithk}.

\section{Proof of \lref{cor:zeta2Bound2} \main}\Label{app:ProofZeta2Asympt}
\begin{proof}[Proof of \lref{cor:zeta2Bound2} \main]
	When $0<\delta\leq B_{n,l}$, we have $z^*\geq n$ and $\overline{\epsilon}^{\rm c}_\lambda(l,n,\delta) =0$.
	So Eqs.~\eqref{eq:zeta2UB}-\eqref{eq:zeta2LB} \main hold obviously.
	In the following, we consider the case $B_{n,l}<\delta<1$.
	
	\noindent \textbf{Step 1:} Proof of \eref{eq:zeta2UB} \main.
	When $B_{n,l}<\delta<1$, we have
	\begin{align}
		\begin{split}
			\overline{\epsilon}^{\rm c}_\lambda(l,n,\delta)
			&\leq \frac{g_{z_*}(l,n,\lambda)}{h_{z_*}(l,n,\lambda)}
			= \frac{1}{1+\frac{z_*}{n-z_*+1}\cdot\frac{B_{z_*-1,l}}{B_{z_*,l}}}\\
			&\leq \frac{1}{1+\frac{z_*}{n-z_*+1}\cdot\frac{z_*-l}{\lambda z_*}}
			=\frac{\lambda(n-z_*+1)}{\lambda(n-z_*+1)+z_*-l },
		\end{split}
	\end{align}
	where the first inequality follows from \coref{cor:zeta2Bound1} \main and the fact that $\overline{\epsilon}^{\rm c}_\lambda(l,n,\delta) >0$ when $B_{n,l}<\delta<1$;
	and the last inequality follows from \lref{lem:Bzkz}.
	
	\noindent \textbf{Step 2:} Proof of \eref{eq:zeta2LB1} \main.
	When $B_{n,l}<\delta<1$, we have
	\begin{align}\Label{eq:zeta2LBproof}
		\overline{\epsilon}^{\rm c}_\lambda(l,n,\delta)
		\geq \frac{g_{z^*+1}(l,n,\lambda)}{h_{z^*+1}(l,n,\lambda)}
		= \frac{1}{1+\frac{z^*+1}{n-z^*}\cdot\frac{B_{z^*,l}}{B_{z^*+1,l}}}
		\geq \frac{1}{1+\frac{z^*+1}{n-z^*}\cdot\frac{1}{\lambda}}
		=\frac{\lambda(n-z^*)}{\lambda(n-z^*)+z^*+1},
	\end{align}
	where the first inequality follows from \coref{cor:zeta2Bound1} \main;
	and the last inequality follows from \lref{lem:Bzkz}.

	\noindent \textbf{Step 3:} Proof of \eref{eq:zeta2LB} \main.
	When $B_{n,l}<\delta\leq1/2$, we have $l/\nu\leq z^*\leq n$ according to \coref{cor:z*k/nu}, and
	\begin{align}
		\begin{split}
			\overline{\epsilon}^{\rm c}_\lambda(l,n,\delta)
			&\stackrel{(a)}{\geq} \frac{g_{z^*+1}(l,n,\lambda)}{h_{z^*+1}(l,n,\lambda)}
			= \frac{1}{1+\frac{z^*+1}{n-z^*}\cdot\frac{B_{z^*,l}}{B_{z^*+1,l}}}
			\stackrel{(b)}{\geq} \frac{1}{1+\frac{z^*+1}{n-z^*}\cdot\frac{z^*-l+1+\sqrt{\lambda l}}{(z^*+1)\lambda}}\\
			&=\frac{\lambda(n-z^*)}{\lambda(n-z^*)+z^*-l+1+\sqrt{\lambda l} }
			=1-\frac{z^*-l+1+\sqrt{\lambda l}}{\lambda n+(\nu z^*-l+1+\sqrt{\lambda l}) } \\
			&\stackrel{(c)}{\geq} 1-\frac{z^*-l+1+\sqrt{\lambda l}}{\lambda n }.
		\end{split}
	\end{align}
	where $(a)$ follows from \coref{cor:zeta2Bound1} \main;
	$(b)$ follows from \lref{lem:Bzkz} 
	and that $z^*\geq l/\nu$;
	and $(c)$ follows because $(\nu z^*-l+1+\sqrt{\lambda l}) \geq0$.
\end{proof}

\end{appendix}

\begin{acks}[Acknowledgments]
HZ and ZL are also affiliated to Institute for Nanoelectronic Devices and Quantum Computing, Fudan University and Center for Field Theory and Particle Physics, Fudan University.
MH is also affiliated to International Quantum Academy (SIQA) and Graduate School of Mathematics, Nagoya University.
\end{acks}

\begin{funding}
The work at Fudan is  supported by   the National Natural Science Foundation of China (Grants No.~11875110 and No.~92165109) and  Shanghai Municipal Science and Technology Major Project (Grant No.~2019SHZDZX01).
MH is supported in part by the National Natural Science Foundation of China (Grants No. 62171212 and No.~11875110) and
Guangdong Provincial Key Laboratory (Grant No. 2019B121203002).
\end{funding}

\bibliographystyle{imsart-number}
\bibliography{all_references}

\begin{thebibliography}{17}

\bibitem{Berry41}
\begin{barticle}[author]
\bauthor{\bsnm{Berry},~\bfnm{Andrew~C}\binits{A.~C.}}
(\byear{1941}).
\btitle{The accuracy of the Gaussian approximation to the sum of independent
  variates}.
\bjournal{Transactions of the american mathematical society}
\bvolume{49}
\bpages{122--136}.
\end{barticle}
\endbibitem

\bibitem{DembZ10}
\begin{bbook}[author]
\bauthor{\bsnm{Dembo},~\bfnm{Amir}\binits{A.}} \AND
  \bauthor{\bsnm{Zeitouni},~\bfnm{Ofer}\binits{O.}}
(\byear{2009}).
\btitle{Large deviations techniques and applications}
\bvolume{38}.
\bpublisher{Springer Science \& Business Media}.
\end{bbook}
\endbibitem

\bibitem{Gordon41}
\begin{barticle}[author]
\bauthor{\bsnm{Gordon},~\bfnm{Robert~D}\binits{R.~D.}}
(\byear{1941}).
\btitle{Values of Mills' ratio of area to bounding ordinate and of the normal
  probability integral for large values of the argument}.
\bjournal{The Annals of Mathematical Statistics}
\bvolume{12}
\bpages{364--366}.
\end{barticle}
\endbibitem

\bibitem{Edgeworth}
\begin{bincollection}[author]
\bauthor{\bsnm{Hall},~\bfnm{Peter}\binits{P.}}
(\byear{1992}).
\btitle{Principles of edgeworth expansion}.
In \bbooktitle{The bootstrap and edgeworth expansion}
\bpages{39--81}.
\bpublisher{Springer}.
\end{bincollection}
\endbibitem

\bibitem{HM}
\begin{barticle}[author]
\bauthor{\bsnm{Hayashi},~\bfnm{Masahito}\binits{M.}} \AND
  \bauthor{\bsnm{Morimae},~\bfnm{Tomoyuki}\binits{T.}}
(\byear{2015}).
\btitle{Verifiable measurement-only blind quantum computing with stabilizer
  testing}.
\bjournal{Physical review letters}
\bvolume{115}
\bpages{220502}.
\end{barticle}
\endbibitem

\bibitem{HT}
\begin{barticle}[author]
\bauthor{\bsnm{Hayashi},~\bfnm{Masahito}\binits{M.}} \AND
  \bauthor{\bsnm{Tsurumaru},~\bfnm{Toyohiro}\binits{T.}}
(\byear{2012}).
\btitle{Concise and tight security analysis of the Bennett--Brassard 1984
  protocol with finite key lengths}.
\bjournal{New Journal of Physics}
\bvolume{14}
\bpages{093014}.
\end{barticle}
\endbibitem

\bibitem{Kaas80}
\begin{barticle}[author]
\bauthor{\bsnm{Kaas},~\bfnm{Rob}\binits{R.}} \AND
  \bauthor{\bsnm{Buhrman},~\bfnm{Jan~M}\binits{J.~M.}}
(\byear{1980}).
\btitle{Mean, median and mode in binomial distributions}.
\bjournal{Statistica Neerlandica}
\bvolume{34}
\bpages{13--18}.
\end{barticle}
\endbibitem

\bibitem{KL}
\begin{barticle}[author]
\bauthor{\bsnm{Karagiannidis},~\bfnm{George~K}\binits{G.~K.}} \AND
  \bauthor{\bsnm{Lioumpas},~\bfnm{Athanasios~S}\binits{A.~S.}}
(\byear{2007}).
\btitle{An improved approximation for the Gaussian Q-function}.
\bjournal{IEEE Communications Letters}
\bvolume{11}
\bpages{644--646}.
\end{barticle}
\endbibitem

\bibitem{Lehmann}
\begin{bbook}[author]
\bauthor{\bsnm{Lehmann},~\bfnm{Erich~Leo}\binits{E.~L.}} \AND
  \bauthor{\bsnm{Romano},~\bfnm{Joseph~P}\binits{J.~P.}}
(\byear{2005}).
\btitle{Testing statistical hypotheses}
\bvolume{3}.
\bpublisher{Springer Texts in Statistics}.
\end{bbook}
\endbibitem

\bibitem{q-paper}
\begin{barticle}[author]
\bauthor{\bsnm{Li},~\bfnm{Zihao}\binits{Z.}},
  \bauthor{\bsnm{Zhu},~\bfnm{Huangjun}\binits{H.}} \AND
  \bauthor{\bsnm{Hayashi},~\bfnm{Masahito}\binits{M.}}
(\byear{2022}).
\btitle{Robust and efficient verification of measurement-based quantum
  computing}.
\bnote{In preparation}.
\end{barticle}
\endbibitem

\bibitem{Nowako21}
\begin{barticle}[author]
\bauthor{\bsnm{Nowakowski},~\bfnm{Szymon}\binits{S.}}
(\byear{2021}).
\btitle{Uniqueness of a Median of a Binomial Distribution with Rational
  Probability}.
\bjournal{Advances in Mathematics: Scientific Journal}
\bvolume{10}
\bpages{1951}.
\end{barticle}
\endbibitem

\bibitem{PP}
\begin{bbook}[author]
\bauthor{\bsnm{Papoulis},~\bfnm{Athanasios}\binits{A.}}
(\byear{1965}).
\btitle{Probability, Random Variables, and Stochastic Processes}.
\bpublisher{McGraw Hill}, \baddress{New York}.
\end{bbook}
\endbibitem

\bibitem{OS}
\begin{barticle}[author]
\bauthor{\bsnm{Shamir},~\bfnm{Ohad}\binits{O.}}
(\byear{2016}).
\btitle{Without-replacement sampling for stochastic gradient methods}.
\bjournal{Advances in neural information processing systems}
\bvolume{29}.
\end{barticle}
\endbibitem

\bibitem{TR}
\begin{barticle}[author]
\bauthor{\bsnm{Tanash},~\bfnm{Islam~M}\binits{I.~M.}} \AND
  \bauthor{\bsnm{Riihonen},~\bfnm{Taneli}\binits{T.}}
(\byear{2021}).
\btitle{Improved coefficients for the Karagiannidis--Lioumpas approximations
  and bounds to the Gaussian Q-function}.
\bjournal{IEEE Communications Letters}
\bvolume{25}
\bpages{1468--1471}.
\end{barticle}
\endbibitem

\bibitem{ZhuEVQPSshort19}
\begin{barticle}[author]
\bauthor{\bsnm{Zhu},~\bfnm{Huangjun}\binits{H.}} \AND
  \bauthor{\bsnm{Hayashi},~\bfnm{Masahito}\binits{M.}}
(\byear{2019}).
\btitle{Efficient verification of pure quantum states in the adversarial
  scenario}.
\bjournal{Physical review letters}
\bvolume{123}
\bpages{260504}.
\end{barticle}
\endbibitem

\bibitem{ZhuEVQPSlong19}
\begin{barticle}[author]
\bauthor{\bsnm{Zhu},~\bfnm{Huangjun}\binits{H.}} \AND
  \bauthor{\bsnm{Hayashi},~\bfnm{Masahito}\binits{M.}}
(\byear{2019}).
\btitle{General framework for verifying pure quantum states in the adversarial
  scenario}.
\bjournal{Physical Review A}
\bvolume{100}
\bpages{062335}.
\end{barticle}
\endbibitem

\bibitem{Binomial22}
\begin{barticle}[author]
\bauthor{\bsnm{Zhu},~\bfnm{Huangjun}\binits{H.}},
  \bauthor{\bsnm{Li},~\bfnm{Zihao}\binits{Z.}} \AND
  \bauthor{\bsnm{Hayashi},~\bfnm{Masahito}\binits{M.}}
(\byear{2022}).
\btitle{Nearly tight universal bounds for the binomial tail probabilities}.
\bnote{arXiv:2211.01688}.
\end{barticle}
\endbibitem

\end{thebibliography}

\end{document}